\documentclass[11pt, oneside]{article} 
\usepackage{amsmath,amssymb}  
\usepackage{mathrsfs,amsthm}
\usepackage{graphicx}
\usepackage{stmaryrd}
\usepackage[all]{xy}
\usepackage{amscd}
\usepackage{geometry}
\usepackage{empheq}
\makeatletter
\@addtoreset{equation}{section}

\makeatother

\title{Riemann-Hilbert correspondence for unit $F$-crystals on embeddable algebraic varieties}

\author{Sachio OHKAWA} 
\date{}

\begin{document}
\maketitle

\begin{abstract}
For a separated scheme $X$ of finite type over a perfect field $k$ of characteristic $p>0$ which admits an immersion into a proper smooth scheme over the truncated Witt ring $W_{n}$, we define the bounded derived category of locally finitely generated unit $F$-crystals with finite Tor-dimension on $X$ over $W_{n}$, independently of the choice of the immersion. Then we prove the anti-equivalence of this category with the bounded derived category of constructible \'etale sheaves of ${\mathbb Z}/{p^{n}{\mathbb Z}}$-modules with finite Tor-dimension. We also discuss the relationship of $t$-structures on these derived categories when $n=1$. 
\end{abstract}

\theoremstyle{plain}
\newtheorem{theo}{Theorem}[section]
\newtheorem{defi}[theo]{Definition}
\newtheorem{lemm}[theo]{Lemma}
\newtheorem{prop}[theo]{Proposition}
\newtheorem{coro}[theo]{Corollary}
\newtheorem{Prod}[theo]{Proposition-Definition}

\theoremstyle{definition}
\newtheorem{rema}[theo]{Remark}
\newtheorem{ex}[theo]{Example}

\renewcommand*\proofname{\upshape{\bfseries{Proof}}}

\section{Introduction}

For a complex manifold $X$,
Kashiwara \cite{Kas1}  and Mebkhout \cite{Me1} independently established an anti-equivalence, which is called the Riemann-Hilbert correspondence, between the triangulated category $D^{b}_{\rm rh}({\mathcal D}_{X})$ of ${\mathcal D}_{X}$-modules with regular holonomic cohomologies and that $D_{\rm c}^{b}(X, {\mathbb C})$ of sheaves of ${\mathbb C}$-vector spaces on $X$  with constructible cohomologies.
There is a significant property from the point of view of relative cohomology theories that this anti-equivalence respects Grothendieck's six operations 
$f^{!}$, $f_{!}$, $f^{*}$, $f_{*}$, ${\mathbb R}\underline{\rm Hom}$ and $\otimes^{\mathbb L}$ defined on $D^{b}_{\rm rh}({\mathcal D}_{X})$ and $D_{\rm c}^{b}(X, {\mathbb C})$.

In \cite{EK}, Emerton and Kisin studied a positive characteristic analogue of the Riemann-Hilbert correspondence.
Let $k$ be a perfect field of characteristic $p>0$.
We denote by $W_{n}:=W_{n}(k)$ the ring of Witt vectors of length $n$. 
For a smooth scheme $X$ over $W_{n}$,
Emerton and Kisin defined the sheaf ${\mathcal D}_{F, X}$ of ${\mathcal O}_{X}$-algebras by adjoining to ${\mathcal O}_{X}$ the differential operators of all orders on $X$ and a \lq\lq local lift of Frobenius''.
By using ${\mathcal D}_{F, X}$, they introduced the triangulated category $D^{b}_{\rm lfgu}({\mathcal D}_{F, X})$ of ${\mathcal D}_{X}$-modules with Frobenius structures with  locally finitely generated unit cohomologies  and proved the anti-equivalence
\begin{equation*}
D^{b}_{\rm lfgu}({\mathcal D}_{F, X})^{\circ}\xrightarrow{\cong} D^{b}_{\rm ctf}(X_{\rm \acute{e}t}, {\mathbb Z}/{p^{n}{\mathbb Z}})
\end{equation*}
 between the subcategory $D^{b}_{\rm lfgu}({\mathcal D}_{F, X})^{\circ}$ of $D^{b}_{\rm lfgu}({\mathcal D}_{F, X})$ consisting of complexes of finite Tor dimension over ${\mathcal O}_{X}$ and
the triangulated category $D^{b}_{\rm ctf}(X_{\rm \acute{e}t}, {\mathbb Z}/{p^{n}{\mathbb Z}})$ of  \'etale sheaves of ${\mathbb Z}/{p^{n}{\mathbb Z}}$-modules with constructible cohomologies  and of finite Tor dimension over ${\mathbb Z}/{p^{n}{\mathbb Z}}$, which they call the Riemann-Hilbert correspondence for unit $F$-crystals.
They also introduced three of  Grothendieck's six operations, which are the direct image $f_{+}$, the inverse image $f^{!}$ and the tensor product $\otimes^{\mathbb L}$ on $D^{b}_{\rm lfgu}({\mathcal D}_{F, X})$, and proved that their Riemann-Hilbert correspondence 
exchanges these to $f_{!}$, $f^{-1}$ and $\otimes^{\mathbb L}$ on $D^{b}_{\rm ctf}(X_{\rm \acute{e}t}, {\mathbb Z}/{p^{n}{\mathbb Z}})$.
\bigskip

Emerton and Kisin established the Riemann-Hilbert correspondence for unit $F$-crystals only for smooth schemes $X$ over $W_{n}$. 
Since the triangulated category $D^{b}_{\rm ctf}(X_{\rm \acute{e}t}, {\mathbb Z}/{p^{n}{\mathbb Z}})$ depends only on the mod $p$ reduction of $X$,
it is natural to expect that there exists a definition of the triangulated category $D^{b}_{\rm lfgu}({\mathcal D}_{F, X})^{\circ}$ and the Riemann-Hilbert correspondence depending only on the mod $p$ reduction of $X$.
Also, there should be the Riemann-Hilbert correspondence for algebraic varieties over $k$ which are not smoothly liftable to $W_{n}$.
The purpose of this article is to generalize  the Emerton-Kisin theory to the case of $W_{n}$-embeddable algebraic varieties over $k$.
Here 
we say a separated  $k$-scheme $X$ of finite type is $W_{n}$-embeddable if there exists a proper smooth $W_{n}$-scheme $P$ and an immersion $X\hookrightarrow P$ such that the diagram 
\begin{equation}\label{d3}
\vcenter{
\xymatrix{
{X} \ar[d] \ar@{^{(}-_>}[r] & P \ar[d] \\
{\rm Spec} k \ar[r] & {\rm Spec} W_{n} }
}
\end{equation}
is commutative.
A quasi projective variety over $k$ is a typical example of $W_{n}$-embeddable variety and thus 
$W_{n}$-embeddable varieties form a sufficiently wide class of algebraic varieties in some sense.
\bigskip

The first problem is to define a reasonable $\mathcal D$-module category for $W_{n}$-embeddable algebraic varieties over $k$.
Our construction is based on Kashiwara's theorem which roughly asserts that, for any closed immersion $X\hookrightarrow P$ of smooth algebraic varieties, the category of $\mathcal D$-modules on $P$ supported on $X$ is naturally equivalent to the category of $\mathcal D$-modules on $X$.
Using the characteristic $p>0$ analogue
of Kashiwara's theorem due to Emerton-Kisin \cite[Proposition 15.5.3]{EK}, we show that, when we are given the diagram (\ref{d3}), 
the full triangulated subcategory of $D^{b}_{\rm lfgu}({\mathcal D}_{F, P})^{\circ}$ consisting of complexes supported on $X$ does not depend on the choice of immersion $X\hookrightarrow P$ (Corollary \ref{c3}).
We denote this full subcategory by $D^{b}_{\rm lfgu}(X/W_{n})^{\circ}$.
Then we show the Riemann-Hilbert correspondence
\begin{equation*}
D^{b}_{\rm lfgu}(X/W_{n})^{\circ}\xrightarrow{\cong}D^{b}_{\rm ctf}(X_{\rm \acute{e}t}, {\mathbb Z}/{p^{n}{\mathbb Z}})
\end{equation*}
for any $W_{n}$-embeddable $k$-scheme $X$ (Theorem \ref{t5}).
As in the case of \cite{EK}, we can naturally introduce three of  Grothendieck's six operations, that is, direct and inverse images and tensor products.
We then prove that the Riemann-Hilbert correspondence respects these operations (Theorem \ref{t6}).
A striking consequence of the Riemann-Hilbert correspondence over complex numbers is that one can introduce an exotic $t$-structure on the topological side called the perverse $t$-structure, which corresponds to the standard $t$-structure on the $\mathcal D$-module side.
For an algebraic variety $X$ over $k$, Gabber introduced in \cite{Ga}  the perverse $t$-structure on $D^{b}_{\rm c}(X_{\rm \acute{e}t}, {\mathbb Z}/{p{\mathbb Z}})$, which we call Gabber's perverse $t$-structure.
In the case when $X$ is smooth over $k$, Emerton and Kisin showed that 
the standard $t$-structure on the $\mathcal D$-module side corresponds to Gabber's perverse $t$-structure.
In this paper, we generalize it to the case of $k$-embeddable $k$-schemes.
In the complex situation, conversely, a $t$-structure on the $\mathcal D$-module side corresponding to the standard $t$-structure on the topological side is explicitly described by Kashiwara in \cite{Kas2}. 
In this paper, we construct the analogue of Kashiwara's $t$-structure on $D^{b}_{\rm lfgu}(X/k)$ and discuss the relationship of it and the standard $t$-structure on $D^{b}_{\rm c}(X_{\rm \acute{e}t}, {\mathbb Z}/{p{\mathbb Z}})$.

\bigskip
The content of each section is as follows:
In the second section, we recall several notions, terminologies and cohomological operations on ${\mathcal D}_{F, P}$-modules from \cite{EK} which we often use in this paper.
We also recall the statement of the Riemann-Hilbert correspondence for unit $F$-crystals of Emerton-Kisin (Theorem \ref{t4}).
In the third section, we define the local cohomology functor ${\mathbb R}{\Gamma}_{Z}$ for ${\mathcal D}_{F, P}$-modules and prove compatibilities with ${\mathbb R}{\Gamma}_{Z}$ and other operations for ${\mathcal D}_{F, P}$-modules, which are essential tools to define and study the triangulated category $D^{b}_{\rm lfgu}(X/W_{n})^{\circ}$ for any $W_{n}$-embeddable $k$-scheme $X$.
In subsection \ref{s41}, we introduce the category $D^{b}_{\rm lfgu}(X/W_{n})^{\circ}$ for any $W_{n}$-embeddable $k$-scheme $X$ and
in subsection \ref{s42}, we construct three of Grothendieck's six operations on $D^{b}_{\rm lfgu}(X/W_{n})^{\circ}$.
Our arguments in these subsections are heavily inspired by that of Caro in \cite{Ca}.
In subsection \ref{s43}, we prove the Riemann-Hilbert correspondence for unit $F$-crystals on $W_{n}$-embeddable $k$-schemes, which is our main result.
In the fifth section, we discuss several properties on $D^{b}_{\rm lfgu}(X/k)$ (in the case $n=1$) related to $t$-structures.
In subsection \ref{s51}, we introduce the standard $t$-structure on $D^{b}_{\rm lfgu}(X/k)$ depending on the choice of the immersion $X\hookrightarrow P$. 
We prove that the standard $t$-structure corresponds to Gabber's perverse $t$-structure via the Riemann-Hilbert
correspondence.
As a consequence, we know that the definition of the standard $t$-structure is independent of the choice of $X\hookrightarrow P$ (Theorem \ref{t11}).
In subsection \ref{s52}, we define the abelian category $\mu_{{\rm lfgu}, X}$ as the heart of the standard $t$-structure on $D^{b}_{\rm lfgu}(X/k)$,
and prove that the natural functor $D^{b}(\mu_{{\rm lfgu}, X})\to D^{b}_{\rm lfgu}(X/k)$ is an equivalence of triangulated categories (Theorem \ref{b1}), which can be regarded as an analogue of Beilinson's theorem.
In subsection \ref{s53}, depending on the choice of the immersion $X\hookrightarrow P$, we introduce the constructible $t$-structure on $D^{b}_{\rm lfgu}(X/k)$ by following the arguments in \cite{Kas2} and prove that it corresponds to the standard $t$-structure on the \'etale side via the Riemann-Hilbert correspondence.
As a consequence, we see that the constructible $t$-structure does not depend on the choice of $X\hookrightarrow P$ (Corollary \ref{c5}).

\bigskip
After writing the first version of this article, the author learned that, in \cite{S} Schedlmeier also studied on closely related subjects independently  in terms of the notion of Cartier crystals.
In particular, he proved that, for a separated $F$-finite $k$-scheme $X$ which admits a closed immersion into a smooth (but not necessarily proper) $k$-scheme, there exists an equivalence of triangulated categories between the bounded derived category of Cartier crystals on $X$ and $D^{b}_{\rm c}(X_{\rm \acute{e}t}, {\mathbb Z}/{p{\mathbb Z}})$.

\section*{Acknowledgments} 
This article is the doctor thesis of the author in the university of Tokyo.
The author would like to express his profound gratitude to his supervisor Atsushi Shiho for valuable suggestions, discussions and the careful reading of this paper.
He would like to thank Kohei Yahiro for useful discussions and  Manuel Blickle for informing me a result of Schedlmeier in \cite{S}.
He would also like to thank the referee for reading the paper in detail, pointing out mistakes, and giving valuable comments.
This work was supported by the Program for Leading Graduate 
Schools, MEXT, Japan, the Grant-in-Aid for Scientific Research (KAKENHI No. 26-9259),  and the Grant-in-Aid for JSPS fellows.

\subsection*{Conventions}\label{con}

Throughout this paper, we fix a prime number $p$ and a perfect base field $k$ of characteristic $p$.
We denote by $W$ the ring of Witt vectors associated to $k$ and by $W_{n}$ the quotient ring $W/(p)^{n}$ for any natural number $n$.
For a scheme $X$, we denote the structure sheaf of $X$ by ${\mathcal O}_{X}$.
For a smooth $W_n$-scheme $X$, the dimension of $X$ is a continuous integer valued function on $X$ defined by
$$d_{X}: x\in X\mapsto \textit{dimension of the component of $X$ containing }x.$$
For a morphism $f: X\to Y$ of smooth $W_n$-scheme, we denote a function $d_{X}-d_{Y}\circ f$ by $d_{X/Y}$ and a function $-d_{X/Y}$ by $d_{Y/X}$.
For an abelian category $\mathcal C$,
we denote by $D({\mathcal C})$ the derived category of $\mathcal C$.
For a scheme $X$ and an ${\mathcal O}_{X}$-algebra ${\mathcal A}$,
we denote by $D({\mathcal A})$ the derived category of left ${\mathcal A}$-modules and 
by $D_{\rm qc}({\mathcal A})$ the full triangulated subcategory of $D({\mathcal A})$ consisting of complexes whose cohomology sheaves are quasi-coherent as ${\mathcal O}_{X}$-modules.
For ${\mathcal A}$-modules ${\mathcal F}$ and ${\mathcal G}$,
we denote by $\underline{\rm Hom}_{{\mathcal A}}({\mathcal F}, {\mathcal G})$ the sheaf of ${\mathcal A}$-linear homomorphism from ${\mathcal F}$ to ${\mathcal G}$. 
We denote a complex by a single letter such as $\mathcal M$
and by ${\mathcal M}^{n}$ the $n$-th term of ${\mathcal M}$.
For an object $\mathcal M$ in $D({\mathcal A})$, we denote
by ${\rm H}^{i}({\mathcal M})$ the $i$-th cohomology of $\mathcal M$ and by ${\rm Supp} {\mathcal M}$
the support of $\mathcal M$, which is defined as the closure of $\bigcup_{i} {\rm Supp}{\rm H}^{i}({\mathcal M})$. 

\section{Preliminaries}
In this section we recall the notion of locally finitely generated unit ${\mathcal D}_{F, X}$-modules
introduced in \cite{EK} and the Riemann-Hilbert correspondence for unit $F$-crystals in \cite{EK}.

\subsection{Locally finitely generated unit ${\mathcal D}_{F, X}$-modules}

For a smooth $W_{n}$-scheme $X$,
we denote by ${\mathcal D}_{X}$ the sheaf of differential operators of $X$ over $W_{n}$ defined in \cite[\S16]{EGA4}.
For a morphism of smooth $W_{n}$-schemes $f: X\to Y$ and a left ${\mathcal D}_{Y}$-module $\mathcal M$, 
$f^{*}{\mathcal M}:={\mathcal O}_{X}\otimes_{f^{-1}{\mathcal O}_{Y}} f^{-1}{\mathcal M}$ has a natural structure of left ${\mathcal D}_{X}$-modules. 
When there exists a lifting $F: X\to X$ of the absolute Frobenius on $X\otimes_{W_n} k$, the left ${\mathcal D}_{X}$-module structure on $F^{*}{\mathcal M}$ is known to be independent of the choice of the lifting $F$ up to canonical isomorphism by \cite[Proposition 13.2.1]{EK}.
Since the lifting $F$ above always exists Zariski locally on $X$, we obtain a functor
\begin{equation*}
F^{*}: (\textit{left ${\mathcal D}_{X}$-module})\to (\textit{left ${\mathcal D}_{X}$-module})
\end{equation*}
by glueing for any smooth $W_{n}$-scheme $X$.
We set 
\begin{equation*}
{\mathcal D}_{F, X}:=\bigoplus_{r\geq 0} (F^{*})^{r} {\mathcal D}_{X}.
\end{equation*}
Then ${\mathcal D}_{F, X}$ naturally forms a sheaf of associative $W_{n}$-algebras such that the natural embedding 
${\mathcal D}_{X}\to {\mathcal D}_{F, X}$ is a $W_{n}$-algebra homomorphism by \cite[Corollary 13.3.5]{EK}.
It is proved in \cite[Proposition 13.3.7]{EK} that giving a left ${\mathcal D}_{F, X}$-module ${\mathcal M}$ is equivalent to giving a ${\mathcal D}_{X}$-module ${\mathcal M}$ together with a morphism $\psi_{\mathcal M}: F^{*}{\mathcal M}\to {\mathcal M}$ of left ${\mathcal D}_{X}$-modules, which
we call the structural morphism of $\mathcal M$.

Next let us recall the notion of locally finitely generated unit ${\mathcal D}_{F, X}$-modules.
We say that a left ${\mathcal D}_{F, X}$-module ${\mathcal M}$ is unit if it is quasi-coherent as an ${\mathcal O}_{X}$-module 
and the structural morphism $\psi_{\mathcal M}: F^{*}{\mathcal M}\to {\mathcal M}$ is an isomorphism.
We say that a ${\mathcal D}_{F, X}$-module ${\mathcal M}$ is locally finitely generated unit if it is unit and Zariski locally on $X$, there exists a coherent ${\mathcal O}_{X}$-submodule $M\subset {\mathcal M}$ such that the natural morphism ${\mathcal D}_{F, X}\otimes_{{\mathcal O}_{X}}{M}\to {\mathcal M}$ is surjective.
Then the locally finitely generated unit left ${\mathcal D}_{F, X}$-modules form a thick subcategory of the category of quasi-coherent left ${\mathcal D}_{F, X}$-modules \cite[Proposition 15.3.4]{EK}.
We say that a locally finitely generated unit ${\mathcal D}_{F, X}$-module ${\mathcal M}$ is an $F$-crystal if ${\mathcal M}$ is locally free of finite rank as an ${\mathcal O}_{X}$-module.

Finally we introduce some notations on triangulated categories.
We denote by $D({\mathcal D}_{F, X})$ the derived category of the abelian category of left ${\mathcal D}_{F, X}$-modules and by
$D_{\rm qc}({\mathcal D}_{F, X})$ (resp. $D_{\rm lfgu}({\mathcal D}_{F, X})$) the full triangulated subcategory of $D({\mathcal D}_{F, X})$
consisting of those complexes whose cohomology sheaves are quasi-coherent as ${\mathcal O}_{X}$-modules (resp. are locally finitely generated unit left ${\mathcal D}_{F, X}$-modules).
If $\bullet$ is one of $\emptyset$, $-$, $+$, $b$, we denote by  $D^{\bullet}({\mathcal D}_{F, X})$
the full triangulated subcategories of $D({\mathcal D}_{F, X})$ consisting of those complexes satisfying the appropriate boundedness condition.
We use the notations  $D^{\bullet}_{\rm qc}({\mathcal D}_{F, X})$ and $D^{\bullet}_{\rm lfgu}({\mathcal D}_{F, X})$ in a similar manner.
We denote by $D^{b}_{\rm lfgu}({\mathcal D}_{F, X})^{\circ}$ the full triangulated subcategory of $D^{b}_{\rm lfgu}({\mathcal D}_{F, X})$ consisting of those complexes which are of finite Tor dimension as ${\mathcal O}_{X}$-modules.

\subsection{Cohomological operations for left ${\mathcal D}_{F, X}$-modules}

For a morphism $f: X\to Y$ of smooth $W_{n}$-schemes,
 $f^{*}{\mathcal D}_{F, Y}:={\mathcal O}_{X}\otimes_{f^{-1}{\mathcal O}_{Y}}f^{-1}{\mathcal D}_{F, Y}$ has a natural structure of left ${\mathcal D}_{F, X}$-module by \cite[Corollary 14.2.2]{EK}.
It also forms a right $f^{-1}{\mathcal D}_{F, Y}$-module via the right multiplication on $f^{-1}{\mathcal D}_{F, Y}$.
So $f^{*}{\mathcal D}_{F, Y}$ has a structure of $\left({\mathcal D}_{F, X}, f^{-1}{\mathcal D}_{F, Y}\right)$-bimodule, which we denote by ${\mathcal D}_{F, X\to Y}$.
For a ${\mathcal D}_{F, Y}$-module $\mathcal M$, we define  a left ${\mathcal D}_{F, X}$-module $f^{*}{\mathcal M}$ by ${\mathcal D}_{F, X\to Y}\otimes_{f^{-1}{\mathcal D}_{F, Y}}f^{-1} {\mathcal M}$.
Note that $f^{*}{\mathcal M}\cong {\mathcal O}_{X}\otimes_{f^{-1}{\mathcal O}_{Y}}f^{-1}{\mathcal M}$ as an ${\mathcal O}_{X}$-module.
We then define a functor 
\begin{equation*}
{\mathbb L}f^{*}: D^{-}({\mathcal D}_{F, Y})\to D^{-}({\mathcal D}_{F, X})
\end{equation*}
to be the left derived functor of $f^{*}$.
One has ${\mathbb L}f^{*}{\mathcal M}\cong {\mathcal O}_{X}\otimes^{\mathbb L}_{f^{-1}{\mathcal O}_{Y}}{\mathcal M}$ as a complex of ${\mathcal O}_{X}$-modules.
We also define a functor
\begin{equation*}
f^{!}: D^{-}({\mathcal D}_{F, Y})\to D^{-}({\mathcal D}_{F, X})
\end{equation*}
by $f^{!}{\mathcal M}:={\mathbb L}f^{*}{\mathcal M}[d_{X/Y}]$.
For the definition of the shift functor $(-)[d_{X/Y}]$ by the function $d_{X/Y}$, we refer the reader to \cite[\S0]{EK}.
The second inverse image functor is appropriate to formulate the compatibility with the Riemann-Hilbert correspondence (see Theorem \ref{t4} (2) bellow).
Let $\star$ be one of $\rm qc$ or $\rm lfgu$ and $\ast$ one of $\circ$ or $\emptyset$.
Then the functor $f^{!}$ restricts to a functor
\begin{equation*}
f^{!}: D^{b}_{\star}({\mathcal D}_{F, Y})^{\ast}\to D^{b}_{\star}({\mathcal D}_{F, X})^{\ast}
\end{equation*} 
by \cite[Proposition 14.2.6 and Proposition 15.5.1]{EK}.

Next let us define the direct image functor $f_{+}$ for ${\mathcal D}_{F, X}$-modules
for a morphism $f: X\to Y$ of smooth $W_{n}$-schemes.
First of all, we recall the definition of the direct image functor
\begin{equation*}
f^{\mathcal B}_{+}: D^{-}({\mathcal D}_{X}) \to D^{-}({\mathcal D}_{Y})
\end{equation*} 
for ${{\mathcal D}_{X}}$-modules.
For a smooth $W_{n}$-scheme $Y$, we denote by $\omega_{Y}$ the canonical bundle of $Y$ over $W_{n}$.
Then ${\mathcal D}_{Y}\otimes_{{\mathcal O}_{Y}}{\omega_{Y}^{-1}}$ has two natural left ${\mathcal D}_{Y}$-module structures.
The first one is the tensor product of left ${\mathcal D}_{Y}$-modules ${\mathcal D}_{Y}$ and ${\omega_{Y}^{-1}}$ (cf. \cite[1.2.7.(b)]{B2}).
On the other hand, using the right ${\mathcal D}_{Y}$-module structure on ${\mathcal D}_{Y}$ defined by the multiplication of ${\mathcal D}_{Y}$ on the right, one has the second left ${\mathcal D}_{Y}$-module structure on ${\mathcal D}_{Y}\otimes_{{\mathcal O}_{Y}}{\omega_{Y}^{-1}}$ by \cite[1.2.7.(b)]{B2}.
So ${\mathcal D}_{Y}\otimes_{{\mathcal O}_{Y}}{\omega_{Y}^{-1}}$ naturally forms a left $({\mathcal D}_{Y}, {\mathcal D}_{Y})$-bimodule.
For a morphism $f: X\to Y$ of smooth $W_{n}$-schemes,
by pulling  ${\mathcal D}_{Y}\otimes_{{\mathcal O}_{Y}}{\omega_{Y}^{-1}}$ back with respect to the second ${\mathcal D}_{Y}$-module structure, one has a left $(f^{-1}{\mathcal D}_{Y}, {\mathcal D}_{X})$-module $f^{*}_{\rm d}\left({\mathcal D}_{Y}\otimes_{{\mathcal O}_{Y}}{\omega_{Y}}^{-1}\right)$.
Here, in order to avoid confusion we use the notation $f^{*}_{\rm d}$ instead of $f^{*}$.
By tensoring ${\omega}_{X}$ on the right, one obtains an $(f^{-1}{\mathcal D}_{Y}, {\mathcal D}_{X})$-bimodule
\begin{equation*}
{\mathcal D}_{Y\leftarrow X}:=f_{\rm d}^{*}\left({\mathcal D}_{Y}\otimes_{{\mathcal O}_{Y}}{\omega}_{Y}^{-1}\right)\otimes_{{\mathcal O}_{X}}{\omega}_{X}.
\end{equation*}
On the other hand, one has a left $({\mathcal D}_{X}, f^{-1}{\mathcal D}_{Y})$-module $f^{*}_{\rm g}\left({\mathcal D}_{Y}\otimes_{{\mathcal O}_{Y}}{\omega_{Y}^{-1}}\right)$
by pulling back ${\mathcal D}_{Y}\otimes_{{\mathcal O}_{Y}}{\omega_{Y}^{-1}}$ with respect to the first ${\mathcal D}_{Y}$-module structure.
By tensoring $\omega_{X}$ on the left, we obtain 
 an $(f^{-1}{\mathcal D}_{Y}, {\mathcal D}_{X})$-bimodule
\begin{equation*}
{{\mathcal D}_{Y\leftarrow X}}':={\omega}_{X}\otimes_{{\mathcal O}_{X}}f_{\rm g}^{*}\left({\mathcal D}_{Y}\otimes_{{\mathcal O}_{Y}}{\omega}_{Y}^{-1}\right).
\end{equation*}
Then there exists the natural isomorphism of $(f^{-1}{\mathcal D}_{Y}, {\mathcal D}_{X})$-bimodules
\begin{equation*}
{{\mathcal D}_{Y\leftarrow X}}\xrightarrow{\cong}{{\mathcal D}_{Y\leftarrow X}}'.
\end{equation*}
For more details see \cite[3.4.1]{B2}.
We define a functor $f^{\mathcal B}_{+}: D^{-}({\mathcal D}_{X}) \to D^{-}({\mathcal D}_{Y})$ by
\begin{equation*}
f^{\mathcal B}_{+}{\mathcal M}:={\mathbb R}f_{*}\left({\mathcal D}_{Y\leftarrow X}\otimes_{{\mathcal D}_{X}}^{\mathbb L}{\mathcal M}\right).
\end{equation*}

Let us go back to the situation of ${\mathcal D}_{F, X}$-modules.
We define ${\mathcal D}_{F, Y\leftarrow X}$ by
$${\mathcal D}_{F, Y\leftarrow X}:={\omega}_{X}\otimes_{{\mathcal O}_{X}}f^{*}\left({\mathcal D}_{F, Y}\otimes_{{\mathcal O}_{Y}}{\omega}_{Y}^{-1}\right).$$
Then ${\mathcal D}_{F, Y\leftarrow X}$ has a natural $(f^{-1}{\mathcal D}_{F, Y}, {\mathcal D}_{F, X})$-bimodule structure (see \cite[\S14.3]{EK}).
We remark that ${\mathcal D}_{F, Y\leftarrow X}$ has finite Tor dimension as a right ${\mathcal D}_{F, X}$-module by \cite[Proposition 14.3.5]{EK} and, since $Y$ is a noetherian topological space, ${\mathbb R}f_{*}$ has finite cohomological amplitude.
We define a functor
\begin{equation*}
f_{+}: D^{-}({\mathcal D}_{F, X}) \to D^{-}({\mathcal D}_{F, Y})
\end{equation*}
by 
\begin{equation*}
f_{+}{\mathcal M}:={\mathbb R}f_{*}\left({\mathcal D}_{F, Y\leftarrow X}\otimes_{{\mathcal D}_{F, X}}^{\mathbb L}{\mathcal M}\right).
\end{equation*}
Let $\star$ be one of $\rm qc$ or $\rm lfgu$ and $\ast$ one of $\circ$ or $\emptyset$.
The functor $f_{+}$ restricts to a functor
\begin{equation*}
f_{+}: D^{b}_{\star}({\mathcal D}_{F, X})^{\ast}\to D^{b}_{\star}({\mathcal D}_{F, Y})^{\ast}
\end{equation*} 
by \cite[Proposition 14.3.9 and Proposition 15.5.1]{EK}.

\begin{rema}\label{r2}
Let $f: X\to Y$ be a morphism of smooth $W_{n}$-schemes.
The natural inclusion of $({\mathcal D}_{Y}, {\mathcal D}_{Y})$-bimodules ${\mathcal D}_{Y}\to {\mathcal D}_{F, Y}$ induces an
inclusion of $(f^{-1}{\mathcal D}_{Y}, {\mathcal D}_{X})$-bimodules $\iota: {\mathcal D}_{Y\leftarrow X}\to {\mathcal D}_{F, Y\leftarrow X}$.
We then obtain morphisms in the derived category of $(f^{-1}{\mathcal D}_{Y}, {\mathcal D}_{F, X})$-bimodules
\begin{equation*}
{\mathcal D}_{Y\leftarrow X}\otimes^{\mathbb L}_{{\mathcal D}_{X}}{\mathcal D}_{F, X}\to {\mathcal D}_{Y\leftarrow X}\otimes_{{\mathcal D}_{X}}{\mathcal D}_{F, X}\xrightarrow{D_{1}\otimes D_{2}\mapsto \iota(D_{1})D_{2}} {\mathcal D}_{F, Y\leftarrow X}.
\end{equation*}
For an object $\mathcal M$ in $D^{-}({\mathcal D}_{F, X})$, applying the functor ${\mathbb R}f_{*}(-\otimes^{\mathbb L}_{{\mathcal D}_{F, X}}{\mathcal M})$ to the composite of the above morphisms,
we obtain a ${\mathcal D}_{Y}$-linear morphism 
\begin{equation*}
f^{\mathcal B}_{+}{\mathcal M}\to f_{+}{\mathcal M}.
\end{equation*}
It is proved in \cite[\S14.3.10]{EK} that the morphism $f^{\mathcal B}_{+}{\mathcal M}\to f_{+}{\mathcal M}$ is an isomorphism.
\end{rema}

Let $X$ be a smooth $W_{n}$-scheme.
Let $\mathcal M$ and ${\mathcal N}$ be ${\mathcal D}_{F, X}$-modules with structural morphisms $\psi_{\mathcal M}$ and $\psi_{\mathcal N}$.
Then ${\mathcal M}\otimes_{{\mathcal O}_{X}}{\mathcal N}$ has a natural structure of left ${\mathcal D}_{X}$-modules.
We define the structural morphism on ${\mathcal M}\otimes_{{\mathcal O}_{X}}{\mathcal N}$ to be the composite of ${\mathcal D}_{X}$-linear morphisms
\begin{equation*}
F^{*}\left({\mathcal M}\otimes_{{\mathcal O}_{X}}{\mathcal N}\right){\cong} F^{*}{\mathcal M}\otimes_{{\mathcal O}_{X}}F^{*}{\mathcal N}\xrightarrow{\psi_{\mathcal M}\otimes\psi_{\mathcal N}} {\mathcal M}\otimes_{{\mathcal O}_{X}}{\mathcal N}.
\end{equation*}
Here the first isomorphism follows from \cite[2.3.1]{B2} and its proof.
We thus obtain the ${\mathcal D}_{F, X}$-module structure on ${\mathcal M}\otimes_{{\mathcal O}_{X}}{\mathcal N}$ and define a bi-functor
\begin{equation*}
D^{-}({\mathcal D}_{F, X})\times D^{-}({\mathcal D}_{F, X})\to  D^{-}({\mathcal D}_{F, X})
\end{equation*}
by $({\mathcal M}, {\mathcal N})\mapsto {\mathcal M}\otimes^{\mathbb L}_{{\mathcal O}_{X}}{\mathcal N}$.
This functor restricts to a bi-functor 
\begin{equation*}
D^{b}_{\rm lfgu}({\mathcal D}_{F, X})\times D^{b}_{\rm lfgu}({\mathcal D}_{F, X})^{\circ}\to  D^{b}_{\rm lfgu}({\mathcal D}_{F, X})
\end{equation*}
by \cite[Proposition 15.5.1]{EK}.

\begin{prop}\label{p5}
Let $f: X\to Y$ be a morphism of smooth $W_{n}$-schemes.
If ${\mathcal M}$ and ${\mathcal N}$ are objects in $D^{-}({\mathcal D}_{F, Y})$, then there are natural isomorphisms
\begin{equation*}
{\mathbb L}f^{*}{\mathcal M}\otimes_{{\mathcal O}_{X}}^{\mathbb L}{\mathbb L}f^{*}{\mathcal N}\xrightarrow{\cong} 
{\mathbb L}f^{*}\left({\mathcal M}\otimes_{{\mathcal O}_{Y}}^{\mathbb L}{\mathcal N}\right)
\end{equation*}
and
\begin{equation*}
f^{!}{\mathcal M}\otimes_{{\mathcal O}_{X}}^{\mathbb L}f^{!}{\mathcal N}[d_{Y/X}]\xrightarrow{\cong}
f^{!}\left({\mathcal M}\otimes_{{\mathcal O}_{Y}}^{\mathbb L}{\mathcal N}\right).
\end{equation*}
\end{prop}
\begin{proof}
The second isomorphism follows from the first one.
Let $\mathcal P\to M$ (resp. $\mathcal Q\to N$) be a resolution of $\mathcal M$ (resp. $\mathcal N$) by flat ${\mathcal D}_{F, Y}$-modules.
Note that $\mathcal P$ and $\mathcal Q$ are complexes of flat ${\mathcal O}_{Y}$-modules.
So ${\mathcal P}\otimes_{{\mathcal O}_{Y}}{\mathcal Q}\to {\mathcal M}\otimes_{{\mathcal O}_{Y}}^{\mathbb L}{\mathcal N}$ gives a resolution of ${\mathcal M}\otimes_{{\mathcal O}_{Y}}^{\mathbb L}{\mathcal N}$ by flat ${\mathcal D}_{F, Y}$-modules (cf. \cite[Lemma 4.1]{Ha}) and $f^{*}{\mathcal P}$ is a complex of flat ${\mathcal O}_{X}$-modules.
By the universal mapping property of the tensor product,
one has a natural ${\mathcal D}_{F, X}$-linear morphism $f^{*}{\mathcal P}\otimes_{{\mathcal O}_{X}} f^{*}{\mathcal Q}\to f^{*}\left({\mathcal P}\otimes_{{\mathcal O}_{Y}} {\mathcal Q}\right)$.
Evidently it is an isomorphism as a morphism in $D({\mathcal O}_{X})$ and hence 
it is the desired isomorphism.

\end{proof}

\subsection{Riemann-Hilbert correspondence for unit $F$-crystals}
Let $X$ be a smooth $W_{n}$-scheme.
We denote by $D^{b}(X_{\rm \acute{e}t}, {\mathbb Z}/{p^{n}{\mathbb Z}})$  the bounded derived category of complexes of 
${\mathbb Z}/{p^{n}{\mathbb Z}}$-modules on the \'etale site $X_{\rm \acute{e}t}$.
We let $D^{b}_{\rm ctf}(X_{\rm \acute{e}t}, {\mathbb Z}/{p^{n}{\mathbb Z}})$ denote the full triangulated subcategory of $D^{b}_{\rm ctf}(X_{\rm \acute{e}t}, {\mathbb Z}/{p^{n}{\mathbb Z}})$ consisting of complexes whose cohomology sheaves are constructible and which have finite Tor dimension over ${\mathbb Z}/{p^{n}{\mathbb Z}}$.

For a morphism $f: X\to Y$ of smooth $W_{n}$-schemes,
 the inverse image 
\begin{equation*}
f^{-1}: D^{b}_{\rm ctf}(Y_{\rm \acute{e}t}, {\mathbb Z}/{p^{n}{\mathbb Z}})\to D^{b}_{\rm ctf}(X_{\rm \acute{e}t}, {\mathbb Z}/{p^{n}{\mathbb Z}})
\end{equation*}
and the direct image with proper support 
\begin{equation*}
f_{!}: D^{b}_{\rm ctf}(X_{\rm \acute{e}t}, {\mathbb Z}/{p^{n}{\mathbb Z}})\to D^{b}_{\rm ctf}(Y_{\rm \acute{e}t}, {\mathbb Z}/{p^{n}{\mathbb Z}})
\end{equation*}
are defined.
For a review of constructions of these functors, we refer the reader to \cite[\S 8]{EK}.

Let $X$ be a smooth $W_{n}$-scheme. 
We denote by $\pi_{X}: X_{\rm \acute{e}t}\to X$  the natural morphism of sites,
where $X$ means the Zariski site of $X$.
Then ${\mathcal D}_{F, X_{\rm \acute{e}t}}:=\pi^{*}_{X}{\mathcal D}_{F, X}$ naturally forms a sheaf of associative $W_{n}$-algebras on $X_{\rm \acute{e}t}$.
By \'etale descent, we have an equivalence of triangulated categories (cf. \cite[\S7 and 16.1.1]{EK})
\begin{equation*}
\pi^{*}_{X}: D^{b}_{\rm qc}({\mathcal D}_{F, X})\to D^{b}_{\rm qc}({\mathcal D}_{F, X_{\rm \acute{e}t}})
\end{equation*} 
with quasi-inverse $\pi_{X*}$.
For ${\mathcal M}\in D^{b}_{\rm lfgu}({\mathcal D}_{F, X})^{\circ}$, we set
\begin{equation*}
{\rm Sol}_{X}({\mathcal M})={\mathbb R}\underline{\rm Hom}_{{\mathcal D}_{F, X_{\rm \acute{e}t}}}(\pi_{X}^{*}({\mathcal M}), {\mathcal O}_{X_{\rm \acute{e}t}})[d_{X}].
\end{equation*}
Then this correspondence defines a contravariant functor
\begin{equation*}
{\rm Sol}_{X}: D^{b}_{\rm lfgu}({\mathcal D}_{F, X})^{\circ}\to D^{b}_{\rm ctf}(X_{\rm \acute{e}t}, {\mathbb Z}/{p^{n}{\mathbb Z}})
\end{equation*}
by \cite[Proposition 16.1.7]{EK}.
Conversely, for ${\mathcal L}\in D^{b}_{\rm ctf}(X_{\rm \acute{e}t}, {\mathbb Z}/{p^{n}{\mathbb Z}})$, we set
\begin{equation*}
{\rm M}_{X}({\mathcal L})=\pi_{X*}{\mathbb R}\underline{\rm Hom}_{{\mathbb Z}/{p^{n}{\mathbb Z}}}({\mathcal L}, {\mathcal O}_{X_{\rm \acute{e}t}})[d_{X}].
\end{equation*}
Then this correspondence defines a contravariant functor
\begin{equation*}
{\rm M}_{X}: D^{b}_{\rm ctf}(X_{\rm \acute{e}t}, {\mathbb Z}/{p^{n}{\mathbb Z}})\to D^{+}({\mathcal D}_{F, X}).
\end{equation*}
Now we may state one of the main results in \cite{EK}.
\begin{theo}\label{t4}
For a smooth $W_{n}$-scheme $X$,
the functor ${\rm Sol}_{X}$ is an anti-equivalence of triangulated categories between $D^{b}_{\rm lfgu}({\mathcal D}_{F, X})^{\circ}$ and $D^{b}_{\rm ctf}(X_{\rm \acute{e}t}, {\mathbb Z}/{p^{n}{\mathbb Z}})$ with quasi-inverse ${\rm M}_{X}$.
Furthermore ${\rm Sol}_{X}$ and ${\rm M}_{X}$ satisfy the following properties:

{\rm (1)} If $f: X\to Y$ is a morphism of smooth $W_{n}$-schemes, then  ${\rm Sol}_{X}$ and ${\rm M}_{X}$ interchange $f^{!}$ and $f^{-1}$.

{\rm (2)} Let $f$ be a morphism of smooth $W_{n}$-schemes such that $f$ can be factored as $f=g\circ h$,
where $g$ is an immersion of smooth $W_{n}$-schemes and $h$ is a proper smooth morphism of smooth $W_{n}$-schemes.
Then ${\rm Sol}_{X}$ and ${\rm M}_{X}$ interchange $f_{+}$ and $f_{!}$.

{\rm (3)} ${\rm Sol}_{X}$ and ${\rm M}_{X}$ interchange the functors $\otimes^{\mathbb L}_{{\mathcal O}_{X}}$ and $\otimes^{\mathbb L}_{{\mathbb Z}/{p^{n}{\mathbb Z}}}$ up to shift.
More precisely, for objects $\mathcal M$ and $\mathcal N$ in $D^{b}_{\rm lfgu}({\mathcal D}_{F, X})^{\circ}$,
there exists a canonical isomorphism
$${\rm Sol}_{X}({\mathcal M})\otimes^{\mathbb L}_{{\mathbb Z}/{p^{n}{\mathbb Z}}} {\rm Sol}_{X}({\mathcal N})\xrightarrow{\cong} {\rm Sol}_{X}\left({\mathcal M}\otimes^{\mathbb L}_{{\mathcal O}_{X}}{\mathcal N} \right)[d_{X}].$$
\end{theo}

\begin{proof}
See \cite[Proposition 16.1.10 and Corollary 16.2.6]{EK}.
\end{proof}

\subsection{Remark in the case $n=1$}\label{s2}
Let $X$ be a smooth $k$-scheme and assume that $n=1$ in this subsection.
Let ${\mathcal O}_{F, X}$ denote a sheaf of the non-commutative polynomial ring ${\mathcal O}_{X}[F]$ in a formal variable $F$, which satisfies the relation $Fa=a^{p}F$ for $a\in {\mathcal O}_{X}$.
One can naturally regard ${\mathcal O}_{F, X}$ as a subring of ${\mathcal D}_{F, X}$.
Giving an ${\mathcal O}_{F, X}$-module $\mathcal M$ is equivalent to giving an ${\mathcal O}_{X}$-module $\mathcal M$ with a structural morphism $F^{*}{\mathcal M}\to {\mathcal M}$, 
where $F$ denotes the absolute Frobenius on $X$.
We say an ${\mathcal O}_{F, X}$-module $\mathcal M$ is unit if it is quasi-coherent as an ${\mathcal O}_X$-module and the structural morphism $F^{*}{\mathcal M}\to {\mathcal M}$ is an isomorphism.
We say an ${\mathcal O}_{F, X}$-module $\mathcal M$ is locally finitely generated unit if it is unit and locally finitely generated as an ${\mathcal O}_{F, X}$-module.
Similar to the case of ${\mathcal D}_{F, X}$-modules, the locally finitely generated unit ${\mathcal O}_{F, X}$-modules form a thick subcategory of the category of quasi-coherent ${\mathcal O}_{F, X}$-modules.
So one can consider the bounded derived category $D^{b}_{\rm lfgu}({\mathcal O}_{F, X})$ of complexes of ${\mathcal O}_{F, X}$-modules whose cohomology sheaves are locally finitely generated unit.
Similar to the case of ${\mathcal D}_{F, X}$-modules, one can define the inverse and direct image functors for a morphism of smooth $k$-schemes (see \cite[\S2 and \S3]{EK}) and the derived tensor product on $D^{b}_{\rm lfgu}({\mathcal O}_{F, X})$.
Then Emerton and Kisin proved that the natural functor 
\begin{equation*}
D^{b}_{\rm lfgu}({\mathcal D}_{F, X})\to D^{b}_{\rm lfgu}({\mathcal O}_{F, X})
\end{equation*}
induces an equivalence of categories with quasi inverse ${\mathcal D}_{F, X}\otimes^{\mathbb L}_{{\mathcal O}_{F, X}}(-)$, which is compatible with the functors $f_{+}$, $f^{!}$ and $\otimes^{\mathbb L}_{{\mathcal O}_{X}}$, 
where $f$ is a morphism of smooth $k$-schemes \cite[Proposition 15.4.3]{EK}.
\begin{rema}
In \cite{EK}, Emerton and Kisin firstly established the theory of ${\mathcal O}_{F, X}$-modules for smooth $k$-schemes.
They proved many properties of ${\mathcal D}_{F, X}$-modules for smooth $W_{n}$-schemes including Theorem \ref{t4} by reducing them to the corresponding properties of ${\mathcal O}_{F, X\otimes_{W_{n}} k}$-modules.
\end{rema}


\section{Local cohomology functor}
Let $P$ be a smooth $W_{n}$-scheme.
Let $Z$ be a closed subset of $P$ and $j_{Z}$ the canonical  open immersion $P\setminus Z\hookrightarrow P$.
For a sheaf $\mathcal F$ of abelian groups on $P$, we set ${\Gamma}_{Z}{\mathcal F}:={\rm Ker}({\mathcal F}\to j_{Z*}j_{Z}^{-1}{\mathcal F})$.
If $\mathcal M$ is a left ${\mathcal D}_{F, P}$-module, then ${\Gamma}_{Z}{\mathcal M}$ naturally forms a left ${\mathcal D}_{F, P}$-module.
We have a left exact functor ${\Gamma}_{Z}$ from the category of left ${\mathcal D}_{F, P}$-modules to itself.
Then the local cohomology functor
\begin{equation*}
{\mathbb R}{\Gamma}_{Z}: D^{+}({\mathcal D}_{F, P})\to D^{+}({\mathcal D}_{F, P})
\end{equation*}
is defined to be the right derived functor of ${\Gamma}_{Z}$.
By definition, we have a distinguished triangle 
\begin{equation}\label{e5}
{\mathbb R}{\Gamma}_{Z}{\mathcal M}\to {\mathcal M}\to {\mathbb R}j_{Z*}j_{Z}^{-1}{\mathcal M}\xrightarrow{+}
\end{equation}
for ${\mathcal M}\in D^{+}({\mathcal D}_{F, P})$.
Note that ${\mathbb R}j_{Z*}=j_{Z+}$ and $j_{Z}^{-1}=j_{Z}^{!}$.
We can also define the local cohomology functor 
$${\mathbb R}{\Gamma}_{Z}: D^{+}({\mathcal O}_{P})\to D^{+}({\mathcal O}_{P})$$
on the level of ${\mathcal O}_{P}$-modules.
Then the forgetful functor $D^{+}({\mathcal D}_{F, P})\to D^{+}({\mathcal O}_{P})$ commutes with ${\mathbb R}{\Gamma}_{Z}$.
It is proved by Grothendieck that ${\mathbb R}{\Gamma}_{Z}$ has finite cohomological amplitude.


\begin{lemm}\label{l5}
Let $P$ be a smooth $W_{n}$-scheme and $Z$ a closed subset of $P$.
Denote by $j_{Z}$ the open immersion $P\setminus Z\hookrightarrow P$.
Then the following conditions are equivalent for ${\mathcal M}\in D({\mathcal O}_{P})$.
\begin{enumerate}
\item ${\mathbb R}{\Gamma}_{Z}{\mathcal M}\xrightarrow{\cong} {\mathcal M}$.
\item ${\mathbb R}j_{Z*}j_{Z}^{-1}{\mathcal M}$=0.
\item ${\rm Supp}{\mathcal M}$ is contained in $Z$.
\end{enumerate}
\end{lemm}

\begin{proof}
The equivalence of $1$ and $2$ follows from (\ref{e5}).
Assuming that ${\rm Supp}{\mathcal M} \subset Z$,
one has $j_{Z}^{-1}{\mathcal M}=0$.
This shows $3\Rightarrow 2$.
Finally $1\Rightarrow 3$ is evident.
\end{proof}

\begin{lemm}\label{l3}
Let $P$ be a smooth $W_{n}$-scheme and $Z$ a closed subset of $P$.
There exists a natural ${\mathcal O}_{P}$-linear isomorphism 
\begin{equation}\label{e11}
{\mathbb R}{\Gamma}_{Z} \left({\mathcal M}\right) \otimes_{{\mathcal O}_{P}}^{\mathbb L} {\mathcal N}\xrightarrow{\cong} 
{\mathbb R}{\Gamma}_{Z} \left({\mathcal M}\otimes_{{\mathcal O}_{P}}^{\mathbb L} {\mathcal N} \right)
\end{equation}
for any ${\mathcal M}\in D^{-}({\mathcal O}_{P})$ and ${\mathcal N}\in D^{-}_{\rm qc}({\mathcal O}_{P})$.
Furthermore for any ${\mathcal M}\in D^{-}({\mathcal D}_{P})$ (resp. ${\mathcal M}\in D^{-}({\mathcal D}_{F, P})$) and ${\mathcal N}\in D^{-}_{\rm qc}({\mathcal D}_{P})$ (resp. ${\mathcal N}\in D^{-}_{\rm qc}({\mathcal D}_{F, P})$) (\ref{e11}) is a ${\mathcal D}_{P}$-linear (resp. ${\mathcal D}_{F, P}$-linear) isomorphism.
\end{lemm}
\begin{proof}
Note first that both sides are well-defined.
Let us construct a natural morphism in the Lemma.
Let ${\mathcal M}$ be an object of $D^{-}({\mathcal O}_{P})$ and ${\mathcal N}$ an object of $D^{-}_{\rm qc}({\mathcal O}_{P})$.
One has ${\mathbb R}{\Gamma}_{Z} \left({\mathcal M}\right) \otimes_{{\mathcal O}_{P}}^{\mathbb L} {\mathcal N}\to {\mathcal M} \otimes_{{\mathcal O}_{P}}^{\mathbb L} {\mathcal N}$. 
Then since ${\mathbb R}{\Gamma}_{Z} \left({\mathcal M}\right) \otimes_{{\mathcal O}_{P}}^{\mathbb L} {\mathcal N}$ is supported on $Z$,
we have ${\mathbb R}{\Gamma}_{Z}\left({\mathbb R}{\Gamma}_{Z} \left({\mathcal M}\right) \otimes_{{\mathcal O}_{P}}^{\mathbb L} {\mathcal N}\right)\xrightarrow{\cong} {\mathbb R}{\Gamma}_{Z} \left({\mathcal M}\right) \otimes_{{\mathcal O}_{P}}^{\mathbb L} {\mathcal N}$ by Lemma \ref{l5}.
So ${\mathbb R}{\Gamma}_{Z} \left({\mathcal M}\right) \otimes_{{\mathcal O}_{P}}^{\mathbb L} {\mathcal N}\to {\mathcal M} \otimes_{{\mathcal O}_{P}}^{\mathbb L} {\mathcal N}$ uniquely factors as 
\begin{equation*}
{\mathbb R}{\Gamma}_{Z} \left({\mathcal M}\right) \otimes_{{\mathcal O}_{P}}^{\mathbb L} {\mathcal N}\to {\mathbb R}{\Gamma}_{Z} \left({\mathcal M}\otimes_{{\mathcal O}_{P}}^{\mathbb L} {\mathcal N} \right)\to  {\mathcal M} \otimes_{{\mathcal O}_{P}}^{\mathbb L} {\mathcal N}
\end{equation*}
and we get the desired morphism.
Note that if ${\mathcal M}$ is an object of $D^{-}({\mathcal D}_{P})$ (resp. $D^{-}({\mathcal D}_{F, P})$) and ${\mathcal N}$ is an object of $D^{-}_{\rm qc}({\mathcal D}_{P})$ (resp. $D^{-}_{\rm qc}({\mathcal D}_{F, P})$) then (\ref{e11}) is ${\mathcal D}_{P}$-linear (resp. ${\mathcal D}_{F, P}$-linear).
Let us prove that (\ref{e11}) is an isomorphism.
It suffices to show that this is an isomorphism in $D({\mathcal O}_{P})$.
The assertion is Zariski local on $P$ and so we may assume that $P$ is affine.
Note that the source is a way-out left functor in $\mathcal N$.
Also, since ${\mathbb R}{\Gamma}_{Z}$ is finite cohomological amplitude, the target is also a way-out left functor in $\mathcal N$.
Using the lemma on way-out functors (cf. \cite[Chap I, Proposition 7.1]{Ha}), we reduce to the case where $\mathcal N$ is a single quasi-coherent ${\mathcal O}_{P}$-module.
Furthermore since any quasi-coherent ${\mathcal O}_{P}$-module is a quotient of a free ${\mathcal O}_{P}$-module (because $P$ is affine), we may assume that $\mathcal N$ is a single free ${\mathcal O}_{P}$-module.
Now since ${\mathbb R}{\Gamma}_{Z}$ commutes with infinite direct sums, we reduce the assertion to prove the case when ${\mathcal N}={\mathcal O}_{P}$.
Then both sides are equal to ${\mathbb R}{\Gamma}_{Z}{\mathcal M}$ and we are done.
\end{proof}

\begin{prop}
Let $P$ be a smooth $W_{n}$-scheme and $Z$ a closed subset of $P$.
The local cohomology functor induces a functor
\begin{equation*}
{\mathbb R}{\Gamma}_{Z}: D^{b}_{\rm lfgu}({\mathcal D}_{F, P})^{\circ}\to D^{b}_{\rm lfgu}({\mathcal D}_{F, P})^{\circ}.
\end{equation*}
\end{prop}
\begin{proof}
Let us first show that, for any ${\mathcal M}\in  D^{b}_{\rm lfgu}({\mathcal D}_{F, P})$, ${\mathbb R}{\Gamma}_{Z}{\mathcal M}$
has locally finitely generated cohomology sheaves.
Let $j_{Z}$ denote the open immersion $P\setminus Z\hookrightarrow P$.
Then there exists a distinguished triangle
\begin{equation*}
{\mathbb R}{\Gamma}_{Z}{\mathcal M}\to {\mathcal M}\to {\mathbb R}j_{Z*}j_{Z}^{-1}{\mathcal M}\xrightarrow{+}.
\end{equation*}
Since ${\mathcal M}$ and ${\mathbb R}j_{Z*}j_{Z}^{-1}{\mathcal M}$ are objects of  $D^{b}_{\rm lfgu}({\mathcal D}_{F, P})$
by \cite[Proposition 15.5.1]{EK}, ${\mathbb R}{\Gamma}_{Z}{\mathcal M}$ is also an object of $D^{b}_{\rm lfgu}({\mathcal D}_{F, P})$ by \cite[Proposition 15.3.4]{EK}.
Next we show that, for any ${\mathcal M}\in  D^{b}_{\rm lfgu}({\mathcal D}_{F, P})^{\circ}$, ${\mathbb R}{\Gamma}_{Z}{\mathcal M}$ is of finite Tor dimension over ${\mathcal O}_{P}$.
According to  \cite[I, Proposition 5.1]{SGA6}, it is enough to show that ${\mathbb R}{\Gamma}_{Z} \left({\mathcal M}\right) \otimes_{{\mathcal O}_{P}}^{\mathbb L} {\mathcal N}$ is a bounded complex for any ${\mathcal O}_{P}$-module $\mathcal N$.
First of all, suppose that $\mathcal N$ is a quasi-coherent ${\mathcal O}_{P}$-module.
Then, by Lemma \ref{l3}, we have ${\mathbb R}{\Gamma}_{Z} \left({\mathcal M}\right) \otimes_{{\mathcal O}_{P}}^{\mathbb L} {\mathcal N}\xrightarrow{\cong} {\mathbb R}{\Gamma}_{Z} \left({\mathcal M}\otimes_{{\mathcal O}_{P}}^{\mathbb L} {\mathcal N} \right)$.
Since ${\mathcal M}$ is of finite Tor dimension, ${\mathcal M}\otimes_{{\mathcal O}_{P}}^{\mathbb L} {\mathcal N}$ is a bounded complex of ${\mathcal O}_{P}$-modules. 
So we know that ${\mathbb R}{\Gamma}_{Z} \left({\mathcal M}\right) \otimes_{{\mathcal O}_{P}}^{\mathbb L} {\mathcal N}$ is bounded since ${\mathbb R}{\Gamma}_{Z}$ is finite cohomological amplitude.
Now if $\mathcal N$ is an arbitrary ${\mathcal O}_{P}$-module, then the stalks of ${\mathbb R}{\Gamma}_{Z} \left({\mathcal M}\right) \otimes_{{\mathcal O}_{P}}^{\mathbb L} {\mathcal N}$ are uniformly bounded, hence so is ${\mathbb R}{\Gamma}_{Z} \left({\mathcal M}\right) \otimes_{{\mathcal O}_{P}}^{\mathbb L} {\mathcal N}$ because $P$ is quasi-compact.
\end{proof}

\begin{lemm}\label{l4}
Let $f: P\to Q$ be a morphism of smooth $W_{n}$-schemes and $Z_{Q}$ a closed subset of $Q$.
We denote by $Z_{P}$ the inverse image of $Z_{Q}$.
There exists a natural isomorphism in $D^{-}({\mathcal D}_{F, P})$
\begin{equation*}
{\mathbb L}f^{*}\circ {\mathbb R}{\Gamma}_{Z_{Q}}({\mathcal O}_{Q})\xrightarrow{\cong} {\mathbb R}{\Gamma}_{Z_{P}}({\mathcal O}_{P}).
\end{equation*}
\end{lemm}
\begin{proof}
One has natural morphisms
\begin{equation}\label{e41}
{\mathbb L}f^{*}{\mathbb R}{\Gamma}_{Z_{Q}}({\mathcal O}_{Q})\to {\mathbb L}f^{*}{\mathcal O}_{Q}\to {\mathcal O}_{P}.
\end{equation}
Let us denote by $j_{Z_{Q}}$ (resp. $j_{Z_{P}}$) the open immersion $Q\setminus Z_{Q}\hookrightarrow Q$ (resp. $P\setminus Z_{P}\hookrightarrow P$) and by $f'$ the restriction of $f$ to $P\setminus Z_{P}$.
Then one has 
$$j_{Z_{P}}^{-1}{\mathbb L}f^{*}{\mathbb R}{\Gamma}_{Z_{Q}}({\mathcal O}_{Q})\cong {\mathbb L}f'^{*}j_{Z_{Q}}^{-1} {\mathbb R}{\Gamma}_{Z_{Q}}({\mathcal O}_{Q})\cong 0.$$
Hence, we know that ${\mathbb L}f^{*}{\mathbb R}{\Gamma}_{Z_{Q}}({\mathcal O}_{Q})$ is supported on $Z_{P}$ and the morphism (\ref{e41})  
uniquely factors as
\begin{equation*}
{\mathbb L}f^{*}{\mathbb R}{\Gamma}_{Z_{Q}}{\mathcal O}_{Q}\xrightarrow{a} {\mathbb R}{\Gamma}_{Z_{P}}({\mathcal O}_{P}) \to {\mathcal O}_{P}.
\end{equation*}
It suffices to prove that $a$ is an isomorphism in $D({\mathcal O}_P)$.
One has a morphism of distinguished triangles
\[\xymatrix{
{{\mathbb L}f^{*}{\mathbb R}{\Gamma}_{Z_{Q}}{\mathcal O}_{Q}}\ar[r] \ar[d]_a & {{\mathbb L}f^{*}{\mathcal O}_{Q}} \ar[d]_b \ar[r] & {\mathbb L}f^{*}{\mathbb R}j_{Z_{Q}*}j_{Z_{Q}}^{-1} {\mathcal O}_{Q}\ar[d]_c\ar[r]^-+& \\
{{\mathbb R}{\Gamma}_{Z_{P}}({\mathcal O}_{P})}\ar[r] & {{\mathcal O}_{P}} \ar[r] & {\mathbb R}j_{Z_{P}*}j_{Z_{P}}^{-1}{{\mathcal O}_{P}} \ar[r]^-+&.}
\]
Here $c$ is defined to be the composite of morphisms
\begin{eqnarray*}
{\mathbb L}f^{*}{\mathbb R}j_{Z_{Q}*}j_{Z_{Q}}^{-1}{{\mathcal O}_{Q}} &\to& {\mathbb R}j_{Z_{P}*}j_{Z_{P}}^{-1}{\mathbb L}f^{*}{\mathbb R}j_{Z_{Q}*}j_{Z_{Q}}^{-1}{{\mathcal O}_{Q}}\\
&\cong& {\mathbb R}j_{Z_{P}*}{\mathbb L}f'^{*}j_{Z_{Q}}^{-1}{\mathbb R}j_{Z_{Q}*}j_{Z_{Q}}^{-1}{{\mathcal O}_{Q}}\\
&\xrightarrow{\cong}& {\mathbb R}j_{Z_{P}*}{\mathbb L}f'^{*}j_{Z_{Q}}^{-1}{{\mathcal O}_{Q}}\\
&\cong& {\mathbb R}j_{Z_{P}*}j_{Z_{P}}^{-1}{\mathbb L}f^{*}{{\mathcal O}_{Q}}\\
&\to & {\mathbb R}j_{Z_{P}*}j_{Z_{P}}^{-1}{{\mathcal O}_{P}},
\end{eqnarray*}
where the first morphism and the third one are induced from the adjunction morphisms ${\rm id}\to  {\mathbb R}j_{Z_{P}*}j_{Z_{P}}^{-1}$ and
$j_{Z_{Q}}^{-1}{\mathbb R}j_{Z_{Q}*}\xrightarrow{\cong} {\rm id}$ respectively.
Then $b$ is evidently an isomorphism and $c$ is an isomorphism by \cite[Lemma 35.18.3]{Stacks}.
So $a$ is an isomorphism.

\end{proof}

\begin{prop}\label{p4}
Let $f: P\to Q$ be a morphism of smooth $W_{n}$-schemes and $Z_{Q}$ a closed subset of $Q$.
We denote by $Z_{P}$ the inverse image of $Z_{Q}$.
Then, for any ${\mathcal M}\in D^{-}_{\rm qc}({\mathcal D}_{F, Q})$, there exists a natural isomorphism
\begin{equation*}
 {\mathbb L}f^{*}\circ {\mathbb R}{\Gamma}_{Z_{Q}}{\mathcal M}\xrightarrow{\cong} {\mathbb R}{\Gamma}_{Z_{P}}\circ {\mathbb L}f^{*}{\mathcal M}
\end{equation*}
and also a natural isomorphism
\begin{equation*}
 f^{!}\circ {\mathbb R}{\Gamma}_{Z_{Q}}{\mathcal M}\xrightarrow{\cong} {\mathbb R}{\Gamma}_{Z_{P}}\circ f^{!}{\mathcal M}.
\end{equation*}
\end{prop}
\begin{proof}
The second isomorphism follows by applying the shift operator to the first one.
By using Proposition \ref{p5}, Lemma \ref{l3} and Lemma \ref{l4}, we obtain
\begin{eqnarray*}
{\mathbb L}f^{*} {\mathbb R}{\Gamma}_{Z_{Q}}({\mathcal M}) &\xleftarrow{\cong}& {\mathbb L}f^{*}\left({\mathbb R}{\Gamma}_{Z_{Q}}({\mathcal O}_{Q})\otimes^{\mathbb L}_{{\mathcal O}_{Q}}{\mathcal M}\right)\\
&\xrightarrow{\cong}& {\mathbb L}f^{*}{\mathbb R}{\Gamma}_{Z_{Q}}({\mathcal O}_{Q})\otimes^{\mathbb L}_{{\mathcal O}_{P}}{\mathbb L}f^{*}{\mathcal M}\\
&\xrightarrow{\cong}&{\mathbb R}{\Gamma}_{Z_{P}}({\mathcal O}_{P}) \otimes^{\mathbb L}_{{\mathcal O}_{P}}{\mathbb L}f^{*}{\mathcal M}={\mathbb R}{\Gamma}_{Z_{P}}({\mathbb L}f^{*}{\mathcal M}). 
\end{eqnarray*}

\end{proof}

Next we show the compatibility of the local cohomology functor and the direct image.
We begin with the corresponding result for usual ${\mathcal D}_{P}$-modules (without Frobenius structures).

\begin{prop}\label{p9}
Let $f: P\to Q$ be a morphism of smooth $W_{n}$-schemes and $Z_{Q}$ a closed subset of $Q$.
We denote by $Z_{P}$ the inverse image of $Z_{Q}$.
Let $\mathcal M$ be an object in $D_{\rm qc}^{b}({\mathcal D}_{P})$.
Then there exists a natural isomorphism of functors
\begin{equation*}
{\mathbb R}{\Gamma}_{Z_{Q}}\circ f^{\mathcal B}_{+}({\mathcal M})\to f^{\mathcal B}_{+}\circ {\mathbb R}{\Gamma}_{Z_{P}}({\mathcal M}).
\end{equation*}
\end{prop}
We need some lemmas.
\begin{prop}\label{p6}
Let $f: P\to Q$ be a morphism of smooth $W_{n}$-schemes.
If ${\mathcal M}$ is an object in $D^{-}_{\rm qc}({\mathcal D}_{Q})$ and $\mathcal N$ is an object in $D^{-}(f^{-1}{\mathcal D}_{Q})$, 
then there exists a natural isomorphism
\begin{equation*}
{\mathcal M}\otimes^{\mathbb L}_{{\mathcal O}_{P}} {\mathbb R}f_{*}{\mathcal N}\xrightarrow{\cong}{\mathbb R}f_{*}\left(f^{-1}{\mathcal M}\otimes^{\mathbb L}_{f^{-1}{\mathcal O}_{Q}}{\mathcal N}\right).
\end{equation*}
in $D^{-}({\mathcal D}_{Q})$.
\end{prop}
\begin{proof}
Note first that both sides are defined.
Let us take an $f_{*}$-acyclic resolution $\mathcal I$ of $\mathcal N$ and a ${\mathcal D}_{Q}$-flat resolution 
$\mathcal P$ of $\mathcal M$.
Then we have a natural ${\mathcal D}_{Q}$-linear morphism
\begin{eqnarray*}
{\mathcal M}\otimes^{\mathbb L}_{{\mathcal O}_{P}} {\mathbb R}f_{*}{\mathcal N}&:=& {\mathcal P}\otimes_{{\mathcal O}_{P}} f_{*}{\mathcal I}
\to f_{*} \left(f^{-1}{\mathcal P}\otimes_{f^{-1}{\mathcal O}_{Q}}{\mathcal I}\right)\\
&\to& {\mathbb R}f_{*} \left(f^{-1}{\mathcal P}\otimes_{f^{-1}{\mathcal O}_{Y}}{\mathcal I}\right)={\mathbb R}f_{*}\left(f^{-1}{\mathcal M}\otimes^{\mathbb L}_{f^{-1}{\mathcal O}_{Y}}{\mathcal N}\right).
\end{eqnarray*} 
It is enough to prove that this is an isomorphism in $D({\mathcal O}_{Q})$.
Then this follows from \cite[II, Proposition 5.6]{Ha}.
\end{proof}
\begin{lemm}\label{l18}
Let $f: P\to Q$ be a morphism of smooth $W_{n}$-schemes.
For an object $\mathcal E$ in $D^{-}({\mathcal D}_{Q})$ and an object $\mathcal F$ in $D^{-}({\mathcal D}_{P})$, there exists an isomorphism
\begin{equation*}
\left({f^{-1}{\mathcal E}}\otimes^{\mathbb L}_{f^{-1}{\mathcal O}_{Q}}{\mathcal D}_{Q\leftarrow P}\right)\otimes^{\mathbb L}_{{\mathcal D}_{P}}{\mathcal F}\xrightarrow{\cong}{{\mathcal D}_{Q\leftarrow P}}\otimes^{\mathbb L}_{{\mathcal D}_{P}}\left({\mathbb L}f^{*}{\mathcal E}\otimes^{\mathbb L}_{{\mathcal O}_{P}}{\mathcal F}\right)
\end{equation*}
in $D^{b}(f^{-1}{\mathcal D}_{Q})$.
\end{lemm}
\begin{proof}
The proof is the same as that of the corresponding proposition for ${\mathcal D}$-modules of higher level proved in \cite[Proposition 1.2.25]{Ca}.
\end{proof}

Let us prove Proposition \ref{p9}.
\begin{proof}
 Applying Proposition \ref{p6} to the case with ${\mathcal M}={\mathbb R}{\Gamma}_{Z_{Q}}({\mathcal O}_{Q})$ and 
 ${\mathcal N}={\mathcal D}_{Q\leftarrow P}\otimes^{\mathbb L}_{{\mathcal D}_{P}}{\mathcal M}$, 
we obtain 
\begin{align*}
&{{\mathbb R}{\Gamma}_{Z_{Q}}({\mathcal O}_{Q})}\otimes^{\mathbb L}_{{\mathcal O}_{Q}} {\mathbb R}f_{*}\left({{\mathcal D}_{Q\leftarrow P}\otimes^{\mathbb L}_{{\mathcal D}_{P}}{\mathcal M}}\right)\\
&\,\,\,\,\,\,\,\,\,\, \xrightarrow{\cong}{\mathbb R}f_{*}\left(f^{-1}\left({{\mathbb R}{\Gamma}_{Z_{Q}}({\mathcal O}_{Q}})\right)\otimes^{\mathbb L}_{f^{-1}{\mathcal O}_{Q}}\left({{\mathcal D}_{Q\leftarrow P}\otimes^{\mathbb L}_{{\mathcal D}_{P}}{\mathcal M}}\right)\right).
\end{align*}
The left hand side is isomorphic to ${\mathbb R}{\Gamma}_{Z_{Q}}\circ f^{\mathcal B}_{+}(\mathcal M)$ by Lemma \ref{l3}.
On the other hand, by Lemma \ref{l18}, we have
\begin{align*}
&f^{-1}{{\mathbb R}{\Gamma}_{Z_{Q}}({\mathcal O}_{Q}})\otimes^{\mathbb L}_{f^{-1}{\mathcal O}_{Q}}{{\mathcal D}_{Q\leftarrow P}\otimes^{\mathbb L}_{{\mathcal D}_{P}}{\mathcal M}}\\
&\,\,\,\,\,\,\,\,\,\, {\cong} ({f^{-1}{{\mathbb R}{\Gamma}_{Z_{Q}}({\mathcal O}_{Q}})\otimes^{\mathbb L}_{f^{-1}{\mathcal O}_{Q}}{\mathcal D}_{Q\leftarrow P}) \otimes^{\mathbb L}_{{\mathcal D}_{P}}{\mathcal M}}\\
&\,\,\,\,\,\,\,\,\,\, {\cong}{\mathcal D}_{Q\leftarrow P} \otimes^{\mathbb L}_{{\mathcal D}_{P}}\left({\mathbb L}f^{*}{{\mathbb R}{\Gamma}_{Z_{Q}}({\mathcal O}_{Q}}) \otimes^{\mathbb L}_{{\mathcal O}_{P}}{\mathcal M}\right).
\end{align*}
Lemma \ref{l3} and Proposition \ref{p4} imply that
\begin{align*}
&{\mathcal D}_{Q\leftarrow P} \otimes^{\mathbb L}_{{\mathcal D}_{P}}\left({\mathbb L}f^{*}{{\mathbb R}{\Gamma}_{Z_{Q}}({\mathcal O}_{Q}}) \otimes^{\mathbb L}_{{\mathcal O}_{P}}{\mathcal M}\right)\\
&\,\,\,\,\,\,\,\,\,\,  \xrightarrow{\cong} {\mathcal D}_{Q\leftarrow P} \otimes^{\mathbb L}_{{\mathcal D}_{P}}\left({{\mathbb R}{\Gamma}_{Z_{P}}({\mathcal O}_{P})}\otimes^{\mathbb L}_{{\mathcal O}_{P}} {\mathcal M}\right)\\
&\,\,\,\,\,\,\,\,\,\,  \xrightarrow{\cong} {\mathcal D}_{Q\leftarrow P} \otimes^{\mathbb L}_{{\mathcal D}_{P}}{{\mathbb R}{\Gamma}_{Z_{P}}({\mathcal M})}.
\end{align*}
Therefore the right hand side of the first isomorphism is isomorphic to 
\begin{equation*}
{\mathbb R}f_{*}  \left({\mathcal D}_{Q\leftarrow P} \otimes^{\mathbb L}_{{\mathcal D}_{P}}{{\mathbb R}{\Gamma}_{Z_{P}}({\mathcal M})}\right)=f^{\mathcal B}_{+}\circ {\mathbb R}{\Gamma}_{Z_{P}}(\mathcal M).
\end{equation*}

\end{proof}

\begin{prop}\label{p1}
Let $f: P\to Q$ be a morphism of smooth $W_{n}$-schemes and $Z_{Q}$ a closed subset of $Q$.
We denote by $Z_{P}$ the inverse image of $Z_{Q}$.
Let $\mathcal M$ be an object in $D^{b}_{\rm qc}({\mathcal D}_{F, P})$.
Then there exists a natural isomorphism 
\begin{equation*}
{\mathbb R}{\Gamma}_{Z_{Q}}\circ f_{+}({\mathcal M})\to f_{+}\circ {\mathbb R}{\Gamma}_{Z_{P}}({\mathcal M}).
\end{equation*}
in $D^{b}({\mathcal D}_{F, Q})$
\end{prop}
\begin{proof}
Let us construct a natural transformation ${\mathbb R}{\Gamma}_{Z_{Q}}\circ f_{+}\to f_{+}\circ {\mathbb R}{\Gamma}_{Z_{P}}$.
For an object $\mathcal M$ in $D^{b}({\mathcal D}_{F, P})$,
the natural morphism ${\mathbb R}{\Gamma}_{Z_{P}}{\mathcal M}\to {\mathcal M}$ induces a morphism 
${\mathbb R}{\Gamma}_{Z_{Q}}f_{+}{\mathbb R}{\Gamma}_{Z_{P}}{\mathcal M}\to {\mathbb R}{\Gamma}_{Z_{Q}}f_{+}{\mathcal M}$.
Since, by Remark \ref{r2} and Proposition \ref{p9}, $f_{+}{\mathbb R}{\Gamma}_{Z_{P}}{\mathcal M}\cong f^{\mathcal B}_{+}{\mathbb R}{\Gamma}_{Z_{P}}{\mathcal M}\cong{\mathbb R}{\Gamma}_{Z_{Q}}f^{\mathcal B}_{+}{\mathcal M}$ as a complex of ${\mathcal D}_{Q}$-module, we know that 
$f_{+}{\mathbb R}{\Gamma}_{Z_{P}}{\mathcal M}$ is supported on $Z_{Q}$.
Therefore we have a natural morphism 
\begin{equation*}
f_{+}{\mathbb R}{\Gamma}_{Z_{P}}{\mathcal M}\xleftarrow{\cong}{\mathbb R}{\Gamma}_{Z_{Q}}f_{+}{\mathbb R}{\Gamma}_{Z_{P}}{\mathcal M}\to {\mathbb R}{\Gamma}_{Z_{Q}}f_{+}{\mathcal M}.
\end{equation*}
Again by Proposition \ref{p9} we conclude that it is an isomorphism.
\end{proof}

\section{Riemann-Hilbert correspondence for unit $F$-crystals}

\subsection{Category $D^{b}_{\rm lfgu}(X/W_{n})^{\circ}$}\label{s41}

\begin{defi}

Let $P$ be a smooth $W_{n}$-scheme and
let $Z$ and $T$ be closed subsets of $P$.
We define the category ${\mathcal C}_{P, Z, T}$ to be the full triangulated subcategory of $D^{b}_{\rm lfgu}({\mathcal D}_{F, P})^{\circ}$ consisting of complexes $\mathcal M$ satisfying 
\begin{equation}
{\mathbb R}{\Gamma}_{Z}{\mathcal M}\xrightarrow{\cong}{\mathcal M} \textit{ and\,\, } {\mathbb R}{\Gamma}_{T}{\mathcal M}=0.
\end{equation}
\end{defi}

\begin{lemm}\label{l1}
Let $P$ be a smooth $W_{n}$-scheme.
Let $Z$, $Z'$, $T$ and $T'$ be closed subsets of $P$ satisfying $Z\setminus T=Z'\setminus T'$.
Then we have the equality
\begin{equation*}
{\mathcal C}_{P, Z, T}={\mathcal C}_{P, Z', T'}.
\end{equation*}
\end{lemm}
\begin{proof}
First we prove the equality in the case $Z=Z'$.
One has $Z\cap T=Z\cap T'$.
Then an isomorphism ${\mathbb R}{\Gamma}_{Z}{\mathcal M}\xrightarrow{\cong} {\mathcal M}$ induces 
\begin{equation*}
{\mathbb R}{\Gamma}_{T}{\mathcal M}\xleftarrow{\cong}{\mathbb R}{\Gamma}_{Z\cap T}{\mathcal M}={\mathbb R}{\Gamma}_{Z\cap T'}{\mathcal M}\xrightarrow{\cong}{\mathbb R}{\Gamma}_{T'}{\mathcal M}.
\end{equation*}
Next we consider the case $T=T'$.
We have to show that ${\mathbb R}{\Gamma}_{Z}{\mathcal M}\xrightarrow{\cong}{\mathcal M}$ if and only if ${\mathbb R}{\Gamma}_{Z'}{\mathcal M}\xrightarrow{\cong}{\mathcal M}$ under the assumption ${\mathbb R}{\Gamma}_{T}{\mathcal M}=0$.
For a closed subset $C$ of $P$, let us denote by $j_{C}$ the canonical open immersion $P\setminus C\hookrightarrow P$.
Then the condition ${\mathbb R}{\Gamma}_{Z}{\mathcal M}\xrightarrow{\cong}{\mathcal M}$ is equivalent to the condition ${\mathbb R}j_{Z*}j_{Z}^{-1}{\mathcal M}=0$.
One always has ${\mathbb R}{\Gamma}_{T}{\mathbb R}j_{Z*}j_{Z}^{-1}{\mathcal M}\cong {\mathbb R}j_{Z*}j_{Z}^{-1}{\mathbb R}{\Gamma}_{T}{\mathcal M}=0$ by Proposition \ref{p4} and Proposition \ref{p1}.
By the distinguished triangle
\begin{equation*}
{\mathbb R}{\Gamma}_{T}{\mathbb R}j_{Z*}j_{Z}^{-1}{\mathcal M}\to {\mathbb R}j_{Z*}j_{Z}^{-1}{\mathcal M}\to 
{\mathbb R}j_{T*}j_{T}^{-1}{\mathbb R}j_{Z*}j_{Z}^{-1}{\mathcal M}\xrightarrow{+},
\end{equation*}
we see that the condition ${\mathbb R}j_{Z*}j_{Z}^{-1}{\mathcal M}=0$ is equivalent to the condition ${\mathbb R}j_{T*}j_{T}^{-1}{\mathbb R}j_{Z*}j_{Z}^{-1}{\mathcal M}=0$.
Let us denote by $j$ the open immersion $(P\setminus T)\setminus (Z\setminus T)=(P\setminus T)\setminus (Z'\setminus T)\hookrightarrow P\setminus T$ and by $j'$ the open immersion $(P\setminus T)\setminus (Z\setminus T)\hookrightarrow P$.
We have the following cartesian diagram:
\[\xymatrix{
{(P\setminus T)\setminus (Z\setminus T)} \ar[d] \ar@{^{(}-_>}[r]^{\,\,\,\,\,\,\,\,\,\,\,\,\,\,\,\,j} & {P\setminus T}  \ar@{^{(}-_>}[d]^{j_{T}} \\
{P\setminus Z} \ar@{^{(}-_>}[r]^{j_{Z}} & P. }
\]
Applying the flat base change theorem to the complex $j_{Z}^{-1}{\mathcal M}$ of ${\mathcal O}_{P\setminus Z}$-modules, we obtain
\begin{equation*}
j_{T}^{-1}{\mathbb R}j_{Z*}j_{Z}^{-1}{\mathcal M}\cong {\mathbb R}j_{*}j'^{-1}{\mathcal M}\cong j_{T}^{-1}{\mathbb R}j_{Z'*}j_{Z'}^{-1}{\mathcal M}.
\end{equation*}
Hence ${\mathbb R}j_{T*}j_{T}^{-1}{\mathbb R}j_{Z*}j_{Z}^{-1}{\mathcal M}=0$ if and only if
${\mathbb R}j_{T*}j_{T}^{-1}{\mathbb R}j_{Z'*}j_{Z'}^{-1}{\mathcal M}=0$.
The general case follows from these two cases.
\end{proof}

\begin{prop}\label{p2}
Let $P$ be a smooth $W_{n}$-scheme and $X$ a locally closed subset of $P$.
Let $j: U\hookrightarrow P$ be an open immersion of smooth $W_{n}$-schemes such that an immersion $X\hookrightarrow P$ factors as a closed immersion $X\hookrightarrow U$ and the open immersion $U\hookrightarrow P$.
Let $Z$ be a closed subset of $P$ such that $Z\cap U=X$.
We set $T:=\left(P\setminus U\right)\cap Z$.
Then the direct image functor ${\mathbb R}j_{*}(=j_{+})$ induces an equivalence of triangulated categories
\begin{equation*}
{\mathbb R}j_{*}: {\mathcal C}_{U, X, \emptyset}\to {\mathcal C}_{P, Z, T}
\end{equation*}
 with  quasi-inverse $j^{-1}(=j^{!})$.
\end{prop}
\begin{proof}
Firstly we shall see that the functors ${\mathbb R}j_{*}$ and $j^{-1}$ are well-defined.
Let ${\mathcal M}$ be an object in ${\mathcal C}_{U, X, \emptyset}$.
By Proposition \ref{p1}, we have ${\mathbb R}{\Gamma}_{Z}({\mathbb R}j_{*}{\mathcal M})\cong {\mathbb R}j_{*}{\mathbb R}{\Gamma}_{Z\cap U}{\mathcal M}={\mathbb R}j_{*}{\mathbb R}{\Gamma}_{X}{\mathcal M}\xrightarrow{\cong} {\mathbb R}j_{*}{\mathcal M}$.
We also have ${\mathbb R}{\Gamma}_{T}({\mathbb R}j_{*}{\mathcal M})=0$ as $T\cap U=\emptyset$ and thus know that ${\mathbb R}j_{*}$ restricts to a functor ${\mathcal C}_{U, X, \emptyset}\to {\mathcal C}_{P, Z, T}$.
Conversely, let ${\mathcal N}$ be an object in ${\mathcal C}_{P, Z, T}$.
Applying the functor $j^{-1}$ to ${\mathbb R}{\Gamma}_{Z}{\mathcal N}\xrightarrow{\cong} {\mathcal N}$ we obtain ${\mathbb R}{\Gamma}_{X}{j^{-1}\mathcal N}\xrightarrow\cong {j^{-1}\mathcal N}$.
There exist natural adjunction morphisms (cf. \cite[Lemma 4.3.1]{EK})
\begin{equation*}
j^{-1}{\mathbb R}j_{*}{\mathcal M}\to {\mathcal M} \textit{ and } {\mathcal N}\to {\mathbb R}j_{*}j^{-1}{\mathcal N}.
\end{equation*}
One has $j^{-1}{\mathbb R}j_{*}{\mathcal M}\xrightarrow{\cong} {\mathcal M}$ for any ${\mathcal M}\in {\mathcal C}_{U, X, \emptyset}$.
Let us prove that the adjunction morphism ${\mathcal N}\to {\mathbb R}j_{*}j^{-1}{\mathcal N}$ is an isomorphism for any  ${\mathcal N}\in {\mathcal C}_{P, Z, T}$.
One has a distinguished triangle
\begin{equation*}
{\mathbb R}{\Gamma}_{P\setminus U}{\mathcal N} \to {\mathcal N} \to {\mathbb R}j_{*}j^{-1}{\mathcal N} \xrightarrow{+1}.
\end{equation*}
We need to show that ${\mathbb R}{\Gamma}_{P\setminus U}{\mathcal N}$ is quasi-isomorphic to zero.
Let us consider a distinguished triangle
\begin{equation}\label{e1}
{\mathbb R}{\Gamma}_{Z}{\mathbb R}{\Gamma}_{P\setminus U}{\mathcal N}\to {\mathbb R}{\Gamma}_{P\setminus U}{\mathcal N} \to {\mathbb R}j_{Z*}j_{Z}^{-1}{\mathbb R}{\Gamma}_{P\setminus U}{\mathcal N} \xrightarrow{+1},
\end{equation}
where $j_{Z}$ denotes the open immersion $P\setminus Z\hookrightarrow P$.
One has ${\mathbb R}{\Gamma}_{Z}{\mathbb R}{\Gamma}_{P\setminus U}{\mathcal N}={\mathbb R}{\Gamma}_{T}{\mathcal N}=0$.
On the other hand, we obtain 
\begin{equation*}
{\mathbb R}j_{Z*}j_{Z}^{-1}{\mathbb R}{\Gamma}_{P\setminus U}{\mathcal N}={\mathbb R}{\Gamma}_{P\setminus U}{\mathbb R}j_{Z*}j_{Z}^{-1}{\mathcal N}=0.
\end{equation*}
So the assertion follows from (\ref{e1}).
\end{proof}

Recall that a $W_{n}$-embeddable $k$-scheme is a separated  $k$-scheme $X$ of finite type such that  there exists a  proper smooth $W_{n}$-scheme $P$ and an immersion $X\hookrightarrow P$ which fits in the following commutative diagram:
\[\xymatrix{
{X} \ar[d] \ar@{^{(}-_>}[r] & P \ar[d] \\
{\rm Spec} k \ar[r] & {\rm Spec} W_{n} .}
\]

\begin{defi}
Let $X$ be a  $W_{n}$-embeddable $k$-scheme with an immersion $X\hookrightarrow P$ into a proper smooth $W_{n}$-scheme $P$.
We define the category ${\mathcal C}_{P, X}$ to be ${\mathcal C}_{P, Z, T}$ for some closed subsets $Z$ and $T$ of $P$ with $X=Z\setminus T$.
This definition is well-defined by Lemma \ref{l1}.
\end{defi}

\begin{theo}\label{t1}
Let $f: P \to Q$ be a proper smooth morphism of smooth $W_{n}$-schemes.
Suppose that we are given closed immersions $i_{1}: X\hookrightarrow P$ and $i_{2}: X\hookrightarrow Q$ such that $f\circ i_{1}=i_{2}$.
Then $f_{+}$ induces an equivalence of categories 
\begin{equation}
f_{+}: {\mathcal C}_{P, X, \emptyset}\xrightarrow{\cong} {\mathcal C}_{Q, X, \emptyset}
\end{equation}
with a quasi-inverse ${\mathbb R}\Gamma_{X}\circ f^{!}$.
\end{theo}
\begin{proof}

Since the definition of the category ${\mathcal C}_{P, X, \emptyset}$ depends only on the underlying topological space of $X$ by Lemma \ref{l1}, we may assume that $X$ is reduced.
First of all, we note that the functors $f_{+}$ and ${\mathbb R}\Gamma_{X}\circ f^{!}$ are well-defined.
Indeed, for ${\mathcal M}\in {\mathcal C}_{P, X, \emptyset}$, by Proposition \ref{p1}, we have
\begin{eqnarray*}
{\mathbb R}\Gamma_{X}f_{+} {\mathcal M}&\xrightarrow{\cong}&f_{+} {\mathbb R}\Gamma_{f^{-1}(X)}{\mathcal M}\\
&\xleftarrow{\cong} & f_{+} {\mathbb R}\Gamma_{f^{-1}(X)}{\mathbb R}\Gamma_{X}{\mathcal M}\\
&\xrightarrow{\cong}& f_{+}{\mathbb R}\Gamma_{X}{\mathcal M}\\
&\xrightarrow{\cong}& f_{+}{\mathcal M}.
\end{eqnarray*}
We also have ${\mathbb R}\Gamma_{X}\left({\mathbb R}\Gamma_{X}f^{!}{\mathcal N}\right)\xrightarrow{\cong} {\mathbb R}\Gamma_{X}f^{!}{\mathcal N}$ for any ${\mathcal N}\in {\mathcal C}_{Q, X, \emptyset}$.
Next let us construct a natural transformation from $f_{+}$ to ${\mathbb R}\Gamma_{X}\circ f^{!}$ and its inverse.
By \cite[Corollary 14.5.15]{EK}, there are canonical adjunction morphisms
\begin{equation}
f_{+}f^{!}{\mathcal N}\to {\mathcal N} \textit{ and } {\mathcal M}\to f_{+}f^{!}{\mathcal M}.
\end{equation}
We thus obtain natural transformations of functors
\begin{equation}\label{e12}
f_{+}{\mathbb R}\Gamma_{X}f^{!}{\mathcal N}\to f_{+}f^{!}{\mathcal N}\to {\mathcal N}
\end{equation}
and
\begin{equation}\label{e4}
{\mathcal M}\xleftarrow{\cong} {\mathbb R}\Gamma_{X}{\mathcal M} \to {\mathbb R}\Gamma_{X}f^{!}f_{+}{\mathcal M}.
\end{equation}

Let us prove that these morphisms are isomorphisms by induction on $n$.
We begin with the case $n=1$.
Then $P$ and $Q$ are smooth $k$-schemes.
Let us firstly consider the case when $X$ is smooth over $k$.
Then  \cite[Corollary 15.5.4 and Proposition 15.5.3]{EK} imply that
\begin{equation*}
f_{+}{\mathbb R}\Gamma_{X}f^{!}{\mathcal N}\xrightarrow{\cong} f_{+}i_{1+}i_{1}^{!}f^{!}{\mathcal N} \xrightarrow{\cong}
i_{2+}i_{2}^{!}{\mathcal N} \xrightarrow{\cong} {\mathcal N}.
\end{equation*}
This shows that (\ref{e12}) is an isomorphism.
In order to see that (\ref{e4}) is an isomorphism, 
we claim that the natural morphism $i_{1}^{!}{\mathcal M}\to i_{1}^{!}f^{!}f_{+}{\mathcal M}$ is an isomorphism.
Indeed, since ${\mathcal M}\in {\mathcal C}_{P, X, \emptyset}$ is supported on $X$, there exists ${\mathcal M'}\in D^{b}_{\rm lfgu}({\mathcal D}_{F, X})$ such that $i_{1+}{\mathcal M}'\cong{\mathcal M}$ by \cite[Corollary 15.5.4]{EK}.
Then we have $i_{1}^{!}{\mathcal M}\cong i_{1}^{!}i_{1+}{\mathcal M}'\cong{\mathcal M}'$ and $i_{1}^{!}f^{!}f_{+}{\mathcal M}\cong i_{1}^{!}f^{!}f_{+}i_{1+}{\mathcal M}'\cong i_{2}^{!}i_{2+}{\mathcal M}'\cong {\mathcal M}'$, hence we see the claim.
Applying the functor $i_{1+}$ to the isomorphism $i_{1}^{!}{\mathcal M}\xrightarrow{\cong} i_{1}^{!}f^{!}f_{+}{\mathcal M}$ we see that $\Gamma_{X}{\mathcal M} \to {\mathbb R}\Gamma_{X}f^{!}f_{+}{\mathcal M}$ is an isomorphism by  \cite[Proposition 15.5.3]{EK}.

Next let us prove the case $n=1$ for general $X$ by induction on the dimension $d$ of $X$.
If $d=0$, then $X$ is \'etale over $k$ and the assertion follows from the smooth case.
Let $X_{0}$ be a $d$-dimensional smooth open subscheme of $X$ such that $H:=X\setminus X_{0}$ is of dimension $<d$.
Let us consider the following diagram:
\[\xymatrix{
{X\setminus H}\ar[r]^{} \ar@/_/[rd] &\ar@{}[rd]|{\square} {P\setminus f^{-1}(H)} \ar[d]_{f'} \ar[r]^-{j'} & P \ar[d]^{f} \\
& {Q\setminus H} \ar[r]^j & Q .}
\]
Let us consider the following morphism of distinguished triangles
{\small \[\xymatrix{
{{\mathbb R}\Gamma_{f^{-1}(H)}{\mathcal M}}\ar[r] \ar[d] & {\mathcal M} \ar[d] \ar[r] & {\mathbb R}j'_{*}j'^{-1}{\mathcal M} \ar[d]\ar[r]^-+& \\
{{\mathbb R}\Gamma_{X}f^{!}f_{+}{\mathbb R}\Gamma_{f^{-1}(H)}{\mathcal M}}\ar[r] & {\mathbb R}\Gamma_{X}f^{!}f_{+}{\mathcal M} \ar[r] &{\mathbb R}\Gamma_{X}f^{!}f_{+} {\mathbb R}j'_{*}j'^{-1}{\mathcal M} \ar[r]^-+&.}
\]}
In the left term, we have ${\mathbb R}\Gamma_{f^{-1}(H)}{\mathcal M}\xleftarrow{\cong} {\mathbb R}\Gamma_{f^{-1}(H)}{\mathbb R}\Gamma_{X}{\mathcal M}\xrightarrow{\cong} {\mathbb R}\Gamma_{H}{\mathcal M}$ and we 
also calculate
\begin{equation*}
{\mathbb R}\Gamma_{X}f^{!}f_{+}{\mathbb R}\Gamma_{f^{-1}(H)}{\mathcal M}\cong {\mathbb R}\Gamma_{X}f^{!}f_{+}{\mathbb R}\Gamma_{f^{-1}(H)}{\mathbb R}\Gamma_{H}{\mathcal M} \cong {\mathbb R}\Gamma_{H}f^{!}f_{+}{\mathbb R}\Gamma_{H}{\mathcal M}
\end{equation*}
by Proposition \ref{p1} and Proposition \ref{p4}.
Hence the induction hypothesis implies that the left vertical arrow is an isomorphism.
Similarly, by the smooth case, one can see that the right vertical arrow is an isomorphism.
As a consequence, we see that ${\mathcal M}\to{\mathbb R}\Gamma_{X}f^{!}f_{+}{\mathcal M}$ is an isomorphism.
Next let us consider the following morphism of distinguished triangles
\small{\[\xymatrix{
{f_{+}{\mathbb R}\Gamma_{X}f^{!}{\mathbb R}\Gamma_{H}{\mathcal N}}\ar[r] \ar[d] & f_{+}{\mathbb R}\Gamma_{X}f^{!}{\mathbb R}j_{*}j^{-1}{\mathcal N} \ar[d] \ar[r] & f_{+}{\mathbb R}\Gamma_{X}f^{!}{\mathbb R}j_{*}j^{-1}{\mathcal N}  \ar[d]\ar[r]^-+& \\
{{\mathbb R}\Gamma_{H}{\mathcal N}}\ar[r] & {\mathcal N} \ar[r] & {\mathbb R}j_{*}j^{-1}{\mathcal N}\ar[r]^-+&.}
\]}
In the left term, we have ${f_{+}{\mathbb R}\Gamma_{X}f^{!}{\mathbb R}\Gamma_{H}{\mathcal N}}\cong {f_{+}{\mathbb R}\Gamma_{H}f^{!}{\mathbb R}\Gamma_{H}{\mathcal N}}$ by Proposition \ref{p4}.
Hence the left vertical arrow is an isomorphism by the induction hypothesis.
In the right term, we can calculate as 
\begin{eqnarray*}
f_{+}{\mathbb R}\Gamma_{X}f^{!}{\mathbb R}j_{*}\left(j^{-1}{\mathcal N}\right)\cong 
f_{+}{\mathbb R}\Gamma_{X}{\mathbb R}j'_{*}f'^{!}\left(j^{-1}{\mathcal N}\right)&\cong& f_{+}{\mathbb R}j'_{*}{\mathbb R}\Gamma_{X\setminus H}f'^{!}\left(j^{-1}{\mathcal N}\right)\\
&\cong& {\mathbb R}j_{*}f'_{+}{\mathbb R}\Gamma_{X\setminus H}f'^{!}\left(j^{-1}{\mathcal N}\right)
\end{eqnarray*}
by Proposition \ref{p1}.
So the right vertical arrow is an isomorphism by the smooth case and hence the middle arrow is also an isomorphism.
This finishes the proof in the case $n=1$.
Now let us consider a distinguished triangle 
\begin{equation*}
{\mathcal M}\otimes_{{\mathbb Z}/{p^{n}}{\mathbb Z}}^{\mathbb L}{\mathbb Z}/p{\mathbb Z} \to {\mathcal M}\to {\mathcal M}\otimes_{{\mathbb Z}/{p^{n}}{\mathbb Z}}^{\mathbb L}{\mathbb Z}/p^{n-1}{\mathbb Z}\xrightarrow{+}.
\end{equation*}
Then the induction hypothesis, Lemma \ref{l3} and \cite[Proposition 14.8.1]{EK} reduce us to the case $n=1$ and we are done.

\end{proof}

\begin{coro}\label{c3}
Let $f: P \to Q$ be a proper smooth morphism of proper smooth $W_{n}$-schemes.
Suppose that we are given immersions $i_{1}: X\hookrightarrow P$ and $i_{2}: X\hookrightarrow Q$ such that $f\circ i_{1}=i_{2}$.
Then $f_{+}$ induces an equivalence of categories 
\begin{equation}
f_{+}: {\mathcal C}_{P, X}\xrightarrow{\cong} {\mathcal C}_{Q, X}
\end{equation}
with a quasi-inverse ${\mathbb R}\Gamma_{\bar{X}_{P}}\circ f^{!}$.
Here $\bar{X}_{P}$ denotes the closure of $X$ in $P$.
\end{coro}
\begin{proof}
Let us prove that $f_{+}$ restricts to a functor ${\mathcal C}_{P, X}\to {\mathcal C}_{Q, X}$.
Let $V$ be an open subset of $P$ such that $i_{2}$ factors as a closed immersion $X\hookrightarrow V$ and the open immersion $j_{2}: V\hookrightarrow Q$.
Denote by $U$ the open subset $f^{-1}(V)$ of $Q$.
Then $i_{1}$ factors as a closed immersion $X\hookrightarrow U$ and the open immersion $j_{1}: U\hookrightarrow P$.
For an object $\mathcal M$ in ${\mathcal C}_{P, X}$,
by Proposition \ref{p2}, there exists ${\mathcal M}'\in {\mathcal C}_{U, X, \emptyset}$ satisfying ${\mathbb R}j_{1*}{\mathcal M'}\cong {\mathcal M}$.
We have $f_{+}{\mathcal M}\cong f_{+}{\mathbb R}j_{1*}{\mathcal M'}\cong {\mathbb R}j_{2*}f_{|U+}{\mathcal M'}$.
In the course of the proof of Theorem \ref{t1}, we saw that $f_{|U+}{\mathcal M'}$ is in ${\mathcal C}_{V, X, \emptyset}$.
Hence we know that $f_{+}{\mathcal M}\cong{\mathbb R}j_{2*}f_{|U+}{\mathcal M'}$ is in ${\mathcal C}_{Q, X}$  by Proposition \ref{p2}.

Next let us prove that ${\mathbb R}\Gamma_{\bar{X}_{P}}\circ f^{!}$ restricts to a functor ${\mathcal C}_{Q, X}\to {\mathcal C}_{P, X}$.
Let $T_{Q}$ be a closed subset of $Q$ such that $\bar{X}_{Q}\setminus T_{Q}=X$ in $Q$, where $\bar{X}_{Q}$ denotes the closure of $X$ in $Q$.
We denote by $\bar{X}_{P}$ the the closure of $X$ in $P$.
Let $T$ be a closed subset of $P$ such that $\bar{X}_{P}\setminus T=X$
and we set $T_{P}:=T\cap f^{-1}(T_{Q})$.
Then $T_{P}$ is a closed subset of $P$ such that $\bar{X}_{P}\setminus T_{P}=X$ and 
we have ${\mathcal C}_{P, X}={\mathcal C}_{P, \bar{X}_{P}, T_{P}}$ and ${\mathcal C}_{Q, X}={\mathcal C}_{Q, \bar{X}_{Q}, T_{Q}}$.
For  an object $\mathcal M$  in ${\mathcal C}_{Q, \bar{X}_{Q}, T_{Q}}$,
one has ${\mathbb R}{\Gamma}_{\bar{X}_{P}}\left({\mathbb R}{\Gamma}_{\bar{X}_{P}}f^{!}{\mathcal M}\right)\xrightarrow{\cong}{\mathbb R}{\Gamma}_{\bar{X}_{P}}f^{!}{\mathcal M}$.
Also by assumption, one has ${\mathbb R}{\Gamma}_{T_{Q}}{\mathcal M}=0$.
Applying the functor ${\mathbb R}{\Gamma}_{\bar{X}_{P}}f^{!}$ to this equality we have $0={\mathbb R}{\Gamma}_{\bar{X}_{P}}f^{!}{\mathbb R}{\Gamma}_{T_{Q}}{\mathcal M}\cong{\mathbb R}{\Gamma}_{\bar{X}_{P}\cap f^{-1}(T_{Q})}f^{!}{\mathcal M}$.
Then we have ${\mathbb R}{\Gamma}_{T_{P}} {\mathbb R}{\Gamma}_{\bar{X}_{P}}f^{!}{\mathcal M}\cong {\mathbb R}{\Gamma}_{T_{P}}{\mathbb R}{\Gamma}_{\bar{X}_{P}\cap f^{-1}(T_{Q})}f^{!}{\mathcal M}=0$.

There are natural adjunction morphisms
\begin{equation*}
\Psi: f_{+}{\mathbb R}\Gamma_{\bar{X}_{P}}f^{!}{\mathcal N}\to {\mathcal N} \textit{ and } \Phi: {\mathcal M}\to {\mathbb R}\Gamma_{\bar{X}_{P}}f^{!}f_{+}{\mathcal M}.
\end{equation*}
By Proposition \ref{p2}, $\Psi$ is an isomorphism if and only if so is $\Psi_{|V}=j_{2}^{-1}\Psi$.
Now we can calculate as 
\begin{eqnarray*}
j_{2}^{-1}f_{+}{\mathbb R}\Gamma_{\bar{X}_{P}}f^{!}{\mathcal N}\cong f_{|U+}j_{1}^{-1}{\mathbb R}\Gamma_{\bar{X}_{P}}f^{!}{\mathcal N}&\cong& f_{|U+}{\mathbb R}\Gamma_{X}j_{1}^{-1}f^{!}{\mathcal N}\\
&\cong&  f_{|U+}{\mathbb R}\Gamma_{X}f_{|U}^{!}j_{2}^{-1}{\mathcal N}.
\end{eqnarray*}
Hence we see that $j_{2}^{-1}\Psi: f_{|U+}{\mathbb R}\Gamma_{X}f_{|U}^{!}j_{2}^{-1}{\mathcal N}\to j_{2}^{-1}{\mathcal N}$ is an isomorphism by Theorem \ref{t1}.
One can prove that $\Phi$ is an isomorphism in a similar manner.

\end{proof}

\begin{defi}
Let $X$ be a $W_{n}$-embeddable scheme. 
Let us take an immersion $X\hookrightarrow P$ into a proper smooth $W_{n}$-scheme.
We define the triangulated category $D^{b}_{\rm lfgu}(X/W_{n})^{\circ}$ by ${\mathcal C}_{P, X}$.
This definition is independent of the choice of embedding $X\hookrightarrow P$ up to natural equivalence by 
Corollary \ref{c3}.
\end{defi}

\subsection{Cohomological operations on $D_{\rm lfgu}^{b}(X/W_{n})^{\circ}$}\label{s42}

Let $f: X\to Y$ be a morphism of $W_{n}$-embeddable schemes.
Let us first define a functor
\begin{equation*}\label{d1}
f^{!}: D_{\rm lfgu}^{b}(Y/W_{n})^{\circ}\to D_{\rm lfgu}^{b}(X/W_{n})^{\circ}.
\end{equation*}

Let us take immersions  $i_1: X\hookrightarrow P$ and $i_2: Y\hookrightarrow Q$, 
where  $P$ and $Q$ are proper smooth $W_n$-schemes.
Then an immersion $i: X\hookrightarrow P\times_{W_{n}}Q$ is defined by the composition $X \xrightarrow{{\rm id}\times f} X\times_{k}Y \xhookrightarrow{i_1\times i_2} P\times_{W_{n}}Q$. 
One has $p_2 \circ i=i_2 \circ f$, where $p_2$ is the second projection $P\times_{W_{n}}Q\to Q$.
Hence, for a morphism $f: X\to Y$ of $W_{n}$-embeddable schemes, we can always obtain the following commutative diagram:
\begin{equation}\label{d2}
\vcenter{
\xymatrix{
{X} \ar[d]_{f} \ar@{^{(}-_>}[r]^{i_{1}} & {P}  \ar[d]^{g} \\
{Y} \ar@{^{(}-_>}[r]^{i_{2}} & Q. }
}
\end{equation}
Here $P$ and $Q$ are proper smooth $W_{n}$-schemes, $i_{1}$ and $i_{2}$ are immersions and $g$ is a proper smooth morphism of $W_{n}$-schemes.

\begin{lemm}\label{l8}
Let $f: X\to Y$ be a morphism of $W_{n}$-embeddable schemes.
Suppose that we are given the diagram (\ref{d2}).
Let $j: U\hookrightarrow P$ be an open immersion of smooth $W_{n}$-schemes such that $i_1$ factors as a closed immersion $X\hookrightarrow U$ and $j$.
Then the functor
\begin{equation*}
{\mathbb R}j_*{\mathbb R}{\Gamma}_{X}g_{|U}^{!}: {\mathcal C}_{Q, Y}\to {\mathcal C}_{P, X}
\end{equation*}
does not depend on the choice of the open immersion $j: U\hookrightarrow P$.
\end{lemm}
\begin{proof}
Assume that we are given an open subset $V$ of $U$ such that $X\hookrightarrow U$ factors as a closed immersion $X\hookrightarrow V$ and the open immersion $j': V\hookrightarrow U$.
Put $j'':=j'\circ j$.
Then one has 
$${\mathbb R}j''_*{\mathbb R}{\Gamma}_{X}g_{|V}^{!}\cong {\mathbb R}j_*{\mathbb R}j'_*{\mathbb R}{\Gamma}_{X}j'^{-1}g_{|U}^{!}\cong {\mathbb R}j_*{\mathbb R}j'_*j'^{-1}{\mathbb R}{\Gamma}_{X}g_{|U}^{!}\cong {\mathbb R}j_*{\mathbb R}{\Gamma}_{X}g_{|U}^{!}.$$
This completes the proof.
\end{proof}

Next let us suppose that we are given the following commutative diagram:
\begin{equation}\label{d4}
\vcenter{
\xymatrix{
{X} \ar[d]\ar@{^{(}-_>}[rd] \ar@{^{(}-_>}[rrd]\\
{Y} \ar@{^{(}-_>}[rd] \ar@{^{(}-_>}[rrd] & {P_{1}} \ar[r]^{h_{2}} \ar[d]_{g_{1}}& {P_{2}} \ar[d]^{g_{2}} \\ 
{} & {Q_{1}} \ar[r]_{h_{1}} & {Q_{2}}.
}
}
\end{equation}
Here $P_{1}$, $P_{2}$, $Q_{1}$ and  $Q_{2}$ are proper smooth $W_{n}$-schemes, and $g_{1}$, $g_{2}$, $h_{1}$ and $h_{2}$ are proper smooth morphisms over $W_{n}$, and all slanting allows are immersions.
Let us denote by $\bar{X}_{P_{1}}$ (resp. $\bar{Y}_{Q_{1}}$) the closure of $X$ (resp. $Y$) in $P_{1}$ (resp. $Q_{1}$).
Take an open immersion $j_2: U\hookrightarrow P_2$ such that the immersion $X\hookrightarrow P_2$ factors as a closed immersion $X\hookrightarrow U$ and $j_2$.
We set $V=h_2^{-1}(U)$ and $h'_2=h_{2|V}: V\to U$.
Denote by $j_1$ the open immersion $V\hookrightarrow P_1$.
Then we have the following functors:

\[\xymatrix{
{{\mathcal C}_{P_{1}, X}} & {{\mathcal C}_{P_{2}, X}}  \ar[l]_{{\mathbb R}{\Gamma}_{\bar{X}_{P_{1}}}\circ h_{2}^{!}} \\
{{\mathcal C}_{Q_{1}, Y}} \ar[u]^{{\mathbb R}j_{1*}{\mathbb R}{\Gamma}_{X}g_{1|V}^{!}} & {\mathcal C}_{Q_{2}, Y}\ar[l]^{{\mathbb R}{\Gamma}_{\bar{Y}_{Q_{1}}}\circ h_{1}^{!}}\ar[u]_{{\mathbb R}j_{2*}{\mathbb R}{\Gamma}_{X}\circ g_{2|U}^{!}}. }
\]
This diagram is commutative up to natural isomorphism since we have
\begin{align*}
&{\mathbb R}{\Gamma}_{\bar{X}_{P_{1}}}h_{2}^{!} \circ {\mathbb R}j_{2*}{\mathbb R}{\Gamma}_{X}g_{2|U}^{!}\cong
{\mathbb R}{\Gamma}_{\bar{X}_{P_{1}}} {\mathbb R}j_{1*}{h'_2}^!{\mathbb R}{\Gamma}_{X}g_{2|U}^{!}\\
&\,\,\,\,\,\,\,\,\,\,\cong{\mathbb R}j_{1*}{\mathbb R}{\Gamma}_{X} {\mathbb R}{\Gamma}_{{h'_2}^{-1}(X)}{h'_2}^!g_{2|U}^{!}\cong
{\mathbb R}j_{1*}{\mathbb R}{\Gamma}_{X}{h'_2}^!g_{2|U}^{!}
\end{align*}
(where  the first isomorphism follows from the flat base change theorem) and
\begin{equation*}
{\mathbb R}j_{1*}{\mathbb R}{\Gamma}_{X}g_{1|V}^{!} \circ{\mathbb R}{\Gamma}_{\bar{Y}_{Q_{1}}}h_{1}^{!}\cong
{\mathbb R}j_{1*}{\mathbb R}{\Gamma}_{X}{\Gamma}_{g_{1|V}^{-1}\left(\bar{Y}_{Q_{1}}\right)} g_{1|V}^{!} h_{1}^{!}\cong
{\mathbb R}j_{1*}{\mathbb R}{\Gamma}_{X} g_{1|V}^{!} h_{1}^{!}.
\end{equation*}

For a morphism $f: X\to Y$ of $W_{n}$-embeddable schemes, we take a diagram as in (\ref{d2}) and an open immersion $j: U\hookrightarrow P$ as in Lemma \ref{l8}.
We then define the inverse image functor 
\begin{equation*}
f^{!}: D^{b}_{\rm lfgu}(Y/W_{n})^{\circ}\to D^{b}_{\rm lfgu}(X/W_{n})^{\circ}
\end{equation*}
by $f^{!}={\mathbb R}j_*{\mathbb R}{\Gamma}_{X}g_{|U}^{!}$.
The above argument shows that this definition is independent of the choice of diagram (\ref{d2}) up to natural isomorphism.

Next let us define the direct image functor $f_{+}: D^{b}_{\rm lfgu}(X/W_{n})^{\circ}\to D^{b}_{\rm lfgu}(Y/W_{n})^{\circ}$. 
\begin{lemm}
Let $f: X\to Y$ be a morphism of $W_{n}$-embeddable schemes.
Suppose that we are given a diagram as in (\ref{d2}).
Then the functor $g_{+}$ restricts to a functor
\begin{equation*}
g_{+}: {\mathcal C}_{P, X}\to {\mathcal C}_{Q, Y}.
\end{equation*}
\end{lemm}
\begin{proof}
Take an open subset $V$ of $Q$ such that $i_{2}: Y\hookrightarrow Q$ factors as a closed immersion $Y\hookrightarrow V$ and the open immersion $j_{2}: V\hookrightarrow Q$.
There exists an open subset $U$ of $g^{-1}(V)$ such that the immersion $X\hookrightarrow g^{-1}(V)$ factors as a closed immersion $X\hookrightarrow U$ and the open immersion $U\hookrightarrow g^{-1}(V)$.
We denote by $j_{1}$ the open immersion $U\hookrightarrow P$.
Let $\mathcal M$ be an object in ${\mathcal C}_{P, X}$.
Then, by Proposition \ref{p2}, there exists ${\mathcal N}\in {\mathcal C}_{U, X, \emptyset}$ satisfying ${\mathbb R}j_{1*}{\mathcal N}\cong {\mathcal M}$.
We have $g_{+}{\mathcal M}\cong g_{+}{\mathbb R}j_{1*}{\mathcal N}\cong {\mathbb R}j_{2*}g_{|U+}{\mathcal N}$.
In the course of the proof of Theorem \ref{t1}, we saw that $g_{|U+}{\mathcal N}$ is in ${\mathcal C}_{V, Y, \emptyset}$.
By Proposition \ref{p2}, we know that $g_{+}{\mathcal M}\cong{\mathbb R}j_{2*}g_{|U+}{\mathcal N}$ is in ${\mathcal C}_{Q, Y}$.
\end{proof}
Let us assume that we are given a diagram as in (\ref{d4}).
Then we have a natural isomorphism of functors $h_{1+}\circ g_{1+}\cong h_{2+}\circ g_{2+}$.
For a morphism $f: X\to Y$ of $W_{n}$-embeddable schemes,
we take a diagram as in (\ref{d2}) and
define the direct image functor
\begin{equation*}
f_{+}: D^{b}_{\rm lfgu}(X/W_{n})^{\circ}\to D^{b}_{\rm lfgu}(Y/W_{n})^{\circ}
\end{equation*}
by $f_{+}:=g_{+}$.

Finally let us take an immersion $i: X\hookrightarrow P$ into a proper smooth $W_{n}$-scheme
and $Z$ and $T$ closed subsets $P$ such that $X=Z\setminus T$ as a set. 
For ${\mathcal M}$ and ${\mathcal N}\in {\mathcal C}_{P, Z, T}={\mathcal C}_{P, X}$ we consider
${\mathcal M}\otimes_{{\mathcal O}_{P}}^{\mathbb L}{\mathcal N}[-d_{P}]$ in $D^{b}_{\rm lfgu}({\mathcal D}_{F, P})$.
By Lemma \ref{l3}, we have ${\mathbb R}{\Gamma}_{Z}\left({\mathcal M}\otimes_{{\mathcal O}_{P}}^{\mathbb L}{\mathcal N}\right)\cong \left({\mathbb R}{\Gamma}_{Z}{\mathcal M}\right)\otimes_{{\mathcal O}_{P}}^{\mathbb L}{\mathcal N}\xrightarrow{\cong} {\mathcal M}\otimes_{{\mathcal O}_{P}}^{\mathbb L}{\mathcal N}$.
We also have ${\mathbb R}{\Gamma}_{T}\left({\mathcal M}\otimes_{{\mathcal O}_{P}}^{\mathbb L}{\mathcal N}\right)\cong \left({\mathbb R}{\Gamma}_{T}{\mathcal M}\right)\otimes_{{\mathcal O}_{P}}^{\mathbb L}{\mathcal N}=0$.
Hence ${\mathcal M}\otimes_{{\mathcal O}_{P}}^{\mathbb L}{\mathcal N}[-d_{P}]$ is an object in ${\mathcal C}_{P, Z, T}$.
Assume that we are given another immersion $i': X\hookrightarrow Q$ into a proper smooth $W_{n}$-scheme and a proper smooth $W_{n}$-morphism $f: P\to Q$ with $f\circ i=i'$.
There exists an equivalence ${\mathbb R}{\Gamma}_{\bar{X}}f^{!}: {\mathcal C}_{Q, X}\xrightarrow{\cong} {\mathcal C}_{P, X}$ by 
Theorem \ref{t1}, where $\bar{X}$ denotes the closure of $X$ in $P$.
For objects ${\mathcal M}$ and ${\mathcal N}$  in ${\mathcal C}_{Q, X}$,
applying the functor ${\mathbb R}{\Gamma}_{\bar{X}}f^{!}$ to ${\mathcal M}\otimes_{{\mathcal O}_{Q}}^{\mathbb L}{\mathcal N}[-d_{Q}]$,
we compute that 
\begin{equation*}
{\mathbb R}{\Gamma}_{\bar{X}}f^{!}\left({\mathcal M}\otimes_{{\mathcal O}_{Q}}^{\mathbb L}{\mathcal N}[-d_{Q}]\right)\cong 
{\mathbb R}{\Gamma}_{\bar{X}}\left(f^{!}{\mathcal M} \otimes_{{\mathcal O}_{P}}^{\mathbb L} f^{!}{\mathcal N}\right)[-d_{P}]
\end{equation*}
by Proposition \ref{p5}.
On the other hand, there exist isomorphisms 
\begin{eqnarray*}
{\mathbb R}{\Gamma}_{\bar{X}}f^{!}{\mathcal M} \otimes_{{\mathcal O}_{P}}^{\mathbb L} {\mathbb R}{\Gamma}_{\bar{X}}f^{!}{\mathcal N}[-d_{P}]&\cong& {\mathbb R}{\Gamma}_{\bar{X}}\left(f^{!}{\mathcal M} \otimes_{{\mathcal O}_{P}}^{\mathbb L} {\mathbb R}{\Gamma}_{\bar{X}}f^{!}{\mathcal N}\right)[-d_{P}]\\
&\cong& {\mathbb R}{\Gamma}_{\bar{X}}\left(f^{!}{\mathcal M} \otimes_{{\mathcal O}_{P}}^{\mathbb L} f^{!}{\mathcal N}\right)[-d_{P}]
\end{eqnarray*}
by Lemma \ref{l3}. Therefore we can define a bi-functor 
\begin{equation*}
(-)\otimes^{\mathbb L}(-):  D^{b}_{\rm lfgu}(X/W_{n})^{\circ}\times D^{b}_{\rm lfgu}(X/W_{n})^{\circ}\to D^{b}_{\rm lfgu}(X/W_{n})^{\circ}
\end{equation*}
to be ${\mathcal M}\otimes^{\mathbb L}{\mathcal N}:={\mathcal M}\otimes_{{\mathcal O}_{P}}^{\mathbb L}{\mathcal N}[-d_{P}]$
for some immersion $X\hookrightarrow P$ into a proper smooth $W_{n}$-scheme $P$.

\subsection{Riemann-Hilbert correspondence for unit $F$-crystals}\label{s43}

Let $X$ be a $W_{n}$-embeddable scheme with an immersion $i$ from $X$ into a proper smooth $W_{n}$-scheme $P$.
We define a functor ${\rm Sol}_{X}$ to be the composite of the functors
\begin{eqnarray*}
 D^{b}_{\rm lfgu}(X/W_{n})^{\circ}={\mathcal C}_{P, X}\subset D^{b}_{\rm lfgu}({\mathcal D}_{F, P})&\xrightarrow{{\rm Sol}_{P}}& D_{\rm ctf}^{b}(P_{\rm \acute{e}t}, {\mathbb Z}/{p^{n}\mathbb Z})\\
&\xrightarrow{i^{-1}}& D_{\rm ctf}^{b}(X_{\rm \acute{e}t}, {\mathbb Z}/{p^{n}\mathbb Z}),
\end{eqnarray*}
where the first functor is the natural embedding.
\begin{lemm}\label{l19}
This definition is independent of the choice of embedding $i: X\hookrightarrow P$ up to natural isomorphism.
\end{lemm}
\begin{proof}
Let us first suppose that we are given an open immersion $j: U\hookrightarrow P$ such that $i$ factors as an  closed immersion $i': X\hookrightarrow U$ and $j$.
Then $j^{-1}$ induces an equivalence ${\mathcal C}_{U, X, \emptyset}\xrightarrow{\cong} {\mathcal C}_{P, X}$ by Proposition \ref{p2}.
Let us consider a functor $i'^{-1}\circ {\rm Sol}_{U}: {\mathcal C}_{U, X, \emptyset}\to D_{\rm ctf}^{b}(X_{\rm \acute{e}t}, {\mathbb Z}/{p^{n}\mathbb Z})$.
Then one has 
\begin{equation}\label{e33}
i'^{-1}{\rm Sol}_{U}j^{!}{\mathcal M}\cong i'^{-1}j^{-1}{\rm Sol}_{P}{\mathcal M}\cong i'^{-1}{\rm Sol}_{P}{\mathcal M}
\end{equation}
 for any ${\mathcal M}\in {\mathcal C}_{P, X}$. 
Next let us suppose that we are given a closed immersion $i'': X\hookrightarrow Q$ into a smooth $W_{n}$-scheme $Q$ and a proper smooth $W_{n}$-morphism $U\to Q$ with $f\circ i'=i''$.
Then $f_{+}$ induces an equivalence ${\mathcal C}_{U, X, \emptyset}\xrightarrow{\cong} {\mathcal C}_{Q, X, \emptyset}$ by Theorem \ref{t1}.
Note that, because ${\rm Sol}_{U}$ is compatible with the inverse image functor by Theorem \ref{t4},
we know that ${\rm Sol}_{U}{\mathcal M}$ is supported on $X$ for any ${\mathcal M}\in {\mathcal C}_{U, X, \emptyset}$.
Then, for ${\mathcal M}\in {\mathcal C}_{U, X, \emptyset}$, we can compute that
\begin{equation*}
i''^{-1}{\rm Sol}_{Q}f_{+}{\mathcal M}\cong i''^{-1}f_{!}{\rm Sol}_{U}{\mathcal M}\cong i''^{-1}f_{!}i'_{!}i'^{-1}{\rm Sol}_{U}{\mathcal M}\cong i'^{-1}{\rm Sol}_{U}{\mathcal M}.
\end{equation*}
We can prove the lemma by combining two claims proved above.
\end{proof}

Next let us define a functor ${\rm M}_{X}: D_{\rm ctf}^{b}(X_{\rm \acute{e}t}, {\mathbb Z}/{p^{n}\mathbb Z})\to  D^{b}_{\rm lfgu}(X/W_{n})^{\circ}$.
We define ${\rm M}_{X}$ to be the composite of the functors
\begin{equation*}
D_{\rm ctf}^{b}(X_{\rm \acute{e}t}, {\mathbb Z}/{p^{n}\mathbb Z})\xrightarrow{i_{*}} D_{\rm ctf}^{b}(P_{\rm \acute{e}t}, {\mathbb Z}/{p^{n}\mathbb Z}) \xrightarrow{{\rm M}_{P}}  D^{b}_{\rm lfgu}({\mathcal D}_{F, P})^{\circ}.
\end{equation*}

\begin{lemm}\label{l0}
The essential image of $M_{X}$ is contained in ${\mathcal C}_{P, X}$.
\end{lemm}
\begin{proof}
Let us take an open subscheme $U$ of $P$ such that $i$ factors as a closed immersion $i': X\hookrightarrow U$ and 
an open immersion $j: U\hookrightarrow P$.
Then by \cite[Corollary 16.2.8]{EK} $M_{X}$ is naturally isomorphic to the composition
\begin{eqnarray*}
D_{\rm ctf}^{b}(X_{\rm \acute{e}t}, {\mathbb Z}/{p^{n}\mathbb Z})\xrightarrow{i'_{*}} D_{\rm ctf}^{b}(U_{\rm \acute{e}t}, {\mathbb Z}/{p^{n}\mathbb Z}) &\xrightarrow{{\rm M}_{U}}&  D^{b}_{\rm lfgu}({\mathcal D}_{F, U})^{\circ}\\
&\xrightarrow{{\mathbb R}j_{*}}& D^{b}_{\rm lfgu}({\mathcal D}_{F, P})^{\circ}.
\end{eqnarray*}
So we reduce to the case when $X$ is closed in $P$ by Proposition \ref{p2}.
Now since ${\rm M}_{P}$ is compatible with the inverse image functor by Theorem \ref{t4}, ${\mathcal F}\in  D_{\rm ctf}^{b}(P_{\rm \acute{e}t}, {\mathbb Z}/{p^{n}\mathbb Z})$ is supported on $X$ if and only if so is ${\rm M}_{P}({\mathcal F})$.
\end{proof}
One can prove that this functor is independent of the choice of $X\hookrightarrow P$ as in Lemma \ref{l19}.
By Lemma \ref{l0}, we obtain a functor
\begin{equation*}
{\rm M}_{X}: D_{\rm ctf}^{b}(X_{\rm \acute{e}t}, {\mathbb Z}/{p^{n}\mathbb Z})\to  D^{b}_{\rm lfgu}(X/W_{n})^{\circ}.
\end{equation*}
We now state our main result. 
\begin{theo}\label{t5}
Let $X$ be a $W_{n}$-embeddable $k$-scheme.
Then ${\rm Sol}_{X}$ induces an equivalence of triangulated categories
\begin{equation}
{\rm Sol}_{X}: D_{\rm lfgu}^{b}(X/W_{n})^{\circ}\xrightarrow{\cong} D_{\rm ctf}^{b}(X_{\rm \acute{e}t}, {\mathbb Z}/{p^{n}\mathbb Z})
\end{equation}
with quasi-inverse ${\rm M}_{X}$.
\end{theo}

In order to prove Theorem \ref{t5} we need the following lemma.
\begin{lemm}\label{l10}
Let $X$ be a $W_{n}$-embeddable $k$-scheme with a closed immersion $i$ from $X$ into a smooth $W_{n}$-scheme $P$.
Let us denote by $D_{{\rm ctf}, X}^{b}(P_{\rm \acute{e}t}, {\mathbb Z}/{p^{n}\mathbb Z})$
the full triangulated subcategory of $D_{{\rm ctf}}^{b}(P_{\rm \acute{e}t}, {\mathbb Z}/{p^{n}\mathbb Z})$ consisting of  complexes supported on $X$.
Then ${\rm Sol}_{P}: D_{\rm lfgu}^{b}({\mathcal D}_{F, P})^{\circ}\to D_{{\rm ctf}}^{b}(P_{\rm \acute{e}t}, {\mathbb Z}/{p^{n}\mathbb Z})$
restricts an equivalence
\begin{equation*}
{\mathcal C}_{P, X, \emptyset}\xrightarrow{\cong} D_{{\rm ctf}, X}^{b}(P_{\rm \acute{e}t}, {\mathbb Z}/{p^{n}\mathbb Z}).
\end{equation*}
\end{lemm}
\begin{proof}
By Lemma \ref{l5}, we can write as
\begin{equation*}
{\mathcal C}_{P, X, \emptyset}=\left\{{\mathcal M}\in D_{\rm lfgu}^{b}({\mathcal D}_{F, P})^{\circ}\,|\, {\rm Supp}{\mathcal M}\subset X\right\}.
\end{equation*}
Denote by $j$ the open immersion $P\setminus X\hookrightarrow P$.
For an object $\mathcal M$ in ${\mathcal C}_{P, X, \emptyset}$,
the condition ${\rm Supp}{\mathcal M}\subset X$ is equivalent to the condition $j^{-1}{\mathcal M}=0$.
Applying the functor ${\rm Sol}_{P}$ to $j^{-1}{\mathcal M}=0$, by Theorem \ref{t4},  one has $j^{-1}\left({\rm Sol}_{P}{\mathcal M}\right)\cong {\rm Sol}_{P\setminus X}j^{-1}{\mathcal M}=0$.
Hence we know that ${\rm Sol}_{P}$ restricts to a functor 
\begin{equation*}
{\rm Sol}_{P}: {\mathcal C}_{P, X, \emptyset}\to D_{{\rm ctf}, X}^{b}(P_{\rm \acute{e}t}, {\mathbb Z}/{p^{n}\mathbb Z}).
\end{equation*} 
Similarly, ${\rm M}_{P}$ restricts to a functor ${\rm M}_{P}: D_{{\rm ctf}, X}^{b}(P_{\rm \acute{e}t}, {\mathbb Z}/{p^{n}\mathbb Z})\to {\mathcal C}_{P, X, \emptyset}$ and we are done.
\end{proof}
Let us prove Theorem \ref{t5}.

\begin{proof}
We may assume that there exists a proper smooth $W_{n}$-scheme $P$, an open subset $U$ of $P$ together with a closed immersion $i: X\hookrightarrow U$.
Then ${\rm Sol}_{P}$ is compatible with ${\rm Sol}_{U}$ and
${\rm Sol}_{U}$ induces an equivalence of triangulated categories
\begin{equation*}
{\mathcal C}_{U, X, \emptyset}\xrightarrow{\cong} D_{{\rm ctf}, X}^{b}(U_{\rm \acute{e}t}, {\mathbb Z}/{p^{n}\mathbb Z})
\end{equation*}
with quasi-inverse ${\rm M}_{U}$ by Lemma \ref{l10}.
Also, $i^{-1}: D_{{\rm ctf}, X}^{b}(P_{\rm \acute{e}t}, {\mathbb Z}/{p^{n}\mathbb Z})\to D_{{\rm ctf}}^{b}(X_{\rm \acute{e}t}, {\mathbb Z}/{p^{n}\mathbb Z})$ is an equivalence of triangulated categories 
with quasi-inverse $i_{*}$.
This finishes the proof.
\end{proof}

\begin{theo}\label{t6}
Let $f: X\to Y$ be a morphism of $W_{n}$-embeddable schemes.
Then there exist natural isomorphisms of functors
\begin{equation*}
{\rm Sol}_{Y}\circ f_{+}\cong f_{!}\circ {\rm Sol}_{X}: D^b_{\rm lfgu}(X/W_n)^{\circ}\to D_{{\rm ctf}}^{b}(Y_{\rm \acute{e}t}, {\mathbb Z}/{p^{n}\mathbb Z}),
\end{equation*}
\begin{equation*}
f^{-1}\circ {\rm Sol}_{Y}\xrightarrow{\cong} {\rm Sol}_{X}\circ f^{!}: D^b_{\rm lfgu}(Y/W_n)^{\circ}\to D_{{\rm ctf}}^{b}(X_{\rm \acute{e}t}, {\mathbb Z}/{p^{n}\mathbb Z}) 
\end{equation*}
and a functorial isomorphism
\begin{equation*}
{\rm Sol}_{X}({\mathcal M})\otimes_{{\mathbb Z}/p^{n}{\mathbb Z}}^{\mathbb L}{\rm Sol}_{X}({\mathcal N})\xrightarrow{\cong}{\rm Sol}_{X}\left({\mathcal M}\otimes^{\mathbb L}{\mathcal N}\right)
\end{equation*}
for objects $\mathcal M$ and $\mathcal N$ in $ D^b_{\rm lfgu}(X/W_n)^{\circ}$.
\end{theo}
\begin{proof}

We may assume that there exists a commutative diagram
\[\xymatrix{
{X} \ar[d]_{f} \ar@{^{(}-_>}[r]^{i_{1}} & {P}  \ar[d]^{g} \\
{Y} \ar@{^{(}-_>}[r]^{i_{2}} & Q }
\]
such that  $P$  is a smooth $W_{n}$-scheme which is an open subscheme of a proper smooth $W_n$-scheme $\tilde{P}$, 
$Q$ is a proper smooth $W_{n}$-scheme,
$i_{1}$ is a closed immersion, $i_{2}$ is an immersion,
and $g$ is the composite of an immersion and a proper smooth morphism.
Also, we can identify the categories $D^b_{\rm lfgu}(X/W_n)$ and $D^b_{\rm lfgu}(Y/W_n)$ in the statement with the categories ${\mathcal C}_{P, X, \emptyset}={\mathcal C}_{\tilde{P}, X}$ and ${\mathcal C}_{Q, Y}$ respectively.
Via this identification, ${\rm Sol}_X$ is identified with $i_1^{-1}\circ{\rm Sol}_{P}$ on ${\mathcal C}_{P, X, \emptyset}$  by (\ref{e33}).
For any object $\mathcal M$ in ${\mathcal C}_{P, X, \emptyset}$,
${\rm Sol}_{P}({\mathcal M})$ is supported on $X$.
So we can compute that
\begin{eqnarray*}
{\rm Sol}_{Y}\circ f_{+}:=i_{2}^{-1}{\rm Sol}_{Q} g_{+}\cong i_{2}^{-1}g_{!} {\rm Sol}_{P}\cong i_{2}^{-1}g_{!}i_{1*}i_{1}^{-1}{\rm Sol}_{P}\cong f_{!}\circ {\rm Sol}_{X}.
\end{eqnarray*}

Let us prove the second isomorphism.
Recall that $f^{!}:={\mathbb R}{\Gamma}_{X}g^{!}: {\mathcal C}_{Q, Y}\to {\mathcal C}_{P, X, \emptyset}$.
We define a natural transformation $f^{-1}\circ {\rm Sol}_{Y}\to {\rm Sol}_{X}\circ f^{!}$ to be the
 composite of natural transformations
\begin{equation*}
f^{-1}\circ {\rm Sol}_{Y} \cong i_{1}^{-1} g^{-1} {\rm Sol}_{Q} \cong i_{1}^{-1} {\rm Sol}_{P}g^{!}\to i_{1}^{-1} {\rm Sol}_{P}{\mathbb R}{\Gamma}_{X}g^{!}= {\rm Sol}_{X}\circ f^{!}.
\end{equation*}
Let us prove that it is an isomorphism.
The usual d\'evissage argument reduces the proof to the case $n=1$. 
First of all, suppose that $X$ is smooth over $k$.
Then using \cite[Corollary 15.5.4]{EK} and Theorem \ref{t4}, we obtain isomorphisms
\begin{eqnarray*}
f^{-1}\circ {\rm Sol}_{Y} =f^{-1} i_{2}^{-1}{\rm Sol}_{Q}
&\cong& i_{1}^{-1} i_{1*} i_{1}^{-1} g^{-1} {\rm Sol}_{Q}
\cong i_{1}^{-1} {\rm Sol}_{P} \circ \left(i_{1+}i_{1}^{!}g^{!}\right)\\
&\cong&i_{1}^{-1} {\rm Sol}_{P} \circ \left({\mathbb R}{\Gamma}_{X}g^{!}\right) \cong {\rm Sol}_{X}\circ f^{!}.
\end{eqnarray*}
In general case, we shall prove by the induction on the dimension $d$ of $X$.
Let $X_{0}$ be a $d$-dimensional smooth open subscheme of $X$ such that $H:=X\setminus X_{0}$ is of dimension $<d$. 
Let $a$ denote the open immersion $P\setminus H\hookrightarrow P$ and  ${\mathcal M}$ an object in ${\mathcal C}_{Q, Y}$.
We have a distinguished triangle in $D_{{\rm c}}^{b}(P_{\rm \acute{e}t}, {\mathbb Z}/{p\mathbb Z})$
\begin{equation}\label{e13}
{\rm Sol}_{P}{\mathbb R}{\Gamma}_{X}a_{+}a^{!}g^{!}{\mathcal M}\to {\rm Sol}_{P}{\mathbb R}{\Gamma}_{X}g^{!}{\mathcal M}
\to {\rm Sol}_{P}{\mathbb R}{\Gamma}_{H}g^{!} {\mathcal M}\xrightarrow{+}.
\end{equation}
Let us denote by $i_{H}$ the closed immersion $H\hookrightarrow P$.
For any object ${\mathcal F}$ in $D_{{\rm c}}^{b}(P_{\rm \acute{e}t}, {\mathbb Z}/{p\mathbb Z})$, there exists a distinguished triangle in $D_{{\rm c}}^{b}(P_{\rm \acute{e}t}, {\mathbb Z}/{p\mathbb Z})$
\begin{equation}\label{e2}
a_{!}a^{-1}{\mathcal F}\to {\mathcal F}\to i_{H!}i_{H}^{-1}{\mathcal F}\xrightarrow{+}.
\end{equation}
Applying (\ref{e2}) to ${\mathcal F}=g^{-1}{\rm Sol}_{Q}{\mathcal M}$, one has a distinguished triangle
\begin{equation}\label{e3}
a_{!}a^{-1}{g^{-1}{\rm Sol}_{Q}{\mathcal M}}\to {g^{-1}{\rm Sol}_{Q}{\mathcal M}}\to i_{H!}i_{H}^{-1}{g^{-1}{\rm Sol}_{Q}{\mathcal M}}\xrightarrow{+}.
\end{equation}
There are natural morphisms
\begin{equation*}
\psi: a_{!}a^{-1}g^{-1}{\rm Sol}_{Q}{\mathcal M}\cong  {\rm Sol}_{P}a_{+}a^{!}g^{!}{\mathcal M}\to {\rm Sol}_{P}{\mathbb R}{\Gamma}_{X}a_{+}a^{!}g^{!}{\mathcal M}
\end{equation*}
and
\begin{eqnarray*}
\phi: i_{H!}i_{H}^{-1}g^{-1}{\rm Sol}_{Q}{\mathcal M} \xrightarrow{\cong} i_{H!}i_{H}^{-1}{\rm Sol}_{P}g^!{\mathcal M}
&\to&  i_{H!}i_{H}^{-1}{\rm Sol}_{P}{\mathbb R}{\Gamma}_{H}g^{!} {\mathcal M}\\
&\cong& {\rm Sol}_{P}{\mathbb R}{\Gamma}_{H}g^{!} {\mathcal M}.
\end{eqnarray*}
Here the last isomorphism follows since ${\rm Sol}_{P}{\mathbb R}{\Gamma}_{H}g^{!} {\mathcal M}$ is supported on $H$.
Hence we obtain a morphism of distinguished triangle from  (\ref{e3}) to (\ref{e13}).
We then claim that $i_1^{-1}\psi$ is an isomorphism.
We can calculate as
\begin{eqnarray*}
{\rm Sol}_{P}{\mathbb R}{\Gamma}_{X}a_{+}a^{!}g^{!}{\mathcal M}&\cong& {\rm Sol}_{P}a_{+}{\mathbb R}{\Gamma}_{X\setminus H}a^{!}g^{!}{\mathcal M}\\
&\cong& a_{!} {\rm Sol}_{P\setminus H}{\mathbb R}{\Gamma}_{X\setminus H}a^{!}g^{!}{\mathcal M}.
\end{eqnarray*}
So $i_1^{-1}a_{!}a^{-1}{g^{-1}{\rm Sol}_{Q}{\mathcal M}}$ and $i_1^{-1}{\rm Sol}_{P}{\mathbb R}{\Gamma}_{X}a_{+}a^{!}g^{!}{\mathcal M}$ are supported on $X\setminus H$.
Hence $i_1^{-1}\psi$ is an isomorphism if and only if so is $a'^{-1}i_1^{-1}\psi$,
where $a'$ denotes the open immersion $X\setminus H\hookrightarrow X$. 
We denote by $i'_1$ the closed immersion $X\setminus H\hookrightarrow P\setminus H$.
Applying the functor $a'^{-1}i_1^{-1}$ to the isomorphism ${\rm Sol}_{P}{\mathbb R}{\Gamma}_{X}a_{+}a^{!}g^{!}{\mathcal M}\cong a_{!} {\rm Sol}_{P\setminus H}{\mathbb R}{\Gamma}_{X\setminus H}a^{!}g^{!}{\mathcal M}$,
we have
\begin{eqnarray*}
a'^{-1}i_1^{-1}{\rm Sol}_{P}{\mathbb R}{\Gamma}_{X}a_{+}a^{!}g^{!}{\mathcal M} &\cong&
a'^{-1}i_1^{-1} a_{!} {\rm Sol}_{P\setminus H}{\mathbb R}{\Gamma}_{X\setminus H}a^{!}g^{!}{\mathcal M}\\
&\cong& {i'}_1^{-1} a^{-1}a_{!} {\rm Sol}_{P\setminus H}f_{|X\setminus H}^{!}{\mathcal M}\\
&\cong& {i'}_1^{-1} {\rm Sol}_{P\setminus H}f_{|X\setminus H}^{!}{\mathcal M}
\cong {\rm Sol}_{X\setminus H}f_{|X\setminus H}^{!}{\mathcal M}.
\end{eqnarray*}
We can also calculate as 
\begin{eqnarray*}
a'^{-1}i_1^{-1}a_{!}a^{-1}{g^{-1}{\rm Sol}_{Q}{\mathcal M}}&\cong& {i'}_1^{-1} a^{-1}a_{!}a^{-1}{g^{-1}{\rm Sol}_{Q}{\mathcal M}}\\
&\cong& {i'}_1^{-1}a^{-1}{g^{-1}{\rm Sol}_{Q}{\mathcal M}}\\
&\cong& a'^{-1} f^{-1} {i_2}^{-1} {\rm Sol}_{Q}{\mathcal M} \cong f_{|X\setminus H}^{-1}{\rm Sol}_{Y}{\mathcal M}.
\end{eqnarray*}
Hence $a'^{-1}i_1^{-1}\psi$ is identified with the morphism 
$$f_{|X\setminus H}^{-1}{\rm Sol}_{Y}{\mathcal M}\to {\rm Sol}_{X\setminus H}f_{|X\setminus H}^{!}{\mathcal M}$$
and it is an isomorphism by the smooth case.
On the other hand, 
since ${\rm Sol}_{P}{\mathbb R}{\Gamma}_{H}g^{!} {\mathcal M}$ and $ i_{H!}i_{H}^{-1}g^{-1}{\rm Sol}_{Q}{\mathcal M}$ are supported on $H$, $\phi$ is an isomorphism if and only if so is $i_{H}^{-1}\phi$.
Let us denote by $f_{|H}$ the composite of morphisms $H\hookrightarrow X\xrightarrow{f} Y$.
Applying $i_{H}^{-1}$ to ${\rm Sol}_{P}{\mathbb R}{\Gamma}_{H}g^{!} {\mathcal M}$ and $i_{H!}i_{H}^{-1}g^{-1}{\rm Sol}_{Q}{\mathcal M}$, 
one has 
\begin{equation*}
i_{H}^{-1}{\rm Sol}_{P}{\mathbb R}{\Gamma}_{H}g^{!} {\mathcal M}\cong {\rm Sol}_{H}f_{|H}^{!} {\mathcal M} \textit{ and }  i_{H}^{-1}i_{H!}i_{H}^{-1}g^{-1}{\rm Sol}_{Q}{\mathcal M}\cong f_{|H}^{-1}{\rm Sol}_{H}{\mathcal M}
\end{equation*}
 respectively.
So  $i_{H}^{-1}\phi$ is identified with the morphism $f_{|H}^{-1}{\rm Sol}_{H}{\mathcal M}\to {\rm Sol}_{H}f_{|H}^{!} {\mathcal M}$ and it is an isomorphism by the induction hypothesis.
Now we know that the morphism of distinguished triangle from (\ref{e3}) to (\ref{e13}) is an isomorphism after we apply the functor $i_{1}^{-1}$ to it.
As a consequence, we obtain the desired isomorphism $f^{-1}{\rm Sol}_{Y}\xrightarrow{\cong} {\rm Sol}_{X}f^{!}$.

Finally let us prove  the last isomorphism.
For objects $\mathcal M$ and $\mathcal N$ of $D^{b}_{\rm lfgu}(X/W_{n})={\mathcal C}_{P, X, \emptyset}$,
there exists a natural isomorphism 
\begin{equation*}
{\rm Sol}_{P}{\mathcal M}\otimes^{\mathbb L}_{{\mathbb Z}/p^{n}{\mathbb Z}}{\rm Sol}_{P}{\mathcal N}\xrightarrow{\cong}{\rm Sol}_{P}\left({\mathcal M}\otimes^{\mathbb L}_{{\mathcal O}_{P}}{\mathcal N}\right)[d_{P}]
\end{equation*}
by Theorem \ref{t4}.
Recall that ${\mathcal M}\otimes^{\mathbb L}{\mathcal N}$ is defined to be ${\mathcal M}\otimes^{\mathbb L}_{{\mathcal O}_{P}}{\mathcal N}[-d_{P}]$.
Applying the functor ${\rm Sol}_{X}:=i^{-1}{\rm Sol}_{P}$ to it, one has
\begin{eqnarray*}
{\rm Sol}_{X}\left({\mathcal M}\otimes^{\mathbb L}{\mathcal N}\right)&\cong& i^{-1}{\rm Sol}_{P}\left({\mathcal M}\otimes^{\mathbb L}_{{\mathcal O}_{P}}{\mathcal N}[-d_{P}]\right)\\
&\cong&  i^{-1}{\rm Sol}_{P}\left({\mathcal M}\otimes^{\mathbb L}_{{\mathcal O}_{P}}{\mathcal N}\right)[d_{P}]\\
&\cong& i^{-1}\left({\rm Sol}_{P}{\mathcal M}\otimes^{\mathbb L}_{{\mathbb Z}/p^{n}{\mathbb Z}}{\rm Sol}_{P}{\mathcal N}\right)\\
&\cong& {\rm Sol}_{X}({\mathcal M})\otimes_{{\mathbb Z}/p^{n}{\mathbb Z}}^{\mathbb L}{\rm Sol}_{X}({\mathcal N}).
\end{eqnarray*}
This finishes the proof.
\end{proof}

\section{$t$-structures on  $D_{\rm lfgu}^{b}(X/k)$}
In this section, we study several $t$-structures on $D_{\rm lfgu}^{b}(X/k)$ for a $k$-embeddable $k$-scheme $X$.
Note that, for a smooth $k$-scheme $P$, one has $D^{b}_{\rm lfgu}({\mathcal D}_{F, P})^{\circ}=D^{b}_{\rm lfgu}({\mathcal D}_{F, P})$ and $D^{b}_{\rm lfgu}({\mathcal D}_{F, P})^{\circ}$ is naturally equivalent to $D^{b}_{\rm lfgu}({\mathcal O}_{F, P})$ (see the subsection \ref{s2}).

\subsection{The standard $t$-structure on $D_{\rm lfgu}^{b}(X/k)$}\label{s51}

For a smooth $k$-scheme $P$, we set
\begin{eqnarray*}
D^{\leq n}_{\rm lfgu}({\mathcal D}_{F, P})&=&\left\{{\mathcal M}\in D^{b}_{\rm lfgu}({\mathcal D}_{F, P}) \,|\, {\rm H}^{k}({\mathcal M})=0 \textit{ for } k>n\right\} \textit{ and}\\
D^{\geq n}_{\rm lfgu}({\mathcal D}_{F, P})&=&\left\{{\mathcal M}\in D^{b}_{\rm lfgu}({\mathcal D}_{F, P}) \,|\, {\rm H}^{k}({\mathcal M})=0 \textit{ for } k<n\right\} .
\end{eqnarray*}

Let $X$ be a $k$-scheme of finite type.
The middle perversity is the function ${\rm p}: X\to {\mathbb Z}$ defined by
\begin{equation*}
{\rm p}(x)=-{\rm dim}\overline{\{x\}}.
\end{equation*}
For $x\in X$, we denote by $i_{x}$ the canonical inclusion $\{x\}\hookrightarrow X$.
We then define a full subcategory ${}^{\rm p}D_{\rm c}^{\leq 0}(X_{\rm \acute{e}t}, {\mathbb Z}/{p\mathbb Z})$ (resp. ${}^{\rm p}D_{\rm c}^{\geq 0}(X_{\rm \acute{e}t}, {\mathbb Z}/{p\mathbb Z})$) of $D_{\rm c}^{b}(X_{\rm \acute{e}t}, {\mathbb Z}/{p\mathbb Z})$ by the condition: 
 ${\mathcal F}$ is in ${}^{\rm p}D_{\rm c}^{\leq 0}(X_{\rm \acute{e}t}, {\mathbb Z}/{p\mathbb Z})$ if and only if 
${\rm H}^{k}(i_{x}^{-1}{\mathcal F})=0$ for any $x\in X$ and $k>{\rm p}(x)$ (resp. ${\mathcal F}$ is in ${}^{\rm p}D_{\rm c}^{\geq 0}(X_{\rm \acute{e}t}, {\mathbb Z}/{p\mathbb Z})$ if and only if ${\rm H}^{k}(i_{x}^{!}{\mathcal F})=0$ for any $x\in X$ and $k<{\rm p}(x)$).

Gabber proved that $\left({}^{\rm p}D_{\rm c}^{\leq 0}(X_{\rm \acute{e}t}, {\mathbb Z}/{p\mathbb Z}), {}^{\rm p}D_{\rm c}^{\geq 0}(X_{\rm \acute{e}t}, {\mathbb Z}/{p\mathbb Z})\right)$ forms a $t$-structure on $D_{\rm c}^{b}(X_{\rm \acute{e}t}, {\mathbb Z}/{p\mathbb Z})$, which we call Gabber's perverse $t$-structure, in \cite[Theorem 10.3]{Ga}.
Emerton and Kisin gave another proof of it in the case when $X$ is smooth over $k$ based on the Riemann-Hilbert correspondence \cite[Theorem 11.5.4]{EK}:
Indeed, they proved that $D_{\rm lfgu}^{\leq 0}(X/k)$ (resp. $D_{\rm lfgu}^{\geq 0}(X/k)$) is equivalent to ${}^{\rm p}D_{\rm c}^{\geq 0}(X_{\rm \acute{e}t}, {\mathbb Z}/{p\mathbb Z})$ (resp. ${}^{\rm p}D_{\rm c}^{\leq 0}(X_{\rm \acute{e}t}, {\mathbb Z}/{p\mathbb Z})$) via ${\rm Sol}_{X}$. 
We generalize  \cite[Theorem 11.5.4]{EK} to the case of $k$-embeddable $k$-schemes.
\begin{defi}
Let $P$ be a smooth $k$-scheme with a closed subset $X$ of $P$.
We set
\begin{eqnarray*}
{\mathcal C}_{P, X, \emptyset}^{\leq 0}&=&\left\{{\mathcal M}\in {\mathcal C}_{P, X, \emptyset} \,|\, {\rm H}^{k}({\mathcal M})=0 \textit{ for } k>0\right\} \textit{ and}\\
{\mathcal C}_{P, X, \emptyset}^{\geq 0}&=&\left\{{\mathcal M}\in {\mathcal C}_{P, X, \emptyset} \,|\, {\rm H}^{k}({\mathcal M})=0 \textit{ for } k<0\right\}.
\end{eqnarray*}
Then $\left({\mathcal C}_{P, X, \emptyset}^{\leq 0}, {\mathcal C}_{P, X, \emptyset}^{\geq 0}\right)$ defines a $t$-structure on ${\mathcal C}_{P, X, \emptyset}$, which we call the standard $t$-structure on ${\mathcal C}_{P, X, \emptyset}$.
For a $k$-embeddable $k$-scheme $X$ with an immersion $X\hookrightarrow P$ into a proper smooth $k$-scheme $P$, we take an open immersion $j: U\hookrightarrow P$ such that the immersion $X\hookrightarrow P$ factors as
a closed immersion $X\hookrightarrow U$ and $j$.
We define the standard $t$-structure $\left({\mathcal C}_{P, X}^{\leq 0}, {\mathcal C}_{P, X}^{\geq 0}\right)$ on ${\mathcal C}_{P, X}$ by the essential image of $\left({\mathcal C}_{U, X, \emptyset}^{\leq 0}, {\mathcal C}_{U, X, \emptyset}^{\geq 0}\right)$ under the equivalence ${\mathbb R}j_{*}$.
This definition is independent of the choice of $U\hookrightarrow P$ by the following lemma.
\end{defi}

\begin{lemm}\label{l15}
Let $j: U\hookrightarrow V$ be an open immersion of smooth $k$-schemes.
For any closed subset $X$ of $U$ which is also closed in $V$, 
the functor
\begin{equation*}
{\mathbb R}j_{*}: {\mathcal C}_{U, X, \emptyset}\to {\mathcal C}_{V, X, \emptyset}
\end{equation*}
induces an equivalence of triangulated categories, which is $t$-exact with respect to the standard $t$-structure.
\end{lemm}
\begin{proof}
We can prove that ${\mathbb R}j_{*}$ is an equivalence of triangulated categories with quasi-inverse $j^{-1}$ in the same way as in the proof of Proposition \ref{p2}.
Then it is enough to prove that ${\mathbb R}j_{*}$
and its quasi-inverse $j^{-1}$ are left $t$-exact (cf. \cite[Corollary 10.1.18]{KS}).
These claims are evident.
\end{proof}

We need the following lemma.

\begin{lemm}\label{l9}
Let $P$ be a smooth $W_{n}$-scheme and $i: X\hookrightarrow P$ a closed immersion.
For $\bullet \in\{\leq 0, \geq 0\}$,
we denote by $D_{{\rm c}, X}^{b}(P_{\rm \acute{e}t}, {\mathbb Z}/{p\mathbb Z})$ (resp. ${}^{\rm p}D_{{\rm c}, X}^{\bullet}(P_{\rm \acute{e}t}, {\mathbb Z}/{p\mathbb Z})$)  the full triangulated subcategory of $D_{{\rm c}}^{b}(P_{\rm \acute{e}t}, {\mathbb Z}/{p\mathbb Z})$ (resp. ${}^{\rm p}D_{{\rm c}}^{\bullet}(P_{\rm \acute{e}t}, {\mathbb Z}/{p\mathbb Z})$) consisting of complexes supported on $X$. 
Then the equivalence $i_{*}: D_{{\rm c}}^{b}(X_{\rm \acute{e}t}, {\mathbb Z}/{p\mathbb Z})\xrightarrow{\cong}D_{{\rm c}, X}^{b}(P_{\rm \acute{e}t}, {\mathbb Z}/{p\mathbb Z})$ restricts to an equivalence 
\begin{equation*}
{}^{\rm p}D_{{\rm c}}^{\bullet}(X_{\rm \acute{e}t}, {\mathbb Z}/{p\mathbb Z})\xrightarrow{\cong}{}^{\rm p}D_{{\rm c}, X}^{\bullet}(P_{\rm \acute{e}t}, {\mathbb Z}/{p\mathbb Z})
\end{equation*}
with quasi-inverse $i^{-1}$.
\end{lemm}
\begin{proof}
For an object $\mathcal L$ in ${}^{\rm p}D_{\rm c}^{\leq 0}(X_{\rm \acute{e}t}, {\mathbb Z}/{p\mathbb Z})$,
one obviously has $i_{*}{\mathcal L}\in {}^{\rm p}D_{{\rm c}, X}^{\leq 0}(P_{\rm \acute{e}t}, {\mathbb Z}/{p\mathbb Z})$.
Let $\mathcal L$ be an object of ${}^{\rm p}D_{\rm c}^{\geq 0}(X_{\rm \acute{e}t}, {\mathbb Z}/{p\mathbb Z})$ and $x$ an element of $X\subset P$.
Denote by $i_{x}$ (resp. $i'_{x}$) the canonical inclusion $\{x\}\hookrightarrow X$ (resp. $\{x\}\hookrightarrow P$).
One has $i'^{!}_{x}i_{*}{\mathcal L}\cong i_{x}^{!}i^{!}i_{*}{\mathcal L}\cong i_{x}^{!}{\mathcal L}$.
So we have ${\rm H}^{k}\left(i_{x}'^{!}{i_{*}\mathcal L}\right)=0$ for any $k<{\rm p}(x)$.
If $x\in P\setminus X$, one has ${\rm H}^{k}\left(i_{x}'^{!}{i_{*}\mathcal L}\right)=0$ for any $k$.
Hence we see $i_{*}{\mathcal L}\in {}^{\rm p}D_{{\rm c}, X}^{\geq 0}(P_{\rm \acute{e}t}, {\mathbb Z}/{p\mathbb Z})$.
Conversely, for an object ${\mathcal L}$ in ${}^{\rm p}D_{{\rm c}, X}^{\leq 0}(P_{\rm \acute{e}t}, {\mathbb Z}/{p\mathbb Z})$,
 one obviously has $i^{-1}{\mathcal L}\in {}^{\rm p}D_{{\rm c}}^{\leq 0}(X_{\rm \acute{e}t}, {\mathbb Z}/{p\mathbb Z})$.
Finally let ${\mathcal L}$ be an object of ${}^{\rm p}D_{{\rm c}, X}^{\geq 0}(P_{\rm \acute{e}t}, {\mathbb Z}/{p\mathbb Z})$.
Then, since ${\mathcal L}$ is supported on $X$, we have $i^{!}{\mathcal L}\xrightarrow{\cong} i^{-1}{\mathcal L}$ and hence we have $i^{-1}{\mathcal L}\in {}^{\rm p}D_{{\rm c}}^{\geq 0}(X_{\rm \acute{e}t}, {\mathbb Z}/{p\mathbb Z})$.
\end{proof}

\begin{coro}\label{c4}
Let $X$ be a $k$-embeddable $k$-scheme with a closed immersion $i: X\hookrightarrow P$ into a smooth $k$-scheme $P$.
Then ${\rm Sol}_{X}=i^{-1}{\rm Sol}_{P}: {\mathcal C}_{P, X, \emptyset}\xrightarrow{\cong} D^{b}_{\rm c}(X_{\rm \acute{e}t}, {\mathbb Z}/{p\mathbb Z})$ sends ${\mathcal C}_{P, X, \emptyset}^{\leq 0}$ (resp. ${\mathcal C}_{P, X, \emptyset}^{\geq 0}$) to  ${}^{\rm p}D_{\rm c}^{\geq 0}(X_{\rm \acute{e}t}, {\mathbb Z}/{p\mathbb Z})$ (resp. ${}^{\rm p}D_{\rm c}^{\leq 0}(X_{\rm \acute{e}t}, {\mathbb Z}/{p\mathbb Z})$).
\end{coro}
\begin{proof} 
By Lemma \ref{l10}, ${\rm Sol}_{P}$ restricts to an equivalence of triangulated categories
\begin{equation*}
{\mathcal C}_{P, X, \emptyset}\xrightarrow{\cong}D^{b}_{{\rm c}, X}(P_{\rm \acute{e}t}, {\mathbb Z}/{p\mathbb Z}).
\end{equation*}
We know that ${\rm Sol}_{P}$ sends ${\mathcal C}^{\leq 0}_{P, X, \emptyset}$ to ${}^{\rm p}D_{{\rm c}, X}^{\geq 0}(P_{\rm \acute{e}t}, {\mathbb Z}/{p\mathbb Z})$ and ${\mathcal C}^{\geq 0}_{P, X, \emptyset}$ to ${}^{\rm p}D_{{\rm c}, X}^{\leq 0}(P_{\rm \acute{e}t}, {\mathbb Z}/{p\mathbb Z})$ by \cite[Theorem 11.5.4]{EK}.
By Lemma \ref{l9}, we see that $i^{-1}$ sends ${}^{\rm p}D_{{\rm c}, X}^{\bullet}(P_{\rm \acute{e}t}, {\mathbb Z}/{p\mathbb Z})$ to ${}^{\rm p}D_{{\rm c}}^{\bullet}(X_{\rm \acute{e}t}, {\mathbb Z}/{p\mathbb Z})$ if $\bullet\in \{\leq 0, \geq 0\}$.
This finishes the proof.
\end{proof}

By Lemma \ref{l15} and Corollary \ref{c4}, one has the following theorem.

\begin{theo}\label{t11}
Let $X$ be a $k$-embeddable $k$-scheme with an immersion $X\hookrightarrow P$ into a proper smooth $k$-scheme $P$.
We set
\begin{eqnarray*}
D_{\rm lfgu}^{\leq 0}(X/k)= {\mathcal C}_{P, X}^{\leq 0} \textit{ and } D_{\rm lfgu}^{\geq 0}(X/k)={\mathcal C}_{P, X}^{\geq 0}.
\end{eqnarray*}
Then the $t$-structure $\left( D_{\rm lfgu}^{\leq 0}(X/k), D_{\rm lfgu}^{\geq 0}(X/k)\right)$
is independent of the choice of $X\hookrightarrow P$, which we call the standard $t$-structure on $D_{\rm lfgu}^{b}(X/k)$.
Furthermore, $\left( D_{\rm lfgu}^{\leq 0}(X/k), D_{\rm lfgu}^{\geq 0}(X/k)\right)$ corresponds to Gabber's perverse $t$-structure via ${\rm Sol}_{X}$.
\end{theo}

\subsection{Beilinson's theorem}\label{s52}

In this subsection, we prove an analogue of Beilinson's theorem (Theorem \ref{b1}), which is a generalization of \cite[Corollary 17.2.5]{EK} to the case of $k$-embeddable $k$-schemes.
In the rest of this subsection, we fix a $k$-embeddable $k$-scheme $X$, an immersion $\tilde{i}: X\hookrightarrow \tilde{P}$ into a proper smooth $k$-scheme and an open subscheme $P$ of $\tilde{P}$ such that 
$\tilde{i}$ factors as a closed immersion $i: X\hookrightarrow P$ and the open immersion $P\hookrightarrow \tilde{P}$.
Denote by $\mu_{\rm u}$ (resp. $\mu_{\rm lfgu}$) the category of unit ${\mathcal D}_{F, P}$-modules (resp. locally finitely generated unit ${\mathcal D}_{F, P}$-modules).
We also denote by $\mu_{{\rm u}, X}$ (resp. $\mu_{{\rm lfgu}, X}$) the full subcategory of $\mu_{\rm u}$ (resp. $\mu_{\rm lfgu}$) consisting of objects supported on $X$.
Note that $\mu_{{\rm lfgu}, X}$ is the heart of the standard $t$-structure on $D_{\rm lfgu}^{b}(X/k)={\mathcal C}_{\tilde{P}, X}={\mathcal C}_{P, X, \emptyset}$ and hence it is independent of the choice of $X\hookrightarrow \tilde{P}$ and $P$ by Theorem \ref{t11}.
The following theorem is the main theorem in this subsection.

\begin{theo}\label{b1}
The natural functor 
\begin{equation*}
D^{b}(\mu_{{\rm lfgu}, X})\to D_{\rm lfgu}^{b}(X/k)
\end{equation*}
is an equivalence of triangulated categories.
\end{theo}
The proof of Theorem \ref{b1} is divided into two parts.
First of all, we prove the following theorem.

\begin{theo}\label{b2}
The natural functor
\begin{equation*}
D^{b}(\mu_{{\rm lfgu}, X}) \to D_{\rm lfgu}^{b}(X/k)
\end{equation*}
 is essentially surjective and, for any objects $\mathcal M$ and $\mathcal N$ in $D^{b}(\mu_{{\rm lfgu}, X})$ 
 the map
\begin{equation*}
{\rm Hom}_{D^{b}(\mu_{{\rm lfgu}, X}) }({\mathcal M}, {\mathcal N})\to {\rm Hom}_{D_{\rm lfgu}^{b}(X/k)}({\mathcal M}, {\mathcal N})
\end{equation*}
is surjective.
\end{theo}
We need the following lemma.
\begin{lemm}\label{b3}
Let ${\rm Ind}\text{-}\mu_{{\rm lfgu},X}$ be the full subcategory of $\mu_{{\rm u}, X}$ consisting of objects which are direct limits of objects in $\mu_{{\rm lfgu},X}$.
Then the natural functor 
\begin{equation*}
D^{b}(\mu_{{\rm lfgu},X})\to D^{b}_{\rm lfgu}({\rm Ind}\text{-}\mu_{{\rm lfgu},X})
\end{equation*}
is an equivalence of triangulated categories.
\end{lemm}
\begin{proof}
For an object ${\mathcal M}$ in $D^{b}_{\rm lfgu}({\rm Ind}\text{-}\mu_{{\rm lfgu},X})$, there exists a subcomplex ${\mathcal M}'$ of
${\mathcal M}$ such that  the canonical inclusion ${\mathcal M}'\to {\mathcal M}$ is a quasi-isomorphism and the terms of ${\mathcal M}'$ are locally finitely generated unit.
Since $\mathcal M$ is supported on $X$, so is ${\mathcal M}'$.
Hence $D^{b}(\mu_{{\rm lfgu},X})\to D^{b}_{\rm lfgu}({\rm Ind}\text{-}\mu_{{\rm lfgu},X})$ is essentially surjective.
Let us prove the full faithfulness of the functor.
We denote by $K^{b}(\mu_{{\rm lfgu},X})$ (resp. $K^{b}({\rm Ind}\text{-}\mu_{{\rm lfgu},X})$) the (bounded) homotopy category of $\mu_{{\rm lfgu},X}$ (resp. ${\rm Ind}\text{-}\mu_{{\rm lfgu},X}$) and 
suppose that we are given a quasi-isomorphism ${\mathcal M}\to {\mathcal N}$ in $K^{b}({\rm Ind}\text{-}\mu_{{\rm lfgu},X})$ with  ${\mathcal N}\in K^{b}(\mu_{{\rm lfgu},X})$.
Then all cohomology sheaves of ${\mathcal M}$ are locally finitely generated unit and so there exists a subcomplex ${\mathcal M}'$ of
${\mathcal M}$ such that the canonical inclusion ${\mathcal M}'\to {\mathcal M}$ is a quasi-isomorphism and  the terms of ${\mathcal M}'$ are locally finitely generated unit.
Hence, by \cite[Proposition 1.6.5]{KS}, $D^{b}(\mu_{{\rm lfgu},X})\to D^{b}({\rm Ind}\text{-}\mu_{{\rm lfgu},X})$ is fully faithful
and the assertion follows.
\end{proof}

Let us prove the Theorem \ref{b2}.
\begin{proof}
The proof is a refinement of the proof of \cite[Corollary 17.1.2]{EK}.
For a $k$-scheme $Y$ of finite type, we denote by ${\mathfrak C}_{Y}$ the category of constructible \'etale sheaves of ${\mathbb Z}/{p\mathbb Z}$-modules on $Y_{\rm \acute{e}t}$.
By using the results in \cite[p.94]{De}, we know that the natural functor $D^{b}({\mathfrak C}_{Y})\to D_{\rm c}^{b}(Y_{\rm \acute{e}t}, {\mathbb Z}/{p\mathbb Z})$ is essentially surjective and induces a surjection on ${\rm Hom}$'s.
Let $E$ denote the residual complex of injective quasi-coherent ${\mathcal O}_{P_{\rm \acute{e}t}}$-modules resolving 
${\mathcal O}_{P_{\rm \acute{e}t}}$.
It is proved in \cite[Proposition 17.1.1]{EK} that $E$ naturally forms a complex of unit ${\mathcal D}_{F, P_{\rm \acute{e}t}}$-modules and the terms of $E$ are in ${\rm Ind}\text{-}\mu_{\rm lfgu}$.
Then, as in the proof of \cite[Corollary 17.1.2]{EK}, ${\rm M}_{X}$ may be computed as ${\pi}_{P*}\underline{\rm Hom}_{{\mathbb Z}/{p\mathbb Z}}(-, E)\circ i_{*}$ and we have the following commutative diagram of categories:
\[\xymatrix{
{D^{b}({\mathfrak C}_{X})} \ar[d]_{i_{*}}\ar[r] & {D_{\rm c}^{b}(X_{\rm \acute{e}t}, {\mathbb Z}/{p\mathbb Z})}\ar[d]^{i_{*}} \\
{D^{b}({\mathfrak C}_{P})} \ar[d]_{{\pi}_{P*}\underline{\rm Hom}_{{\mathbb Z}/{p\mathbb Z}}(-, E)}\ar[r] & {D_{\rm c}^{b}(P_{\rm \acute{e}t}, {\mathbb Z}/{p\mathbb Z})}\ar[d]^{{\rm M}_{P}}\\
{D^{b}_{\rm lfgu}({\rm Ind}\text{-}\mu_{\rm lfgu})} \ar[r] & {D^{b}_{\rm lfgu}({\mathcal D}_{F, P})} .}
\]
The composite of the functors 
$$D^{b}({\mathfrak C}_{X})\to D_{\rm c}^{b}(X_{\rm \acute{e}t}, {\mathbb Z}/{p\mathbb Z})\to D_{\rm c}^{b}(P_{\rm \acute{e}t}, {\mathbb Z}/{p\mathbb Z})\to D^{b}_{\rm lfgu}({\mathcal D}_{F, P})$$
induces a functor $D^{b}({\mathfrak C}_{X})\to D^{b}_{\rm lfgu}(X/k)$ which is essentially surjective and induces a surjection on ${\rm Hom}$'s by Theorem \ref{t5}.
On the other hand, the essential image of the composite of the functors $D^{b}({\mathfrak C}_{X})\to D^{b}({\mathfrak C}_{P})\to D^{b}_{\rm lfgu}({\rm Ind}\text{-}\mu_{\rm lfgu})$ is contained in $D^{b}_{\rm lfgu}({\rm Ind}\text{-}\mu_{{\rm lfgu}, X})$
because for an object ${\mathcal G}$ in $D^{b}({\mathfrak C}_{X})$, we have the natural isomorphism
\begin{equation*}
\pi_{P*}\underline{\rm Hom}_{{\mathbb Z}/{p\mathbb Z}}(i_{*}{\mathcal G}, E)\cong \pi_{P*}i_{*}\underline{\rm Hom}_{{\mathbb Z}/{p\mathbb Z}}({\mathcal G}, i^{!}E_{|X}).
\end{equation*}
Hence we know that the functor $D^{b}_{\rm lfgu}({\rm Ind}\text{-}\mu_{{\rm lfgu}, X})\to D^{b}_{\rm lfgu}(X/k)$ is essentially surjective.
 So, by using Theorem \ref{t5} and Lemma \ref{b3}, we see that the functor induces a surjection on ${\rm Hom}$'s.
Now the assertion follows from Lemma \ref{b3}.
\end{proof}
In order to prove the full faithfulness of the functor $D^{b}(\mu_{{\rm lfgu}, X}) \to D_{\rm lfgu}^{b}(X/k)$, we need some preparation. 

\begin{lemm}\label{b4}
The category $\mu_{{\rm u},X}$ has enough injectives.
\end{lemm}
\begin{proof}
For an object $\mathcal M$ in $\mu_{{\rm u},X}$, we can take an injection $\mathcal M \to I$ into an injective object $\mathcal I$ in $\mu_{\rm u}$ by \cite[Corollary 15.1.6]{EK}.
Applying ${\Gamma}_{X}$ to the injection $\mathcal M \to I$, one has an injection ${\mathcal M}={\Gamma}_{X}{\mathcal M}\to {\Gamma}_{X}{\mathcal I}$.
Hence it is enough to prove that $\Gamma_X {\mathcal I}$ is an injective object in $\mu_{{\rm u},X}$.
Suppose that we are given an injection $i: {\mathcal N}' \to {\mathcal N}$ and a morphism $f: {\mathcal N}' \to \Gamma_X {\mathcal I}$ in $\mu_{{\rm u},X}$.
Let us denote by $g$ the natural morphism $\Gamma_X{\mathcal I} \to {\mathcal I}$.
Since ${\mathcal I}$ is an injective object, there exists
a morphism $h: {\mathcal N \to I}$ satisfying $h \circ i = g \circ f$.
Then one has $\Gamma_X h \circ \Gamma_X i = \Gamma_X g \circ \Gamma_X f$.
Since $\Gamma_X i$ is equal to $i: {\mathcal N}' = \Gamma_X {\mathcal N}' \to \Gamma_X {\mathcal N = N}$
and $\Gamma_X g \circ \Gamma_X f: {\mathcal N}' = \Gamma_X {\mathcal N}' \to \Gamma_X \Gamma_X {\mathcal I} \to \Gamma_X {\mathcal I}$ is equal to
$f$, we know the equality $\Gamma_X h \circ i = f$. 
Hence $\Gamma_X {\mathcal I}$ is an injective object in $\mu_{{\rm u},X}$.
\end{proof}

For an object ${\mathcal M}$ in $\mu_{{\rm u},X}$, 
there exists a unique maximal subobject $L({\mathcal M})$ of $\mathcal M$ which lies in ${\rm Ind}\text{-}\mu_{{\rm lfgu}}$ by  \cite[Lemma 17.2.1.(i)]{EK}.
Then $L({\mathcal M})$ belongs to ${\rm Ind}\text{-}\mu_{{\rm lfgu}, X}$ and it is a unique maximal subobject  of $\mathcal M$ which lies in ${\rm Ind}\text{-}\mu_{{\rm lfgu}, X}$.
By \cite[Lemma 17.2.1.(ii)]{EK}, the correspondence ${\mathcal M}\mapsto L({\mathcal M})$ defines a left exact functor
$$ L: \mu_{{\rm u},X} \to {\rm Ind}\text{-}\mu_{{\rm lfgu},X}$$
which is right adjoint to the natural functor $\mu_{{\rm lfgu},X}\to \mu_{{\rm u},X}$.
Since $\mu_{{\rm u},X}$ has enough injectives by Lemma \ref{b4}, we obtain the right derived functor 
$$ {\mathbb R}L: D^+( \mu_{{\rm u},X} ) \to D^+( {\rm Ind}\text{-}\mu_{{\rm lfgu},X}). $$
By using Theorem \ref{b2}, one can prove the following lemma in the same way as \cite[Lemma 17.2.2]{EK}.

\begin{lemm}\label{b5}
Objects in $\mu_{{\rm lfgu},X}$ are acyclic for ${\mathbb R}L$.
\end{lemm}

\begin{proof}
For an object $\mathcal M$  in $\mu_{{\rm lfgu},X}$,
one can choose an injective resolution ${\mathcal M}\to {\mathcal I}$ in $\mu_{u,X}$ by Lemma \ref{b4}.
For a natural number $n\geq 0$, we denote by ${\mathcal E}^{n}$ the image of the differential ${\mathcal I}^n\to {\mathcal I}^{n+1}$.
In order to prove that $\mathcal M$ is ${\mathbb R}L$-acyclic, it is enough to prove that the map $L({\mathcal I}^n)\to L({\mathcal E}^n)$ is surjective.
We denote by ${\mathcal I}^{\leq n}$ the complex defined by $\left({\mathcal I}^{\leq n}\right)^i={\mathcal I}^i$ for $i\leq n$ and by $\left({\mathcal I}^{\leq n}\right)^i=0$ for $i>n$.
Then one has an $(n+1)$-extension
$$0\to {\mathcal M}\to {\mathcal I}^0\to {\mathcal I}^1\to \cdots \to {\mathcal I}^n\to {\mathcal E}^n\to 0$$
 of ${\mathcal E}^n$ by $\mathcal M$ and
 denotes by $c$ the class of this extension in ${\rm Ext}^{n+1}_{\mu_{{\rm u},X}}({\mathcal I}^n, {\mathcal M})$.
For a locally finitely generated unit ${\mathcal D}_{F, P}$-submodule $\mathcal F$ of ${\mathcal E}^n$, 
we denote by $c_{{\mathcal F}}$ the image of $c$ under the map 
\begin{eqnarray*}
{\rm Ext}^{n+1}_{\mu_{{\rm u},X}}({\mathcal E}^n, {\mathcal M})\to {\rm Ext}^{n+1}_{\mu_{{\rm u},X}}({\mathcal F}, {\mathcal M})&=&{\rm Hom}_{D^b(\mu_{{\rm u},X})}({\mathcal F}[-n], {\mathcal M})\\
&\xrightarrow{\alpha}& {\rm Hom}_{D^b_{\rm qc}({\mathcal D}_{F, P})}({\mathcal F}[-n], {\mathcal M}).
\end{eqnarray*}
Note that $\alpha$ is an isomorphism:
Indeed we have the diagram of functors
$$D^b(\mu_{{\rm u},X})\to D^b(\mu_{{\rm u}})\to D^b(\mu_{{\rm qc}})\to D^b_{\rm qc}({\mathcal D}_{F, P})$$
(where $\mu_{{\rm qc}}$ denotes the category of ${\mathcal O}_{P}$-quasi-coherent ${\mathcal D}_{F, P}$-modules)
in which the first (resp. the second, the third) functor is fully faithful by Lemma \ref{b7} below (resp. \cite[Lemma 17.2.3 (iii)]{EK}, Bernstein's theorem  \cite[Corollary 17.2.4]{EK}).
Hence we can regard $c_{\mathcal F}$ also as an element in ${\rm Ext}^{n+1}_{\mu_{{\rm u},X}}({\mathcal F}, {\mathcal M})$.

By Theorem \ref{b2}, there exists an $(n+1)$-extension in $\mu_{{\rm lfgu},X}$ which is sent to $c_{\mathcal F}$ by the map 
\begin{eqnarray*}
{\rm Hom}_{D^{b}(\mu_{{\rm lfgu}, X}) }({\mathcal F}[-n], {\mathcal M})&\to& {\rm Hom}_{D_{\rm lfgu}^{b}(X/k)}({\mathcal F}[-n], {\mathcal M})\\
&=&{\rm Hom}_{D^b_{\rm qc}({\mathcal D}_{F, P})}({\mathcal F}[-n], {\mathcal M}).
\end{eqnarray*}
Let us denote this $(n+1)$-extension by
$$0\to {\mathcal M}\to {\mathcal N}\to {\mathcal F}\to 0,$$
where ${\mathcal N}$ is a complex of locally finitely generated unit ${\mathcal D}_{F, P}$-modules whose terms are supported on $X$ and are $0$ outside $[0, n]$ and such that ${\rm H}^i({\mathcal N})={\mathcal M}$ if $i=0$,  ${\rm H}^i({\mathcal N})={\mathcal F}$ if $i=n$ and  ${\rm H}^i({\mathcal N})=0$ otherwise.
Since ${\mathcal I}$ is a complex of injective objects in $\mu_{u,X}$ there exists a map of extensions

\[\xymatrix{
0 \ar[r] & {\mathcal M} \ar[r] &{{\mathcal I}^{\leq n}} \ar[r]  & {{\mathcal E}^{n}}\ar[r] &0\\
0 \ar[r] & {\mathcal M} \ar[r]\ar[u]^{\rm id} & {{\mathcal N}}\ar[r]\ar[u]^{\phi} & {\mathcal F} \ar[r]\ar[u]^{\psi} & 0.
}\]
Let us consider the exact sequence 
$${\rm Hom}_{\mu_{{\rm u},X}}({\mathcal F}, {\mathcal I}^n)\to {\rm Hom}_{\mu_{{\rm u},X}}({\mathcal F}, {\mathcal E}^n)\xrightarrow{\delta} {\rm Ext}^{n+1}_{\mu_{{\rm u},X}}({\mathcal F}, {\mathcal M})\to 0.$$
By construction of $c_{\mathcal F}$ one has $\delta(\psi)=c_{{\mathcal F}}$. 
If we denote by $\psi'$ the natural inclusion ${\mathcal F}\to {\mathcal E}^n$, then we also have $\delta(\psi')=c_{{\mathcal F}}$.
Thus $\psi-\psi'$ lifts to a map $\widetilde{\psi-\psi'}: {\mathcal F}\to {\mathcal I}^n$.
Let us also denote by  $\widetilde{\psi-\psi'}$ the composite of morphisms ${\mathcal N}^n\to {\mathcal F}\xrightarrow{\widetilde{\psi-\psi'}} {\mathcal I}^n$.
Then we have a locally finitely generated unit ${\mathcal D}_{F, P}$-submodule $\left(\phi^n-\widetilde{\psi-\psi'}\right)\left({\mathcal N}^n\right)$ of ${\mathcal I}^n$ which surjects on $\psi'({\mathcal F})$.
This finishes the proof.

\end{proof}

\begin{lemm}\label{b6}
Let  $\mu_{{L\text{-ac}}, {X}}$ denote the full subcategory of $\mu_{{\rm u}, X}$ consisting of
$L$-acyclic objects. 
Then the natural functors 
$D^b(\mu_{{\rm lfgu}, X}) \to D^+(\mu_{{L\text{-ac}, X}})$ and $D^+(\mu_{L\text{-ac}, X}) \to D^+(\mu_{{\rm u}, {X}})$ are fully faithful.
As a consequence, the natural functor 
$$D^b(\mu_{{\rm lfgu}, X}) \to D^b(\mu_{{\rm u}, X})$$
is fully faithful.
\end{lemm}
\begin{proof}
The strategy of the proof is the same as that of \cite[Lemma 17.2.3]{EK} but we slightly modify their proof.
Let us suppose that we are given a quasi-isomorphism $\mathcal N \to M$ in $K^+(\mu_{{L\text{-ac}}, {X}})$ with ${\mathcal M}\in K^b(\mu_{{\rm lfgu}, X})$.
Then the adjunction morphism $L({\mathcal N}) \to {\mathcal N}$ is a quasi-isomorphism since the terms of $\mathcal N$ and its cohomology sheaves are acyclic for $L$  by Lemma \ref{b5}.
Now since the terms of $L({\mathcal N})$ are in ${\rm Ind}\text{-}\mu_{{\rm lfgu}, X}$ and cohomology sheaves of it are in 
$\mu_{{\rm lfgu}, {X}}$, 
there exists a bounded subcomplex $L({\mathcal N})'$ of  $L({\mathcal N})$ such that $L({\mathcal N})'\to L({\mathcal N})$ is a quasi-isomorphism and $L({\mathcal N})'$ belongs to $K^b(\mu_{{\rm lfgu}, X})$.
Hence the first functor is fully faithful by  \cite[Proposition 1.6.5]{KS}.
Next suppose that we are given a quasi-isomorphism $\mathcal N\to M$ in $K^+(\mu_{L\text{-ac}, X})$ with ${\mathcal N} \in K^+(\mu_{{\rm u}, X})$.
Then, by Lemma \ref{b4}, one has an injective resolution $\mathcal I$ of ${\mathcal M}$ in $K^+(\mu_{{\rm u}, X})$.
So the full faithfulness of the second functor also follows from \cite[Proposition 1.6.5]{KS}.
\end{proof}

Let us consider the following commutative diagram of categories:
\[\xymatrix{
{D^b(\mu_{{\rm lfgu}, X})} \ar[d]\ar[r] & D^b(\mu_{{\rm lfgu}})\ar[d] \\
{D^b(\mu_{{\rm u}, X})} \ar[d]\ar[r] & D^b(\mu_{{\rm u}})\ar[d]\\
{D_{\rm lfgu}^{b}(X/k)} \ar[r] & {D_{\rm lfgu}^{b}({\mathcal D}_{F, P})} .}
\]
In order to prove the full faithfulness of the functor $D^b(\mu_{{\rm lfgu}, X})\to D_{\rm lfgu}^{b}(X/k)$,
by Lemma \ref{b6} and \cite[Corollary 17.2.4]{EK}, it suffices to prove the following lemma. 

\begin{lemm}\label{b7}
The natural functor $D^b(\mu_{{\rm u},X}) \to D^b(\mu_{\rm u})$ is fully faithful.
\end{lemm}
In order to prove Lemma \ref{b7}, we define a functor
$${\mathbb R}'\Gamma_X: D^+(\mu_{\rm u}) \to D^+(\mu_{{\rm u},X})$$
to be the right derived functor of the left exact functor $\Gamma_X: \mu_{\rm u} \to \mu_{{\rm u},X}$.
Here we use the notation ${\mathbb R}'\Gamma_X$ instead of ${\mathbb R}{\Gamma}_{X}: D^b_{\rm qc}({\mathcal D}_{F, P})\to D^b_{\rm qc}({\mathcal D}_{F, P})$ to avoid confusion.
We have the following lemma.

\begin{lemm}\label{b8}
Objects in $\mu_{{\rm u},X}$ are acyclic for ${\mathbb R}'\Gamma_X$.
\end{lemm}
\begin{proof} 
For an object $\mathcal M$ in $\mu_{{\rm u},X}$, take an injective resolution $\mathcal M \to I$ in $\mu_{{\rm u}}$.
Since $\mathcal M$ is supported on $X$, we have ${\rm H}^i({\mathbb R}\Gamma_X{\mathcal M}) = 0$ for $i > 0$.
On the other hand, each ${\mathcal I}^n$ is injective in $\mu_{{\rm u}}$ by definition and it is known by
 \cite[Corollary 15.1.6]{EK} that such an object is always injective in the category of ${\mathcal O}_{P}$-quasi-coherent ${\mathcal D}_{F, P}$-modules.
So we have ${\rm H}^i({\mathbb R}\Gamma_X{\mathcal I}^n) = 0$ for  $i > 0$.
By considering the long exact sequence for ${\mathbb R}\Gamma_X$, we deduce that
$0 \to {\mathcal M} = \Gamma_X {\mathcal M} \to \Gamma_X {\mathcal I}$ is exact. 
Hence $\mathcal M$ is acyclic for ${\mathbb R}'\Gamma_X$.
\end{proof}
Let us prove Lemma \ref{b7}.
\begin{proof}
Let us denote by $\mu_{\Gamma_X\text{-ac}}$ the full subcategory of $\mu_{{\rm u}, X}$ consisting of
${\mathbb R}'\Gamma_X$-acyclic objects. 
It is enough to prove that the natural functors $D^b(\mu_{{\rm u},X}) \to D^+(\mu_{\Gamma_X\text{-ac}})$ and
$D^+(\mu_{\Gamma_X\text{-ac}}) \to D^+(\mu_{\rm u})$ are fully faithful.
Let us suppose that we are given a quasi-isomorphism $\mathcal N \to M$ in $K^+(\mu_{\Gamma_X\text{-ac}})$ with ${\mathcal M}\in K^b(\mu_{{\rm u},X})$.
Then the natural morphism $\Gamma_X {\mathcal N} \to {\mathcal N}$ is a quasi-isomorphism since the terms of $\mathcal N$ and its cohomology sheaves are acyclic for ${\mathbb R}'\Gamma_X$  by Lemma \ref{b8}.
Moreover, since $\Gamma_X {\mathcal N}$ is cohomologically bounded, there exists a bounded subcomplex ${\Gamma_X} {\mathcal N}'$ of $\Gamma_X {\mathcal N}$ such that $\Gamma_X {\mathcal N}'\to \Gamma_X {\mathcal N}$ is a quasi-isomorphism and $\Gamma_X {\mathcal N}'$ belongs to $K^b(\mu_{{\rm u},X})$.
Hence, by \cite[Proposition 1.6.5]{KS}, we know that the first functor is fully faithful.
For the second assertion, let us take a quasi-isomorphism $\mathcal N \to M$ in $K^+(\mu_{\Gamma_X\text{-ac}})$ with 
 ${\mathcal N} \in K^+(\mu_{{\rm u}})$.
Then $\mathcal M$ is quasi-isomorphic to its injective resolution $\mathcal I$ in $K^+(\mu_{\rm u})$.
Hence the second functor is fully faithful by \cite[Proposition 1.6.5]{KS}.

\end{proof}
By Theorem \ref{b2} and Lemma \ref{b7}, we finish the proof of Theorem \ref{b1}.

\subsection{The constructible $t$-structure on $D^{b}_{\rm lfgu}(X/k)$}\label{s53}
Let $P$ be a smooth $k$-scheme.
Let ${\mathcal A}$ be a sheaf of ${\mathcal O}_{P}$-algebra which is quasi-coherent as a left ${\mathcal O}_{P}$-module and left noetherian.
Let us first recall a $t$-structure on $D_{\rm qc}^{b}({\mathcal A})$ introduced by Kashiwara in \cite{Kas2}.
For more detail, we refer the reader to \cite[\S3]{Kas2}.
We define a support datum ${\mathfrak S}=\{{\mathfrak S}^{n}\}$ by
\begin{equation*}
{\mathfrak S}^{n}:=\left\{ Z\,|\, Z \textit{ is a closed subset of }P \text{ of codimension}\geq n  \right\}.
\end{equation*}
Then ${\mathfrak S}^{n}$ has the structure of an ordered set by the natural inclusion.
For a sheaf $\mathcal F$ of $\mathcal A$-modules, we define ${\Gamma}_{{\mathfrak S}^{n}}({\mathcal F}):=\displaystyle\varinjlim_{Z\in {\mathfrak S}^{n}}{\Gamma}_{Z}({\mathcal F})$.
Then ${\Gamma}_{{\mathfrak S}^{n}}$ defines a left exact functor from the category of $\mathcal A$-modules to itself
and we obtain the right derived functor ${\mathbb R}{{\Gamma}_{{\mathfrak S}^{n}}}: D^{b}_{\rm qc}({\mathcal A})\to D^{b}_{\rm qc}({\mathcal A})$. 

We define a full subcategory ${}^{\mathfrak S}D_{\rm qc}^{\leq k}({\mathcal A})$ (resp. ${}^{\mathfrak S}D_{\rm qc}^{\geq k}({\mathcal A})$) of $D_{\rm qc}^{b}({\mathcal A}) $ by the condition: 
 ${\mathcal M}$ is in ${}^{\mathfrak S}D_{\rm qc}^{\leq k}({\mathcal A})$ if and only if 
 ${\mathbb R}{{\Gamma}_{{\mathfrak S}^{n-k}}}{\rm H}^{n}\left({\mathcal M}\right)\xrightarrow{\cong} {\rm H}^{n}\left({\mathcal M}\right)$ for any $n$ (resp. ${\mathcal M}$ is in ${}^{\mathfrak S}D_{\rm qc}^{\geq k}({\mathcal A})$ if and only if 
 ${\mathbb R}{\Gamma}_{Z}{\mathcal M}\in D_{\rm qc}^{\geq n+k}({\mathcal A})$ for any $n$ and  $Z\in {\mathfrak S}^{n}$).

Kashiwara proved that $\left({}^{\mathfrak S}D_{\rm qc}^{\leq 0}({\mathcal A}), {}^{\mathfrak S}D_{\rm qc}^{\geq 0}({\mathcal A}) \right)$ forms a $t$-structure on $D_{\rm qc}^{b}({\mathcal A})$.
We call this $t$-structure the constructible $t$-structure.
In particular, we have a $t$-structure $\left({}^{\mathfrak S}D_{\rm qc}^{\leq 0}({\mathcal O}_{F, P}), {}^{\mathfrak S}D_{\rm qc}^{\geq 0}({\mathcal O}_{F, P}) \right)$ on $D_{\rm qc}^{b}({\mathcal O}_{F, P})$.
Moreover, we have the following theorem.

\begin{theo}\label{t2}
Let $P$ be a smooth separated $k$-scheme.
We set
\begin{eqnarray*}
{}^{\mathfrak S}D_{\rm lfgu}^{\leq 0}({\mathcal O}_{F, P})&:=& {}^{\mathfrak S}D_{\rm qc}^{\leq 0}({\mathcal O}_{F, P})\cap D_{\rm lfgu}^{b}({\mathcal O}_{F, P})\textit{ and}\\
{}^{\mathfrak S}D_{\rm lfgu}^{\geq 0}({\mathcal O}_{F, P})&:=& {}^{\mathfrak S}D_{\rm qc}^{\geq 0}({\mathcal O}_{F, P})\cap D_{\rm lfgu}^{b}({\mathcal O}_{F, P}).
\end{eqnarray*}
Then $\left({}^{\mathfrak S}D_{\rm lfgu}^{\leq 0}({\mathcal O}_{F, P}), {}^{\mathfrak S}D_{\rm lfgu}^{\geq 0}({\mathcal O}_{F, P})\right)$ defines a $t$-structure on $D_{\rm lfgu}^{b}({\mathcal O}_{F, P})$. 
\end{theo}

\begin{proof}
It suffices to show that for any ${\mathcal M}\in D_{\rm lfgu}^{b}({\mathcal O}_{F, P})$, 
there exists a distinguished triangle
\begin{equation*}
\mathcal M'\to M\to M''\xrightarrow{+}
\end{equation*}
such that ${\mathcal M'}\in {}^{\mathfrak S}D_{\rm lfgu}^{< 0}({\mathcal O}_{F, P})$ and ${\mathcal M''}\in {}^{\mathfrak S}D_{\rm lfgu}^{\geq 0}({\mathcal O}_{F, P})$.
We show it by  induction on the codimension $d$ of $S:={\rm Supp}({\mathcal M})$. 
Let us consider a distinguish triangle
\begin{equation}\label{e14}
\tau^{<d}{\mathcal M}\to {\mathcal M}\to \tau^{\geq d}{\mathcal M}\xrightarrow{+},
\end{equation}
where $\tau$ denotes the truncation functor with respect to the standard $t$-structure.
Evidently, one has ${\mathbb R}{\Gamma}_{{\mathfrak S}^{k+1}}\left({\rm H}^{k}\left(\tau^{<d}{\mathcal M}\right)\right)\xrightarrow{\cong} {\rm H}^{k}\left(\tau^{<d}{\mathcal M}\right)$ for any $k\geq d$.
For any $k<d$, one has $S\in {\mathfrak S}^{k+1}$ and 
$${\mathbb R}{\Gamma}_{{\mathfrak S}^{k+1}}\left({\rm H}^{k}\left(\tau^{<d}{\mathcal M}\right)\right)\xrightarrow{\cong} {\rm H}^{k}\left(\tau^{<d}{\mathcal M}\right).$$
Hence we have $\tau^{<d}{\mathcal M}\in  {}^{\mathfrak S}D_{\rm lfgu}^{< 0}({\mathcal O}_{F, P})$.
By using \cite[Lemma 2.1]{Kas2} with (\ref{e14}), we are reduced to the case where ${\mathcal M}$ is an object in $D_{\rm lfgu}^{\geq d}({\mathcal O}_{F, P})$.
Let $S_{0}$ be a $d$-codimensional smooth open subscheme of $S$ such that $H:=S\setminus S_{0}$ is of codimension $>d$.
We set $U:=X\setminus H$. 
Then we have a closed immersion $i: S_{0}\hookrightarrow U$ and the open immersion $j: U\hookrightarrow P$.
Since ${\mathcal M}_{|U}$ is supported on $S_{0}$, by \cite[Corollary 5.11.3]{EK}, there exists an object ${\mathcal N}\in D_{\rm lfgu}^{b}({\mathcal O}_{F, S_{0}})$ such that $i_{+}{\mathcal N}\cong {\mathcal M}_{|U}$.
Note that, by \cite[Corollary 3.3.6]{EK}, $i_{+}$ is $t$-exact with respect to the standard $t$-structure.
So $\mathcal N$ belongs to $D_{\rm lfgu}^{\geq d}({\mathcal O}_{F, S_{0}})$.
Applying \cite[Proposition 6.9.6]{EK}, by shrinking $S_{0}$ if necessary, we may assume that all cohomology sheaves of $\mathcal N$ are unit $F$-crystals.
In particular, these are locally free of finite rank.
Then we claim that $i_{+}{\mathcal N}$ belongs to ${}^{\mathfrak S}D_{\rm lfgu}^{\geq 0}({\mathcal O}_{F, U})$.
In order to see this claim, by the induction on the cohomological length of $\mathcal N$, we may assume that $\mathcal N$ is a single unit $F$-crystal supported on degree $\geq d$. 
Then for any $n$-codimensional closed subset $Z$ of $U$, we have ${\mathbb R}{\Gamma}_{Z\cap S_{0}}({\mathcal N})\cong {\mathbb R}{\Gamma}_{Z\cap S_{0}}({\mathcal O}_{S_{0}})\otimes {\mathcal N}\in D_{\rm lfgu}^{\geq n}({\mathcal O}_{F, S_{0}})$.
Then since $i_{+}$ is left $t$-exact with respect to the standard $t$-structure, we have ${\mathbb R}{\Gamma}_{Z}i_{+}{\mathcal N}\cong i_{+}{\mathbb R}{\Gamma}_{Z\cap S_{0}}{\mathcal N}\in D_{\rm lfgu}^{\geq n}({\mathcal O}_{F, U})$ as desired.
Because ${\mathbb R}j_{*}$ is left $t$-exact with respect to the constructible $t$-structure by \cite[Lemma 3.7]{Kas2},
 one has ${\mathbb R}j_{*}i_{+}{\mathcal N}{\cong}{\mathbb R}j_{*}j^{-1}{\mathcal M}\in {}^{\mathfrak S}D_{\rm lfgu}^{\geq 0}({\mathcal O}_{F, P})$.
Let us consider a distinguished triangle
\begin{equation*}
{\mathbb R}{\Gamma}_{H}{\mathcal M}\to {\mathcal M}\to {\mathbb R}j_{*}j^{-1}{\mathcal M}\xrightarrow{+}.
\end{equation*}
Since the codimension of ${\rm Supp}\left({\mathbb R}{\Gamma}_{H}{\mathcal M}\right)$ is greater than $d$, then the induction proceeds by \cite[Lemma 2.1]{Kas2}.
\end{proof}
\begin{coro}
For a smooth separated $k$-scheme $P$ with closed subsets $Z$ and $T$ of $P$, 
we set
\begin{eqnarray*}
{}^{\mathfrak S}{\mathcal C}^{\leq 0}_{P, Z, T}&:=& {}^{\mathfrak S}D_{\rm qc}^{\leq 0}({\mathcal O}_{F, P})\cap {\mathcal C}_{P, Z, T}\textit{ and}\\
{}^{\mathfrak S}{\mathcal C}^{\geq 0}_{P, Z, T}&:=& {}^{\mathfrak S}D_{\rm qc}^{\geq 0}({\mathcal O}_{F, P})\cap {\mathcal C}_{P, Z, T}.
\end{eqnarray*}
Then $\left({}^{\mathfrak S}{\mathcal C}^{\leq 0}_{P, Z, T}, {}^{\mathfrak S}{\mathcal C}^{\geq 0}_{P, Z, T}\right)$ defines a $t$-structure on ${\mathcal C}_{P, Z, T}$, which we call the constructible $t$-structure on ${\mathcal C}_{P, Z, T}$.

\end{coro}
\begin{proof}
Denote by $j$ the open immersion $P\setminus T\hookrightarrow P$.
For an object ${\mathcal M} \in{\mathcal C}_{P, Z, T}\subset D_{\rm lfgu}^{b}({\mathcal O}_{F, P})$.
 there exists a distinguished triangle
\begin{equation*}
\mathcal M'\to M\to M''\xrightarrow{+}
\end{equation*}
such that ${\mathcal M'}\in {}^{\mathfrak S}D_{\rm lfgu}^{< 0}({\mathcal O}_{F, P})$ and ${\mathcal M''}\in {}^{\mathfrak S}D_{\rm lfgu}^{\geq 0}({\mathcal O}_{F, P})$.
Since ${\mathbb R}{\Gamma}_{Z}$, ${\mathbb R}j_{*}$ and $j^{-1}$ are $t$-exact with respect to the constructible $t$-structure by \cite[Proposition 4.1, Proposition 4.2 and Lemma 3.7]{Kas2} respectively, we have a 
desired distinguished triangle
\begin{equation*}
 {\mathbb R}j_{*}j^{-1}{\mathbb R}{\Gamma}_{Z}{\mathcal M'}\to{\mathcal M}\to {\mathbb R}j_{*}j^{-1}{\mathbb R}{\Gamma}_{Z}{\mathcal M''}\xrightarrow{+}
\end{equation*} 
such that ${\mathbb R}j_{*}j^{-1}{\mathbb R}{\Gamma}_{Z}{\mathcal M'}\in {}^{\mathfrak S}{\mathcal C}^{<0}_{P, Z, T}$ and ${\mathbb R}j_{*}j^{-1}{\mathbb R}{\Gamma}_{Z}{\mathcal M''}\in {}^{\mathfrak S}{\mathcal C}^{\geq 0}_{P, Z, T}$.
This finishes the proof.
\end{proof}
For a $k$-embeddable $k$-scheme $X$ with an immersion $X\hookrightarrow P$ into a proper smooth $k$-scheme $P$, we define by
$${}^{\mathfrak S}{\mathcal C}^{\leq 0}_{P, X}:={}^{\mathfrak S}{\mathcal C}^{\leq 0}_{P, Z, T} \textit{ and } {}^{\mathfrak S}{\mathcal C}^{\geq 0}_{P, X}:={}^{\mathfrak S}{\mathcal C}^{\geq 0}_{P, Z, T}$$
for some closed subsets $Z$ and $T$ of $P$ satisfying $X=Z\setminus T$.
Then this definition is independent of the choice of $Z$ and $T$ by Lemma \ref{l1} and $\left({}^{\mathfrak S}{\mathcal C}^{\leq 0}_{P, X}, {}^{\mathfrak S}{\mathcal C}^{\geq 0}_{P, X}\right)$ defines a $t$-structure on ${\mathcal C}_{P, X}$, which we call the constructible $t$-structure on ${\mathcal C}_{P, X}$. 
By \cite[Proposition 4.2 and Lemma 3.7]{Kas2}, one immediately obtains the following lemma.

\begin{lemm}\label{l16}
Let $X$ be a $k$-embeddable $k$-scheme with an immersion $X\hookrightarrow P$ into a proper smooth $k$-scheme $P$.
Let $U$ be an open subscheme of $P$ such that the immersion $X\hookrightarrow P$ factors as a closed immersion $X\hookrightarrow U$ and the open immersion $j: U\hookrightarrow X$.
Then the equivalence in Proposition \ref{p2}
\begin{equation*}
{\mathbb R}j_{*}: {\mathcal C}_{U, X, \emptyset}\xrightarrow{\cong} {\mathcal C}_{P, X}
\end{equation*}
is $t$-exact with respect to the constructible $t$-structure.
\end{lemm}
\begin{theo}\label{t8}
Let $X$ be a $k$-embeddable $k$-scheme with a closed immersion $i$ into a smooth separated $k$-scheme $P$.
We set 
\begin{eqnarray*}
D_{\rm c}^{\leq 0}(X_{\rm \acute{e}t}) &=&\left\{{\mathcal F}\in D_{\rm c}^{b}(X_{\rm \acute{e}t}, {\mathbb Z}/{p\mathbb Z}) \,|\, {\rm H}^{k}({\mathcal F})=0 \textit{ for } k>0\right\} \textit{ and}\\
D_{\rm c}^{\geq 0}(X_{\rm \acute{e}t}) &=&\left\{{\mathcal F}\in D_{\rm c}^{b}(X_{\rm \acute{e}t}, {\mathbb Z}/{p\mathbb Z}) \,|\, {\rm H}^{k}({\mathcal F})=0 \textit{ for } k<0\right\} . 
\end{eqnarray*}
Then the equivalence of triangulated categories
\begin{equation*}
{\rm Sol}_{X}=i^{-1}{\rm Sol}_{P}: {\mathcal C}_{P, X, \emptyset}\xrightarrow{\cong} D_{\rm c}^{b}(X_{\rm \acute{e}t}, {\mathbb Z}/{p\mathbb Z})
\end{equation*}
sends $\left({}^{\mathfrak S}{\mathcal C}^{\leq -d_{P}}_{P, X, \emptyset}, {}^{\mathfrak S}{\mathcal C}^{\geq -d_{P}}_{P, X, \emptyset}\right)$ to  $\left(D_{\rm c}^{\leq 0}(X_{\rm \acute{e}t}), D_{\rm c}^{\geq 0}(X_{\rm \acute{e}t})\right)$.
\end{theo}

In order to prove Theorem \ref{t8} we need the following lemma.

\begin{lemm}\label{t3}
Let $P$ be a smooth $k$-scheme of dimension $d_{P}$ and $\mathcal M$ a complex in  $D_{\rm lfgu}^{b}({\mathcal O}_{F, P})$.
The following conditions are equivalent.
\begin{enumerate}
\item ${\mathcal M}\in {}^{\mathfrak S}D_{\rm lfgu}^{\geq -d_{P}}({\mathcal O}_{F, P})$.
\item ${\mathcal M}$ is quasi-isomorphic to a bounded complex $\mathcal N$ of flat ${\mathcal O}_{P}$-modules such that ${\mathcal N}^{n}=0$ for any $n<-d_{P}$.
\item ${\rm H}^{k}\left(i_{x}^{!}{\mathcal M}\right)=0$ for any $k< 0$ and any closed point $x$ of $P$,
where $i_{x}$ denotes the canonical closed immersion $\{x\}\hookrightarrow P$.
\end{enumerate}
\end{lemm}
\begin{proof}
The equivalence of $1$ and $2$ follows from \cite[Proposition 4.6]{Kas2}.
Let us prove that the condition $2$ implies the condition $3$.
Suppose that $\mathcal M$ is quasi-isomorphic to a bounded complex $\mathcal N$ of flat ${\mathcal O}_{P}$-modules such that ${\mathcal N}^{n}=0$ for any $n<-d_{P}$.
Let us denote by $\kappa(x)$ the residue field at $x$.
Then, as a complex of $\kappa(x)$-modules, we can calculate as
\begin{eqnarray*}
i_{x}^{!}{\mathcal M}&{\cong}& \kappa(x) \otimes_{i_{x}^{-1}{\mathcal O}_{X}}^{\mathbb L}i_{x}^{-1}{\mathcal M}[-d_{P}]\\
&=& \kappa(x) \otimes_{i_{x}^{-1}{\mathcal O}_{X}}i_{x}^{-1}{\mathcal N}[-d_{P}].
\end{eqnarray*}
We have the condition $3$ from this description.
Next we show the condition $3$ implies the condition $1$.
Suppose that $\mathcal M$ satisfies the condition $3$.
We prove that  ${\mathcal M}$ belongs to ${}^{\mathfrak S}D_{\rm lfgu}^{\geq -d_{P}}({\mathcal O}_{F, P})$ by the induction on the
codimension $d$ of $S:={\rm Supp}(\mathcal M)$.
Let $S_{0}$ be a $d$-codimensional smooth open subscheme of $S$ such that $H:=S\setminus S_{0}$ is of codimension $>d$.
Then we have a closed immersion $i: S_{0}\hookrightarrow U:=P\setminus H$ and the open immersion $j: U\hookrightarrow P$.
Since ${\mathcal M}_{|U}$ is supported on $U$, by \cite[Corollary 5.11.3]{EK}, there exists ${\mathcal N}\in D_{\rm lfgu}^{b}({\mathcal O}_{F, S_{0}})$ such that $i_{+}{\mathcal N}\cong {\mathcal M}_{|U}$.
By shrinking $S_{0}$ if necessary, we may assume that all cohomology sheaves of $\mathcal N$ are unit $F$-crystals.
We fix a closed point $x\in S_{0}$ and denote by $i_{x}$ (resp. $i'_{x}$) the closed immersion $\{x\}\hookrightarrow P$ (resp. $\{x\}\hookrightarrow S_{0}$).
By pulling back the isomorphism $i_{+}{\mathcal N}\cong {\mathcal M}_{|U}$ to $\{x\}$, we have $i'^{!}_{x}{\mathcal N}\cong i^{!}_{x}{\mathcal M}$.
Let us take a flat resolution $\mathcal F\to N$ as ${\mathcal O}_{S_{0}}$-modules.
One has 
${\kappa}(x)\otimes i_{x}'^{-1}{\mathcal F}\cong i^{!}_{x}{\mathcal M}[d_{S_{0}}]$.
By this description combined with the condition $3$,
we know 
$${\mathcal N}\in D_{\rm lfgu}^{\geq -d_{S_{0}}}({\mathcal O}_{F, S_{0}}).$$
By a similar argument in the proof of Theorem \ref{t2}, one has $i_{+}{\mathcal N}\in {}^{\mathfrak S}D_{\rm lfgu}^{\geq -d_{P}}({\mathcal O}_{F, U})$.
Since ${\mathbb R}j_{*}$ is left $t$-exact with respect to the constructible $t$-structure by \cite[Lemma 3.7]{Kas2}, we have
${\mathbb R}j_{*}i_{+}{\mathcal N}\in {}^{\mathfrak S}D_{\rm lfgu}^{\geq -d_{P}}({\mathcal O}_{F, P})$.
Let us consider a distinguished triangle
\begin{equation*}
i^{!}_{x}{\mathbb R}{\Gamma}_{H}{\mathcal M}\to i^{!}_{x}{\mathcal M}\to i^{!}_{x}j_{+}j^{-1}{\mathcal M}\xrightarrow{+}.
\end{equation*}
By taking the long exact sequence, we see ${\rm H}^{k}\left(i^{!}_{x}{\mathbb R}{\Gamma}_{H}{\mathcal M} \right)=0$ for any $k<0$.
Hence the induction hypothesis implies ${\mathbb R}{\Gamma}_{H}{\mathcal M}\in {}^{\mathfrak S}D_{\rm lfgu}^{\geq -d_{P}}({\mathcal O}_{F, P})$
and we obtain ${\mathcal M}\in {}^{\mathfrak S}D_{\rm lfgu}^{\geq -d_{P}}({\mathcal O}_{F, P})$.

\end{proof}
Now we may start to prove Theorem \ref{t8}.

\begin{proof}
First of all, we shall prove that the equivalence
\begin{equation*}
{\rm Sol}_{P}: D_{\rm lfgu}^{b}({\mathcal O}_{F, P})\xrightarrow{\cong} D_{\rm c}^{b}(P_{\rm \acute{e}t}, {\mathbb Z}/{p\mathbb Z})
\end{equation*}
sends $\left({}^{\mathfrak S}D_{\rm lfgu}^{\leq -d_{P}}({\mathcal O}_{F, P}), {}^{\mathfrak S}D_{\rm lfgu}^{\geq -d_{P}}({\mathcal O}_{F, P})\right)$ to $\left(D_{\rm c}^{\leq 0}(P_{\rm \acute{e}t}), D_{\rm c}^{\geq 0}(P_{\rm \acute{e}t})\right)$.
Since ${\rm Sol}_{P}$ is an equivalence of triangulated categories, it suffices to show that 
${}^{\mathfrak S}D_{\rm lfgu}^{\geq -d_{P}}({\mathcal O}_{F, P})$ corresponds to $D_{\rm c}^{\leq 0}(P_{\rm \acute{e}t}, {\mathbb Z}/{p\mathbb Z})$ via ${\rm Sol}_{P}$ (cf. \cite[Corollary 10.1.18]{KS}).
Let us first suppose that ${\mathcal M}$ is an object in ${}^{\mathfrak S}D_{\rm lfgu}^{\geq -d_{P}}({\mathcal O}_{F, P})$.
Let $x$ be a point in $P$.
Denote by $\overline{\{x\}}$ the closure of $\{x\}$ in $P$.
For an open subset $U$ of $\overline{\{x\}}$, we denote by $i_{U}$ the canonical immersion $U\hookrightarrow P$.
By Lemma \ref{t3}, there exists an ${\mathcal O}_{P}$-flat resolution $\mathcal N$ of ${\mathcal M}$ such that ${\mathcal N}^{n}=0$ for any $n<-d_{P}$.
We then calculate
\begin{eqnarray*}
i_{U}^{!}{\mathcal M}&{\cong}& {\mathcal O}_{U} \otimes_{i_{U}^{-1}{\mathcal O}_{P}}^{\mathbb L}i_{U}^{-1}{\mathcal M}[d_{U/P}]\\
&=&{\mathcal O}_{U}  \otimes_{i_{U}^{-1}{\mathcal O}_{P}}i_{U}^{-1}{\mathcal N}[d_{U/P}].
\end{eqnarray*}
By this description, we have ${\rm H}^{k}\left({i_{U}^{!}{\mathcal M}}\right)=0$ for any $k<-d_{U}$.
By shrinking $U$ if necessary, we may assume that all cohomology sheaves of $i_{U}^{!}{\mathcal M}$ are unit $F$-crystals.
Then, by \cite[Proposition 9.3.2]{EK}, $i_{U}^{!}{\mathcal M}$ is $\underline{\rm Hom}_{{\mathcal O}_{F, U_{{\rm \acute{e}t}}}}(\pi_{U}^{*}(-), {\mathcal O}_{U_{\rm \acute{e}t}})$-acyclic.
Hence we can calculate as
\begin{eqnarray*}
i_{U}^{-1}{\rm Sol}_{P}({\mathcal M})&\cong& {\rm Sol}_{U}\left(i_{U}^{!}{\mathcal M}\right)\\
&=& \underline{\rm Hom}_{{\mathcal O}_{F, U_{{\rm \acute{e}t}}}}(\pi_{U}^{*}(i_{U}^{!}{\mathcal M} ), {\mathcal O}_{U_{\rm \acute{e}t}})[d_{U}].
\end{eqnarray*}
By this description, for any $n>0$ the equality ${\rm H}^{n}\left(i_{U}^{-1}{\rm Sol}_{P}({\mathcal M})\right)=0$ holds. 
So we have ${\rm Sol}_{P}({\mathcal M})\in D_{\rm c}^{\leq 0}(P_{\rm \acute{e}t}, {\mathbb Z}/{p\mathbb Z})$.
Conversely, suppose that we are given an object ${\mathcal F}$ in $D_{\rm c}^{\leq 0}(P_{\rm \acute{e}t}, {\mathbb Z}/{p\mathbb Z})$.
By \cite[p.94, Lemma 4.7]{De}, we may assume that ${\mathcal F}$ is a bounded complex of constructible ${\mathbb Z}/{p\mathbb Z}$-modules.
For any closed point $x$, we can calculate as
\begin{eqnarray*}
i_{x}^{!}{\rm M}_{P}({\mathcal F})&\cong& {\rm M}_{\{x\}}\left(i_{x}^{-1}{\mathcal F}\right)\\
&=&{\mathbb R}\underline{\rm Hom}_{{\mathbb Z}/{p\mathbb Z}}(i_{x}^{-1}{\mathcal F}, \kappa(x))\\
&=&\underline{\rm Hom}_{{\mathbb Z}/{p\mathbb Z}}(i_{x}^{-1}{\mathcal F}, \kappa(x)).
\end{eqnarray*} 
By this description, we see the condition $3$ in Lemma \ref{t3} for ${\rm M}_{P}({\mathcal F})$ and thus
${\rm M}_{P}({\mathcal F})\in {}^{\mathfrak S}D_{\rm lfgu}^{\geq -d_{P}}({\mathcal O}_{F, P})$.

Now let $X$ be a $k$-embeddable $k$-scheme with a closed immersion $i: X\hookrightarrow P$.
Let $D^{b}_{{\rm c}, X}(P_{\rm \acute{e}t}, {\mathbb Z}/{p\mathbb Z})$ denote the full triangulated subcategory of 
 $D_{\rm c}^{b}(P_{\rm \acute{e}t}, {\mathbb Z}/{p\mathbb Z})$ consisting of complexes supported on $X$.
By Lemma \ref{l10}, ${\rm Sol}_{P}$ restricts to an equivalence of triangulated categories
\begin{equation*}
{\rm Sol}_{P}: {\mathcal C}_{P, X, \emptyset}\xrightarrow{\cong} D^{b}_{{\rm c}, X}(P_{\rm \acute{e}t}, {\mathbb Z}/{p\mathbb Z}).
\end{equation*}
Then ${}^{\mathfrak S}{\mathcal C}^{\geq -d_{P}}_{P, X, \emptyset}$ corresponds to $D_{{\rm c}, X}^{\leq 0}(P_{\rm \acute{e}t}):=D_{\rm c}^{\leq 0}(P_{\rm \acute{e}t})\cap D^{b}_{{\rm c}, X}(P_{\rm \acute{e}t}, {\mathbb Z}/{p\mathbb Z})$ via ${\rm Sol}_{P}$. 
Moreover, since the equivalence
\begin{equation*}
i^{-1}: D^{b}_{{\rm c}, X}(P_{\rm \acute{e}t}, {\mathbb Z}/{p\mathbb Z}) \xrightarrow{\cong}D^{b}_{{\rm c}}(X_{\rm \acute{e}t}, {\mathbb Z}/{p\mathbb Z})
\end{equation*}
is $t$-exact with respect to the standard $t$-structure, we see that $D_{{\rm c}, X}^{\leq 0}(P_{\rm \acute{e}t})$ corresponds to $D_{{\rm c}}^{\leq 0}(X_{\rm \acute{e}t})$ via $i^{-1}$.
As a consequence, we know that ${\mathcal M}\in {}^{\mathfrak S}{\mathcal C}^{\geq -d_{P}}_{P, X, \emptyset}$ if and only if ${\rm Sol}_{X}({\mathcal M})\in D_{{\rm c}}^{\leq 0}(X_{\rm \acute{e}t})$.

\end{proof}
\begin{coro}\label{c5}
Let $X$ be a $k$-embeddable $k$-scheme with an immersion from $X$ into a proper smooth $k$-scheme $P$.
We set
\begin{eqnarray*}
{}^{\mathfrak S}D_{\rm lfgu}^{\leq 0}(X/k):= {}^{\mathfrak S}{\mathcal C}^{\leq -d_{P}}_{P, X}\textit{ and }
{}^{\mathfrak S}D_{\rm lfgu}^{\geq 0}(X/k):= {}^{\mathfrak S}{\mathcal C}^{\geq -d_{P}}_{P, X}.
\end{eqnarray*} 
Then the $t$-structure $\left({}^{\mathfrak S}D_{\rm lfgu}^{\leq 0}(X/k), {}^{\mathfrak S}D_{\rm lfgu}^{\geq 0}(X/k)\right)$ is independent of the choice of $X\hookrightarrow P$, which we call the constructible $t$-structure.
Moreover, the constructible $t$-structure corresponds to the standard $t$-structure on $D^{b}_{{\rm c}}(X_{\rm \acute{e}t}, {\mathbb Z}/{p\mathbb Z})$ via ${\rm Sol}_{X}$.
\end{coro}

\end{document}